\documentclass[english]{amsart}
\usepackage{amssymb,amsmath}
\usepackage{caption}
\usepackage{float}
\usepackage{graphicx,subfigure,epsfig}
\usepackage{epstopdf}
\usepackage{mathtools}
\usepackage{multicol}
\usepackage{fancybox}
\usepackage[usenames, dvipsnames]{color}
\newtheorem{theorem}{Theorem}[section]
\newtheorem{definition}{Definition}[section]
\newtheorem{lemma}{Lemma}[section]
\newtheorem{remark}{Remark}[section]

\newtheorem{example}{Example}[section]
\newtheorem{proposition}{Proposition}[section]

\newcommand{\restr}[1]{|_{#1}}

\begin{document}
\markboth{Kirandeep Kaur and Sandeep Sharama}{Extending  quasi-alternating links II}

\title{Extending Quasi-alternating Links II }

\author{Kirandeep Kaur}

\address{Department of Computational, Statistics and Data Analytics\\Guru Nanak Dev University, Punjab 143005, India}
\email{kirandeepoffical@gmail.com}

\author{Sandeep Sharama}
\address{Department of Computational, Statistics and Data Analytics\\Guru Nanak Dev University, Punjab 143005, India }
\email{sandeep.cse@gndu.ac.in}

\begin{abstract}
Champanerkar and Kofman \cite{champanerkar2009twisting} introduced  an innovative method for constructing quasi-alternating links by substituting a quasi-alternating crossing $c$ in a quasi-alternating link with a rational tangle of the same type.
This construction was later extended by Chbili and Kaur \cite{chbili2020quasiI} for any alternating tangle of same type.
 However, this approach does not generally apply to non-alternating tangles or alternating tangles of opposite types, although there are examples where it holds for alternating tangles of opposite type and certain non-alternating tangles.  This raises a natural question to find the conditions and specific tangles for which this approach is applicable. In this paper, we explore this question by  identifying several conditions under which this construction of new quasi-alternating links can still be valid for alternating tangles of opposite types. Additionally, we establish specific families of non-alternating tangles for which this construction is effective in creating new quasi-alternating links. 
\end{abstract}

\makeatletter{\renewcommand*{\@makefnmark}{}
\makeatother}
\maketitle
\section{Introduction}

\noindent Let $L$ be a non-split  alternating link in the three-sphere $S^3$ and $\Sigma(L)$ be  the branched double-cover of  $S^3$  along $L$. Ozsv\'{a}th and Szab\'{o} \cite{ozsvath2005heegaard} proved that the Heegaard Floer homology of $\Sigma(L)$   is determined  entirely  by   the determinant of the link, $\det(L)$.
Additionally, they observed that this notable homological property extends to a broader class of links, which  they called quasi-alternating. These links are defined recursively as follows.

\begin{definition} The set $\mathcal{Q}$ of quasi-alternating links is the smallest set satisfying the following properties:
\begin{itemize}
\item  The unknot is in $\mathcal{Q}$.
\item If a link $L$ has a diagram  containing a crossing $c$ such that
\begin{enumerate}
\item both smoothings of the diagram of $L$ at the crossing $c$, $L_0$ and $L_\infty$, as in Fig.~\ref{1},
belong to $\mathcal{Q}$,
\item $\det(L_0)$, $\det(L_\infty) \geq 1$ and  $\det(L) = \det(L_0) + \det(L_\infty)$;
 \end{enumerate}
then $L$ is in $\mathcal{Q}$.  The crossing $c$ is called a quasi-alternating crossing and $L$ is said to be quasi-alternating at $c$.
\end{itemize}
\end{definition}
 \begin{figure}[!ht] {\centering
 \subfigure[$L$]{\includegraphics[scale=0.8]{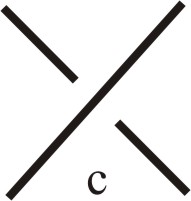} } \hspace{1cm}
\subfigure[$L_0$] {\includegraphics[scale=0.8]{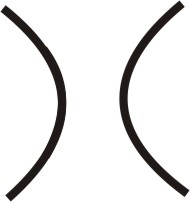}}\hspace{1cm}
\subfigure[$L_\infty$] {\includegraphics[scale=0.8]{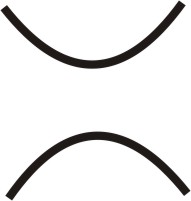}}
\caption{A link $L$ and its smoothings at crossing $c$, $L_0$ and $L_\infty$.}\label{1}}\end{figure}
\noindent With an elementary induction on the determinant of the link, this definition can be used to prove that non-split alternating links belong to $\mathcal{Q}$  \cite{ozsvath2005heegaard}. The knots $8_{20}$ and $8_{21}$ are the first examples of non-alternating quasi-alternating knots in the knot table.  Generally, determining whether a given link is quasi-alternating  through  this recursive definition is quite challenging.
Throughout  the past years, several obstruction criteria for identifying quasi-alternating links  have  been introduced. For instance, quasi-alternating links are proved to be  homologically-thin in  both Khovanov homology and link  Floer homology   \cite{manolescu2007khovanov}.  In addition, the reduced odd Khovanov homology group of any quasi-alternating link is thin and torsion-free \cite{ozsvath2013odd}. On the other hand, the branched double-cover of any quasi-alternating link is an $L$-space \cite{ozsvath2005heegaard}, and the signature of an oriented quasi-alternating  link $L$ equals minus four times the Heegaard Floer correction term of $(\sum(L), t_0)$ \cite{lisca2015signatures}.
 The behavior of polynomial invariants of quasi-alternating links has been also  investigated  leading to obstruction criteria  involving the Q-polynomial and the two-variable Kauffman polynomial  \cite{qazaqzeh2015new,teragaito2014quasi,teragaito2015quasi}.\\
In \cite{champanerkar2009twisting}, Champanerkar and  Kofman  presented an intriguing approach to constructing  new quasi-alternating links from existing ones. They  demonstrated that a  new quasi-alternating link can be obtained by
 replacing a quasi-alternating crossing by an alternating rational tangle of same type. This approach has been extended  to algebraic tangles in \cite{qazaqzehchbiliQublan20014} and any alternating tangles of same type in \cite{chbili2020quasiI}. 
 It has also been used to classify quasi-alternating Montesinos links \cite{qazaqzehchbiliQublan20014}.\\
   However, this construction is generally not applicable to alternating tangles of opposite type or non-alternating tangles. 
  Examples shown in Fig.~\ref{exm:In1}(b) and \ref{exm:In2}(b) illustrate that replacing a quasi-alternating crossing in quasi-alternating links depicted in Fig.~\ref{exm:In1}(a) and \ref{exm:In2}(a) with a non-alternating tangle and an alternating tangle of opposite type, respectively, results in links that are not quasi-alternating. Conversely, examples illustrated in Fig.~\ref{exm:8_1 & 11n}(a) and \ref{Exam:2Last}(b) indicate that the construction can still be valid in specific cases. This raises a natural question: \\
  
\noindent $''$ What are the necessary conditions and types of tangles for which this construction is applicable?$"$

\begin{figure}[!ht] {\centering
 \subfigure[ $(2,8)$-torus link] {\includegraphics[scale=.35]{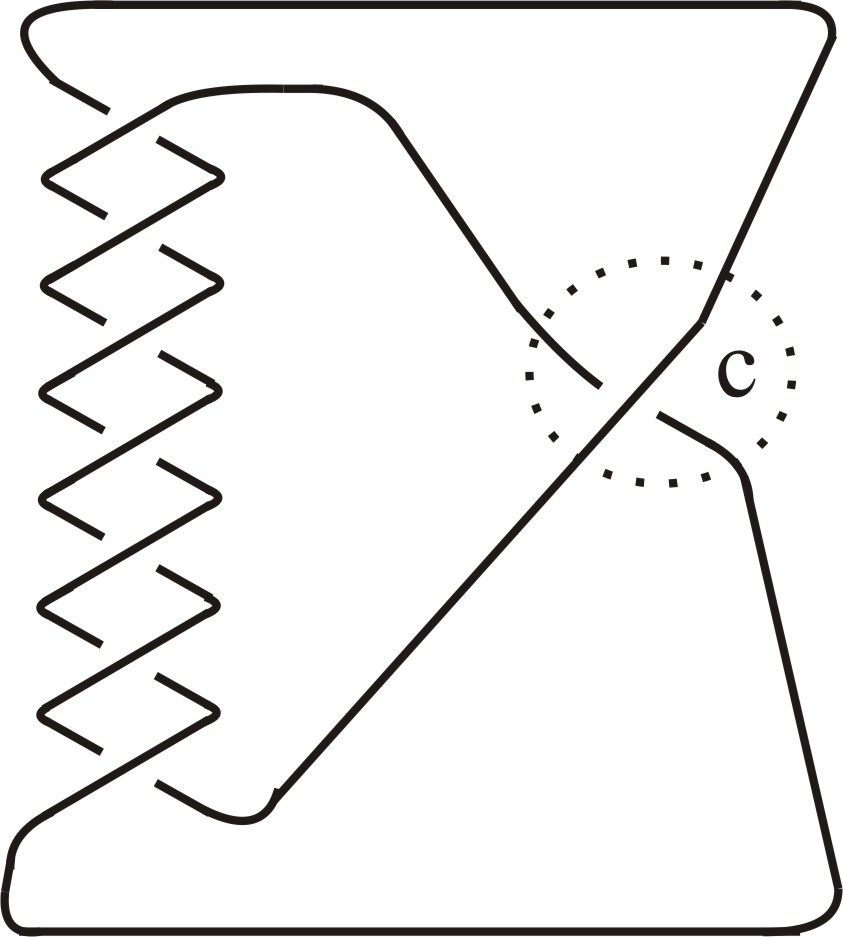} } \hspace{1.5cm}
\subfigure[Non quasi-alternating  link] {\includegraphics[scale=.35]{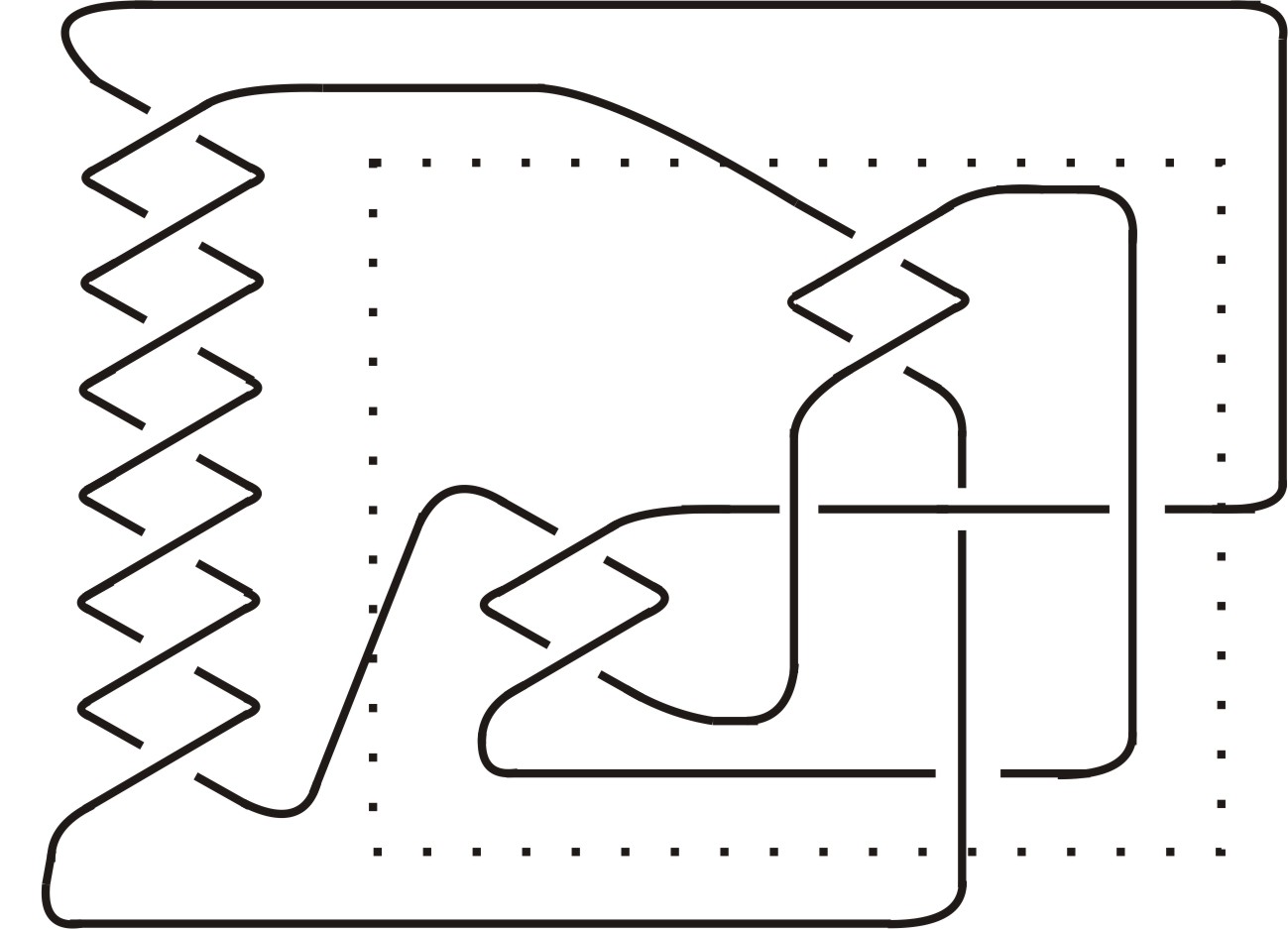}\hspace{.2cm}} \label{exm:In1c}
\caption{ A non quasi-alternating link obtained from $(2,8)$-torus link by extending quasi-alternating crossing $c$ to a non-alternating.}\label{exm:In1}}
\end{figure}
\begin{figure}[!ht] {\centering
 \subfigure[$8_{21}$] {\includegraphics[scale=.55]{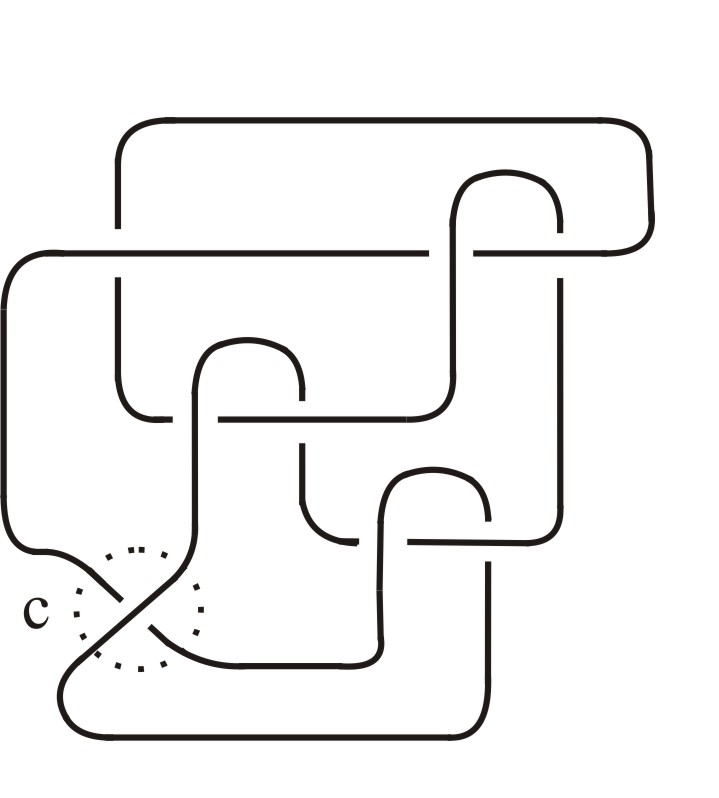} } \hspace{1cm}
\subfigure[$9_{42}$] {\includegraphics[scale=.55]{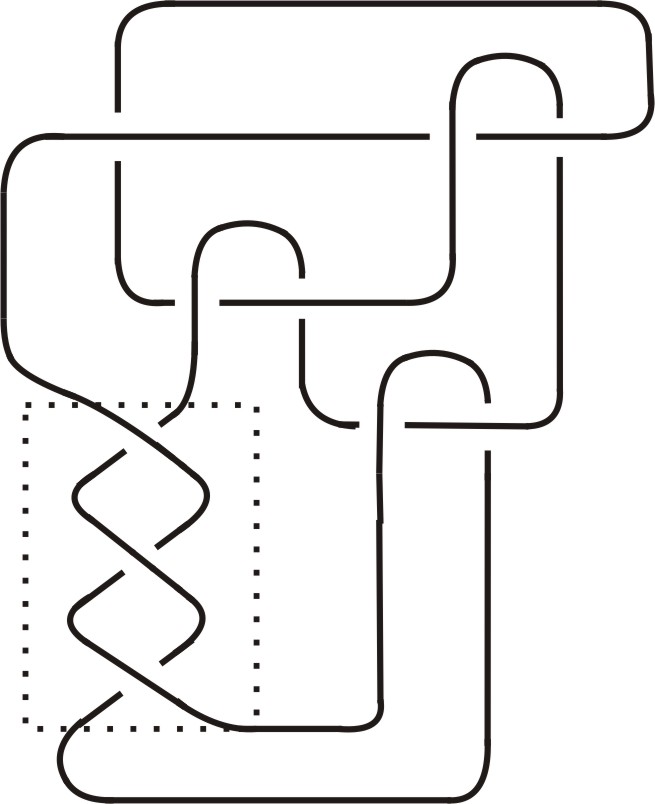}} 
\caption{A non quasi-alternating knot $9_{42}$ obtained from $8_{21}$ knot by replacing a quasi-alternating crossing $c$ to an alternating tangle of opposite type.}\label{exm:In2}}
\end{figure}
\noindent In this paper, we explore this question by identifying specific conditions and establishing certain classes of alternating tangles of opposite type for which the construction is valid. 
We also examine several families of non-alternating tangles where this construction remains applicable. Through examples, we support our findings and and illustrate cases where the specified conditions do not apply, leading to links that fail to be quasi-alternating. Furthermore, we demonstrate the existence of link diagrams that satisfy the specified conditions.

\noindent This construction proves to be a valuable tool for classifying quasi-alternating links. One significant application of our approach is the ability to construct non-alternating quasi-alternating links from existing alternating links, an example is illustrated in Fig.~\ref{Exam:2Last}.  While,  such a construction  is not possible using the earlier construction rules presented in \cite{champanerkar2009twisting,qazaqzehchbiliQublan20014,chbili2020quasiI}. Therefore, this work extends the previous studies by Champanerkar and Kofman \cite{champanerkar2009twisting} and their generalizations in \cite{qazaqzehchbiliQublan20014,chbili2020quasiI}.

\noindent Diagrams play an significant role in constructing new quasi-alternating links from existing ones using this technique.  For instance, the $11n_{90}$ knot, shown in Fig.~\ref{exm:8_1 & 11n}(b), is a quasi-alternating knot obtained by modifying $L$, as shown in Fig.~\ref{exm:43}. This modification involves replacement of the quasi-alternating crossing $c$  with three vertical half twists of opposite type. However, the minimal crossing number of the knot $L$ shown in Fig.~\ref{exm:43} is no greater seven. This means that it is impossible to construct a 11 crossing knot or link  from a quasi-alternating knot having at most seven crossings by replacing any quasi-alternating  crossing with three half twists or any tangle with three crossings.

\noindent Furthermore, the knot diagram $L$ shown in Fig.~\ref{exm:43} is obtained from the figure eight knot by replacing crossing $c$ with a non-alternating tangle diagram, see Fig.~\ref{exm:43org}. However, the replaced tangle posses an alternating tangle diagram with four crossings. Thus, the knot diagram $L$ shown in Fig.~\ref{exm:43} cannot be obtained  from the figure eight knot by replacing a quasi-alternating crossing with a tangle of four crossings.  This illustrates the significant role that tangle diagrams play in the construction of quasi-alternating links.

\noindent Therefore, in this paper,  we treat non-alternating tangles as non-alternating tangle diagrams. Additionally, the quasi-alternating crossing $c$ considered in this paper is of type shown in Fig.~\ref{1}(a). 
 If the crossing does not initially have this form, it can be transformed into this type by applying a simple isotopy in space. 
 
\begin{figure}[!ht] {\centering
 \includegraphics[scale=.55]{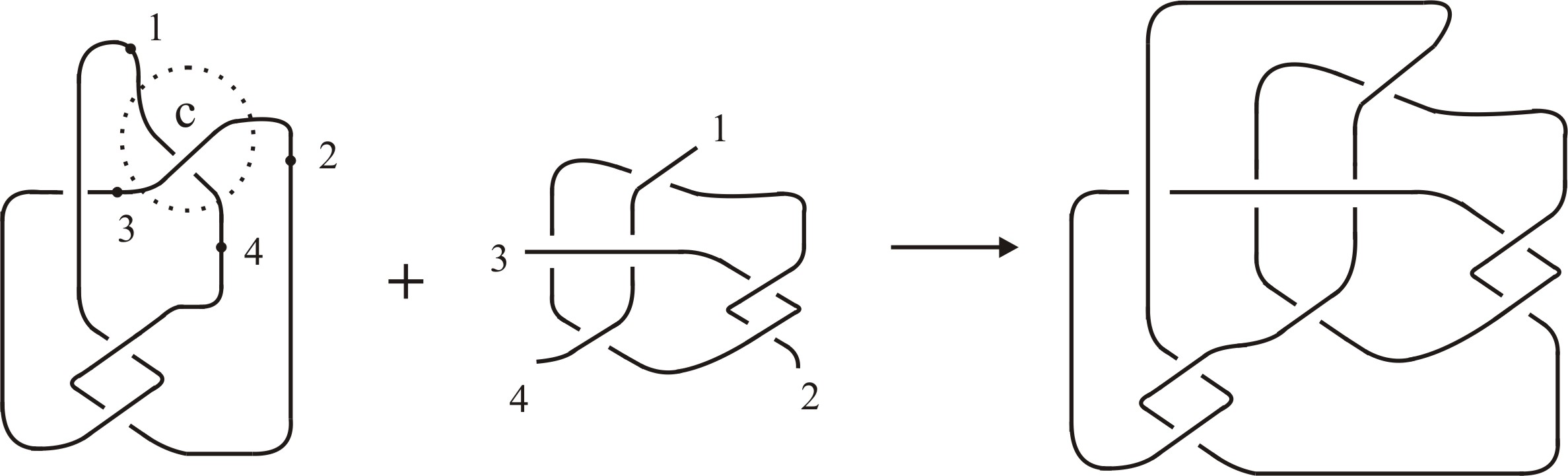}  
\caption{A knot obtained from figure eight knot by replacing a non-alternating tangle diagram.}\label{exm:43org}}
\vspace{-.1cm}
\end{figure} 

\noindent This paper is organized as follows: Section 2 contains the preliminaries that
are required to prove the main results of the paper.  In Section 3, we present the main results of this paper, and provide examples supporting our results and counter examples.  We conclude this paper in Section 4 by discussing the work accomplished, the difficulties encountered, and potential future extensions of this research.  

\section{Background and notations}
\noindent  This section will begin with a review of key definitions and results related to quasi-alternating links, which will be essential for the subsequent discussions in the paper.

\noindent Two link diagrams are said to be equivalent if one can be deformed into another through a finite sequence of classical Reidemeister moves RI, RII, RIII, as shown in Fig.~\ref{fig-rei}.

\noindent A detour move is a consequence of  Reidemeister moves which allows to move a segment of a link diagram containing only either over crossings or under crossings, freely in the plane.  When relocating this segment to a different position in the plane, all new crossings that occur should be marked as over or under crossings based on the original crossings of the segment, wherever it cuts across the link diagram, as shown in see Fig.~\ref{detmove}.

\begin{figure}[!ht] 
\centering
\includegraphics[scale=0.4]{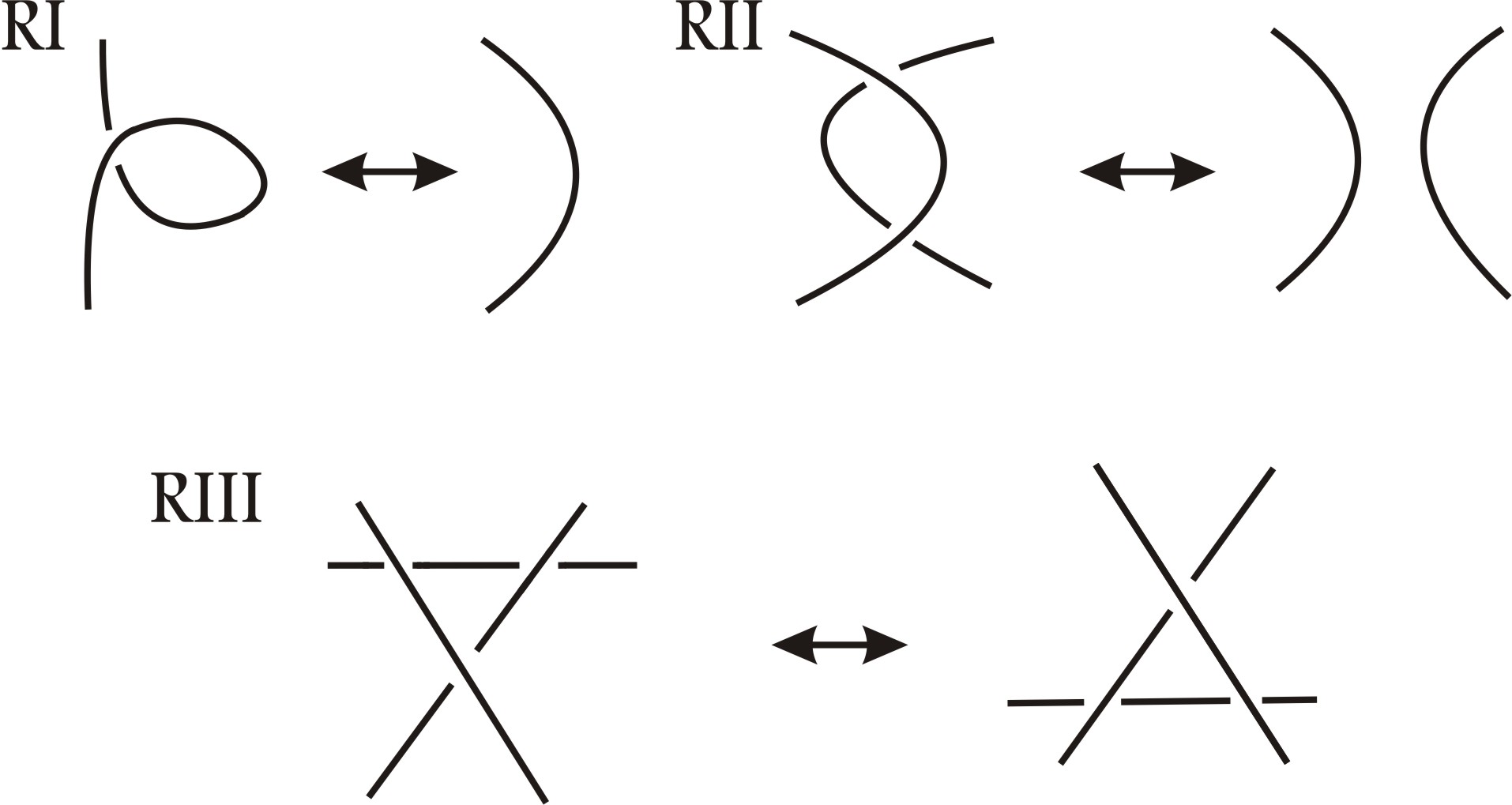}
\caption{Reidemeister moves.} \label{fig-rei}
\end{figure} 

 \begin{figure}[!ht] {\centering
 \subfigure[]{\includegraphics[scale=0.5]{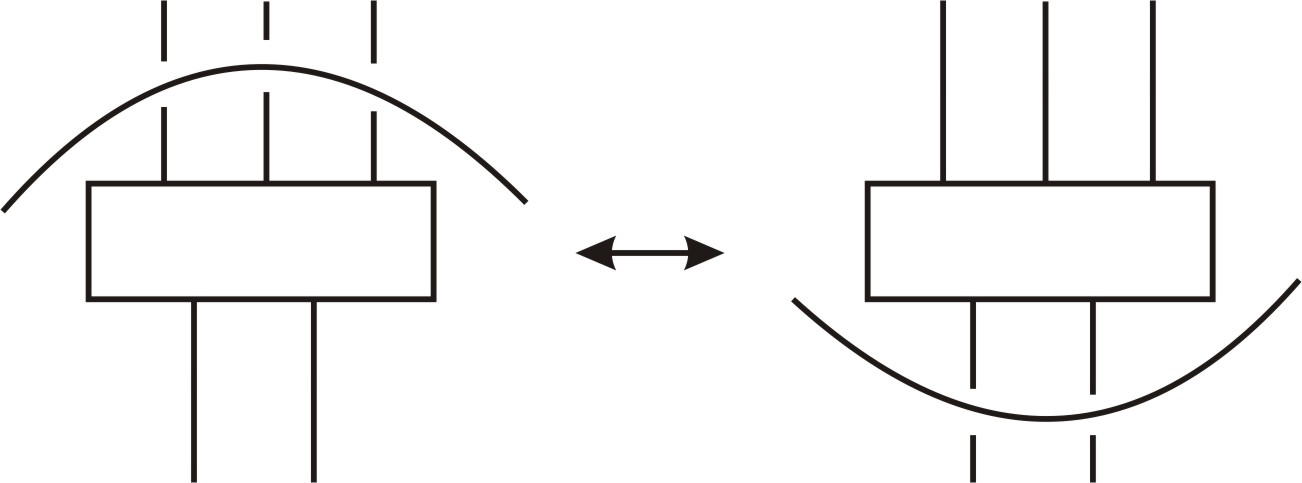} } \hspace{1cm}
\subfigure[] {\includegraphics[scale=0.5]{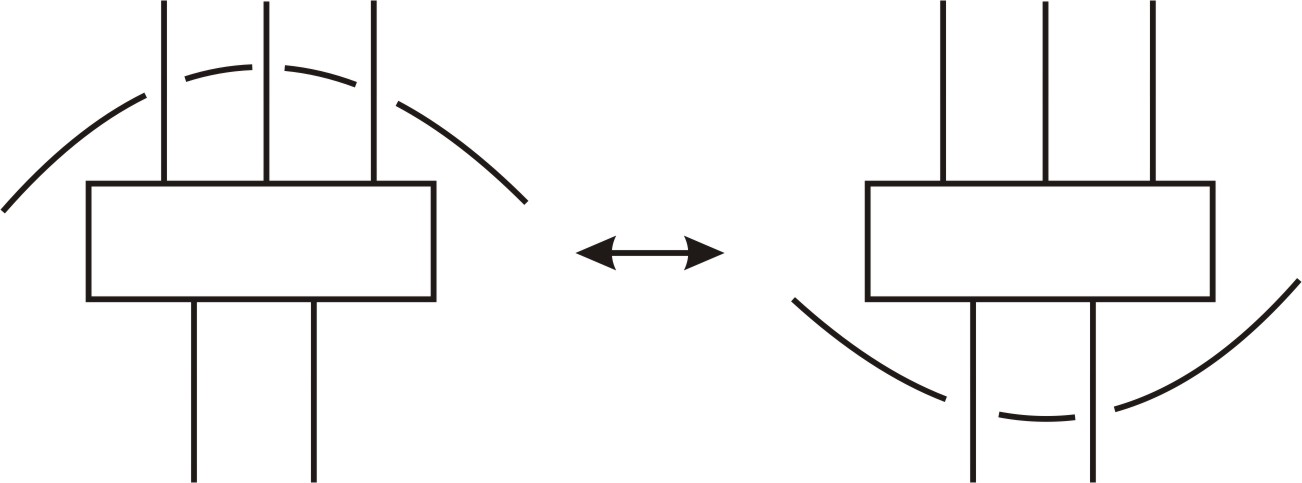}}
\caption{ Detour moves.}\label{detmove}}\end{figure}

\noindent A tangle is a proper  embedding of disjoint union of 2 arcs and possibly  loops into a three-ball $B^3$, with the endpoints of these arcs located at four fixed points on the boundary of $B^3$.
Tangles are represented by their planar diagrams and are considered up to isotopy of $B^3$ keeping the boundary fixed. We say that a tangle is connected if it is represented by a connected diagram.
Throughout this paper, unless otherwise stated, we are only considering connected tangles.
For a given tangle $T$, one can connect its endpoints with simple arcs to define the numerator closure $N(T)$ and the denominator  closure $D(T)$, as illustrated in Fig.~\ref{nt2}(a) and (b).
 
 \noindent A crossing $c$ in a link diagram $L$ is considered to be nugatory if it appears as shown Fig.~\ref{nt2}(c).
We refer to a crossing $c$ of a tangle $T$ as nugatory in $T$ if it is nugatory in both closures of $T$.  An example of a nugatory crossing $c$ is illustrated in Fig.~\ref{nt2}(d). 
\begin{figure}[!ht] {\centering
\subfigure[$N(T)$] {\hspace{.2cm}\includegraphics[scale=0.45]{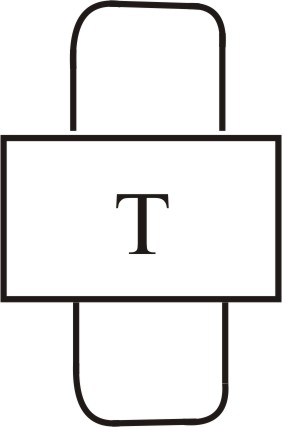}\hspace{.3cm}} \hspace{.3cm}
\subfigure[$D(T)$] {\includegraphics[scale=0.45]{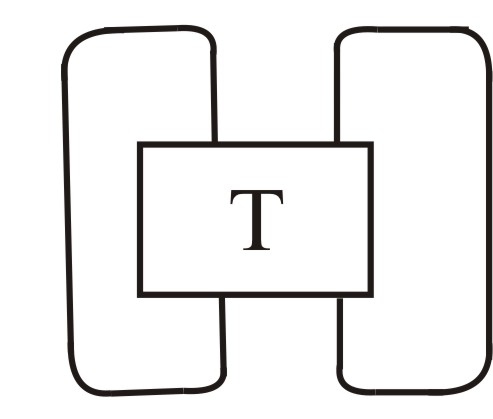}}\hspace{.5cm}
\subfigure[] {\includegraphics[scale=0.45]{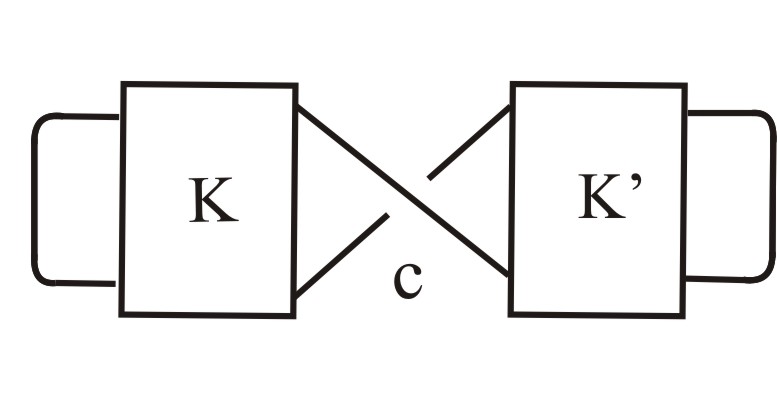}} \hspace{.5cm}
\subfigure[] {\includegraphics[scale=0.45]{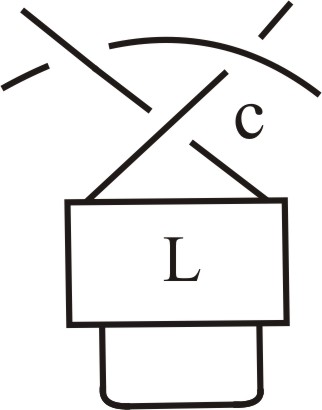}}
\caption{Numerator closure, Denominator closure, and a nugatory crossing in a link diagram and in a tangle. }\label{nt2}}\end{figure}

\noindent A tangle diagram $T$ is considered alternating if overpasses and underpasses appear alternately  as one travels along any strand or any loop. A tangle is considered  alternating if it admits an alternating tangle diagram. An alternating tangle diagram is said to be  reduced  if it contains no nugatory crossings. In this paper, we are considering alternating tangles as reduced alternating tangle diagrams. 
For a tangle $T$, we denote the tangles $T_+,~T_{-},~T^{+}$, $T^{-}$, $\chi_{T}$ and $\chi_{\overline{T}}$ as illustrated in Fig.~\ref{T+} and ~\ref{chiT}.
 \begin{figure}[!ht] {\centering
 \subfigure[$T_{+}$]{\hspace{.5cm}\includegraphics[scale=0.5]{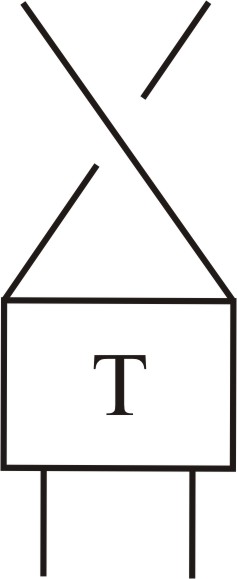} } \hspace{1cm}
\subfigure[$T_{-}$] {\includegraphics[scale=0.5]{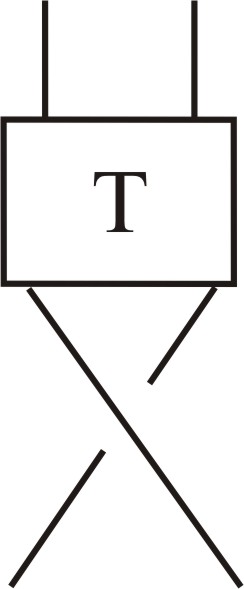}}\hspace{1cm}
\subfigure[$T^{+}$] {\includegraphics[scale=0.5]{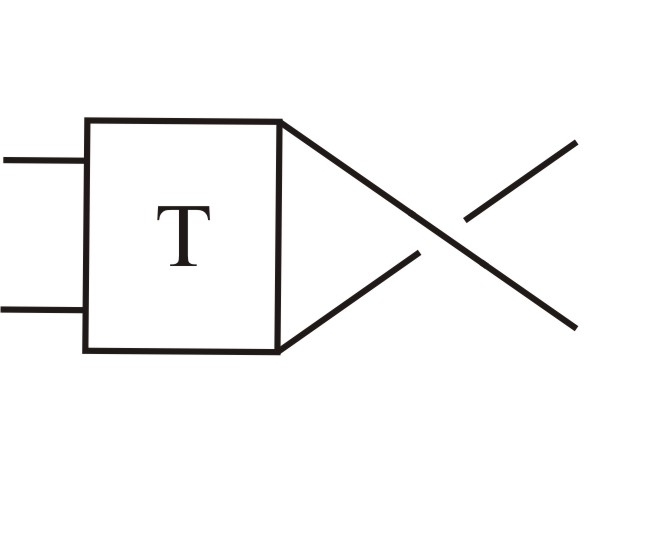}}\hspace{.5cm}
\subfigure[$T^{-}$] {\includegraphics[scale=0.5]{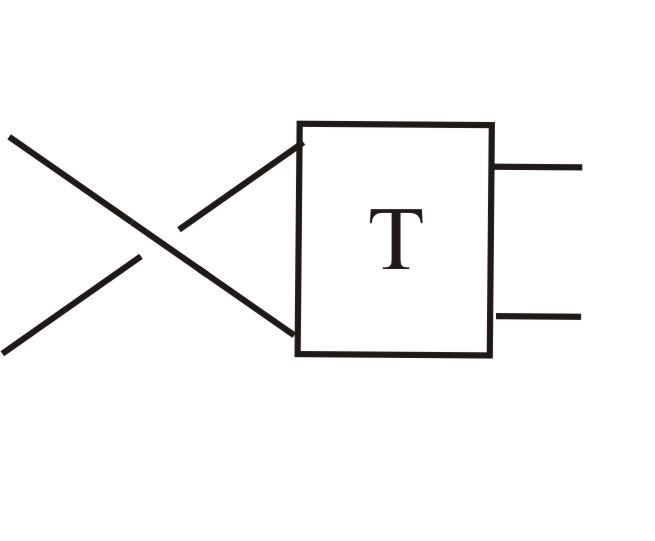}}
\caption{ Tangles $T_+,~T_{-},~T^{+}$ and $T^{-}$.}\label{T+}}\end{figure}
 
 \begin{figure}[!ht] {\centering
  \subfigure[$\chi_{T}$]{\includegraphics[scale=0.5]{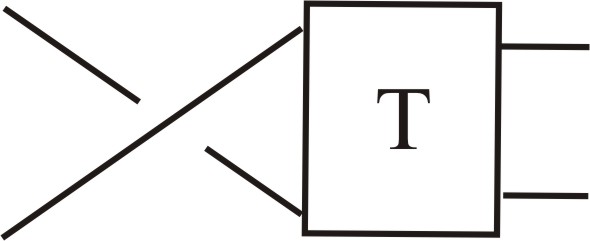}} \hspace{.5cm}
\subfigure[$\chi_{\overline{T}}$]{\includegraphics[scale=0.5]{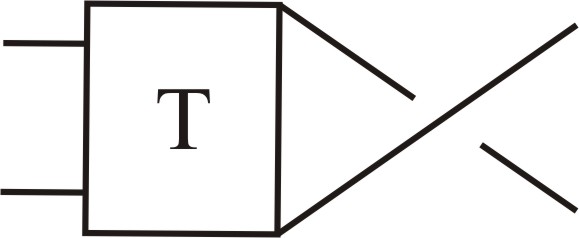}}
\caption{Tangles $\chi_{T}$ and $\chi_{\overline{T}}$.}\label{chiT}}
\end{figure}

\noindent A spanning tree of a connected planar graph $\mathcal{G}$, is a connected, acyclic sub-graph that includes all  vertices of $\mathcal{G}$. 
The determinant of a link $L$, denoted by $\det(L)$, is a well-known numerical invariant of links, which can be determined  by counting   the number of spanning trees in a checkerboard graph of a  link projection. Let us now  review this computation  method.
\begin{figure}[!ht] {\centering\includegraphics[scale=.8]{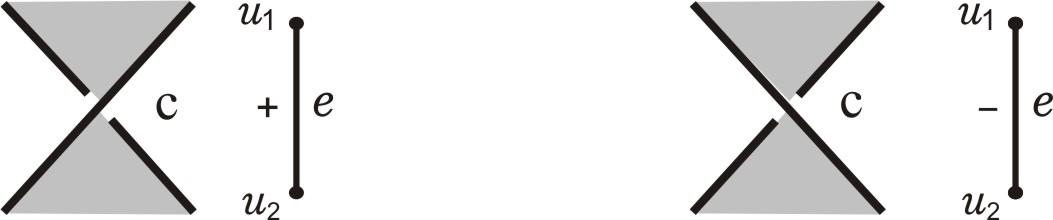}
\[\hspace{2cm}\text{(a)}\hspace{4cm}\text{(b)}\hspace{3cm}\]
\caption{}\label{2}}\end{figure}

\noindent A Tait graph, $\mathcal{G}(L)$, of a connected link diagram $L$ is a  graph which is constructed by first coloring the diagram in a checkerboard fashion. Then, a vertex is assigned to each shaded region, and an edge is assigned to each crossing, with a sign indicated as illustrated in Fig.~\ref{2}. The determinant of the  link $L$ is calculated  as follows:

\begin{lemma}\cite{champanerkar2009twisting} For any spanning tree $\mathcal{T}$ of $\mathcal{G}(L)$, let $\mathsf{v}(\mathcal{T})$ be the number of positive edges in $\mathcal{T}$. Let $s_{v}(L)=\sharp\{\text{spanning trees $\mathcal{T}$ of~} \mathcal{G}(L)\mid \mathsf{v}(\mathcal{T})=v \}$. Then
\[\det(L)=\Big|\displaystyle \sum_v (-1)^vs_v(L)\Big|.\]
\end{lemma}

\begin{remark}[Lemma 2.3, \cite{chbili2020quasiI}] \label{rem:ab} Let $c$ be a quasi-alternating crossing of a link diagram $L$ and $e$ be the corresponding edge in $\mathcal{G}(L)$. If $e$ is a positive edge in $\mathcal{G}(L)$, then
\[\Big(\displaystyle \sum_v (-1)^vs_{v-1}(L_0)\Big).\Big(\displaystyle \sum_v (-1)^vs_v(L_\infty)\Big)>0.\]
\noindent For our convenience, in this paper we use the following notation
\[\mathfrak{a}=\displaystyle \sum_v (-1)^vs_{v-1}(L_0),\quad \text{and},\quad \mathfrak{b}=\displaystyle \sum_v (-1)^vs_v(L_\infty).\]
It is evident that $|\mathfrak{a}|=\det(L_0)$ and $|\mathfrak{b}|=\det(L_{\infty})$.
\end{remark}

\begin{definition}\cite{chbili2020quasiI} A sub-graph $\mathcal{H}$ of a connected planar graph $\mathcal{G}$  is called an  almost spanning tree with respect to vertices $u_1$ and $u_2$, if there is no path between $u_1$ and $u_2$, and by adding one edge it becomes a spanning tree of $\mathcal{G}$.
\end{definition}

\begin{definition}  Given a tangle $T$, we define the Tait graph of $T$ as  the Tait graph of its  numerator closure $N(T)$. This graph is denoted by $\mathcal{G}(T)$.
\end{definition}
\begin{definition}\cite{chbili2020quasiI} An alternating tangle $T$ is said to be positive (respectively,  negative) if the edges of the  Tait graph $\mathcal{G}(T)$, with induced checkerboard coloring as shown in Fig.~\ref{2}, are positive (respectively,  negative).
\end{definition}

\noindent Let $e$ be the edge corresponding to a crossing $c$ with checkerboard coloring as illustrated in Fig.~\ref{2}.  Then an alternating tangle $T$  is a tangle of same type as $c$ or extends $c$, if $T$ is positive whenever $e$ is positive and negative whenever  $e$ is negative. If $T$ fails to adhere to this rule,  we refer to 
$T$ as being of opposite type with respect to $c$. The tangle shown in Fig.~\ref{positivetangle} is a positive tangle and of same type with respect to crossing $c$ depicted  in Fig.~\ref{positivetangle}. However, the tangle shown in Fig.~\ref{positivetangle} is of opposite type with respect to crossing $c$ as shown in Fig.~\ref{2}(b).
 
 \begin{figure}[!ht] {\centering
\includegraphics[scale=.7]{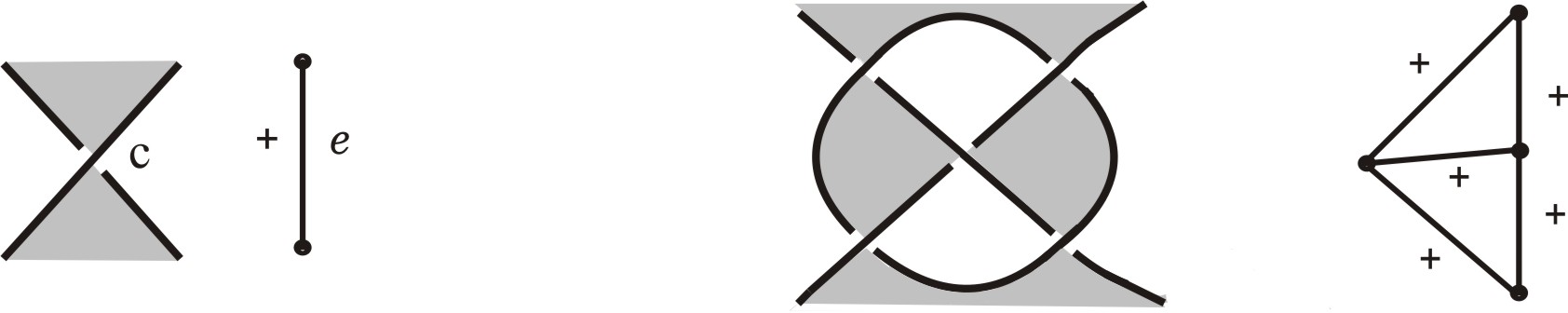}
\caption{A positive tangle. }\label{positivetangle}}\end{figure}

\begin{theorem}\cite{chbili2020quasiI}\label{Thm:pre:m}
 Let $L$ be a quasi-alternating link diagram at a crossing $c$.
Let $L'$ be the  link obtained from  $L$ by replacing $c$ with a reduced alternating tangle $T$ that extends $c$. Then $L'$ is  quasi-alternating at every crossing of $T$.
 \end{theorem}

\section{The Main Theorem}\label{main theorem}
\noindent Before outlining the main result in this section, we will first introduce some key definitions.

\begin{definition} A crossing $c$ in a link diagram $L$ is said to satisfies the determinant property, if the following condition holds at crossing $c$
\[ \det(L_0),~~ \det(L_\infty) \geq 1 \quad \text{and}\quad  \det(L) = \det(L_0) + \det(L_\infty),\]
where $L_0$ and $L_{\infty}$ are the diagrams as shown in Fig.~\ref{1}.
\end{definition}
\begin{definition} A quasi-alternating crossing $c$ in a link diagram $L$ is considered to have property (I) if, after switching crossing $c$, the resulting diagram is a diagram of a quasi-alternating link. An example of such a crossing $c$  with property (I) is depicted in Fig.~\ref{exm:In1}(a).
\end{definition}

\begin{definition} A quasi-alternating crossing $c$ in a  link diagram $L$ is said to have property (II), if the following conditions hold at crossing $c$
\begin{enumerate}
\item $\det(L_0)<\det(L_{\infty})$, where $L_0$ and $L_{\infty}$ are the diagrams as shown in Fig.~\ref{1},
\item property (I) is satisfied.
\end{enumerate}
\end{definition}

\begin{definition} A quasi-alternating crossing $c$ in a  link diagram $L$ is said to have property (III), if the following conditions hold at crossing $c$
\begin{enumerate}
\item $\det(L_{\infty})<\det(L_{0})$, where $L_0$ and $L_{\infty}$ are the diagrams as shown in Fig.~\ref{1},
\item property (I) is satisfied.
\end{enumerate}
\end{definition}
\noindent This section is divided into subsections~\ref{subsection:Alt} and \ref{subsection:NonAlt}.
\subsection{Quasi-alternating links via alternating tangles of opposite type}
\label{subsection:Alt}
\begin{theorem}\label{Thm:alt_tan_op} 
Let $L$ be a quasi-alternating link diagram with a quasi-alternating crossing $c$, and let $T$ be an alternating tangle of opposite type with respect to $c$. Then, a link $L'$ obtained from $L$ by replacing crossing $c$ with a tangle $T'$  is quasi-alternating, if either of the following conditions holds:
\begin{enumerate}
\item $T'\in \{T_+, T_-\}$ and  property (II) holds at $c$, or
\item $T'\in \{T^+, T^-\}$ and  property (III) holds at $c$ ,
\end{enumerate}
where $T_+,~T_-,~T^+$ and $T^-$ are the tangles depicted in Fig.~\ref{T+}.\\
 More precisely, every crossing of the tangle $T$ is  quasi-alternating in $L'$.
\end{theorem}
\noindent To prove Theorem~\ref{Thm:alt_tan_op}, we first need to establish the following lemma.
\begin{lemma}\label{lemma:opptwist} 
Let $L$ be a quasi-alternating link diagram with a quasi-alternating crossing $c$ such that property (I) holds at crossing $c$.
Let $L^{\overline{n}}$ and $L^{-\overline{n}}$ be  the links obtained from $L$ by replacing  crossing $c$ with $n > 1$ vertical and horizontal half twists of opposite type, respectively, see Fig.~\ref{twop}. Then 
\begin{enumerate}
\item $L^{\overline{n}}$ is quasi-alternating, if $\det(L_{0})< \det(L_{\infty})$, and
\item $L^{-\overline{n}}$ is quasi-alternating, if $\det(L_{0})>\det(L_{\infty})$, 
\end{enumerate} where $L_{\infty}$ and $L_{0}$ are the diagrams as shown in Fig.~\ref{1}. Furthermore, every new crossing of $L^{\overline{n}}$ and $L^{-\overline{n}}$ is quasi-alternating.
\end{lemma}
\begin{proof}
Let $c$ be a quasi-alternating crossing of a quasi-alternating link diagram $L$ and $e$ be its corresponding edge in $\mathcal{G}(L)$.
Let $L^{\overline{n}}$ and $L^{-\overline{n}}$ be the links obtained from $L$ by replacing  crossing $c$ with $n$ vertical and horizontal half twists of opposite type, respectively, as illustrated in Fig.~\ref{twop}. 
 \begin{figure}[!ht] {\centering
 \subfigure[$L^{\overline{n}}$]{\includegraphics[scale=0.5]{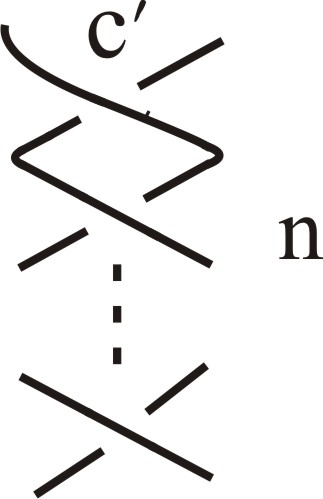} } \hspace{1cm}
\subfigure[$L$] {\includegraphics[scale=0.5]{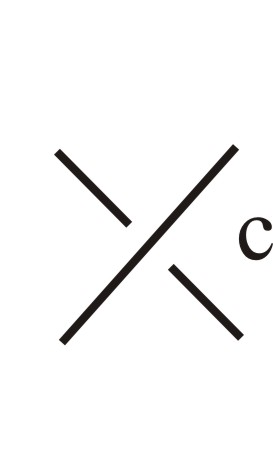}}\hspace{1cm}
\subfigure[$L^{-\overline{n}}$] {\includegraphics[scale=0.5]{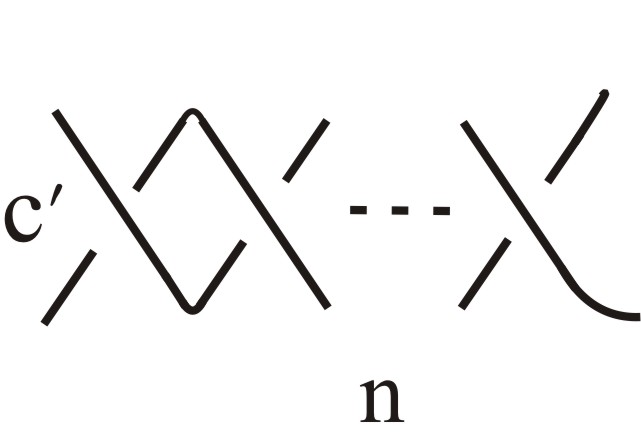}}
\caption{ Links $L$, $L^{\overline{n}}$ and $L^{-\overline{n}}$.}\label{twop}}\end{figure}

\noindent Choose the checkerboard coloring of $L$ such that the edge $e(=u_1u_2)$ in $\mathcal{G}(L)$ is positive. Let $\mathcal{G}(L^{\overline{n}})$ and $\mathcal{G}(L^{-\overline{n}})$ be the Tait graphs of $L^{\overline{n}}$ and $L^{-\overline{n}}$, respectively, with the induced checkerboard coloring. First, we consider the case of vertical half twists.

\noindent Case (1) Let $c'$ be any crossing that belongs to the new $n$ vertical half twists in  $L^{\overline{n}}$. Without lose of generality, we take $c'$ as shown in Fig.~\ref{twop}(a), and $e'$ be its corresponding edge in $\mathcal{G}(L^{\overline{n}})$.  Then the total number of spanning trees in $L^{\overline{n}}$ with $v$ number of  positive edges are given as

\[s_{v}(L^{\overline{n}})=\sharp \left \{\text{ spanning trees ~} \mathcal{T} \text{~of~} L ~| ~e\in \mathcal{T}, \text{~and~} \mathsf{v}(\mathcal{T})=v+1 \right \} + \]
 \[\hspace{1cm} n~~ \sharp \left \{ \text{ spanning trees ~} \mathcal{T} \text{~of~} L ~| ~e\notin \mathcal{T}, \text{~and~} \mathsf{v}(\mathcal{T})=v \right \} . \]

 \[ s_{v}(L^{\overline{n}})=s_{v}(L_0)+ns_{v}(L_{\infty})\]

 \begin{equation}\label{eq2:twop}
\begin{split}
 \det(L^{\overline{n}})&=\Big| \displaystyle \sum_{v}(-1)^{v}s_{v}(L^{\overline{n}})\Big|=\Big|\displaystyle \sum_{v}(-1)^{v} \left\{s_{v}(L_0)+n s_{v}(L_{\infty})\right\} \Big| \\
&=\Big|\displaystyle \sum_{v}(-1)^{v} s_{v}(L_0)+n\displaystyle \sum_{v}(-1)^{v}s_{v}(L_{\infty}) \Big| = \Big|-\mathfrak{a}+n\mathfrak{b} \Big|
\end{split}
\end{equation}
Since $c$ is a quasi-alternating crossing of $L$, we have $\mathfrak{a}\mathfrak{b}>0$ by Remark~\ref{rem:ab}.
It is easy to observe that by performing smoothings at crossing $c'$ in $L^{\overline{n}}$,  $L^{\overline{n}}_{0}$ is equivalent to the link $L_{\infty}$, while $L^{\overline{n}}_{\infty}$ corresponds to the link $L^{\overline{n-1}}$.

\noindent Therefore, we have $\det(L^{\overline{n}}_{0})=\det(L_{\infty})=|\mathfrak{b}|$. Using Eq.~(\ref{eq2:twop}), we find that
\[\det(L^{\overline{n}}_{\infty})=\det(L^{\overline{n-1}})=|-\mathfrak{a}+(n-1)\mathfrak{b}|.\]
Thus,  the determinant condition holds at crossing $c'$ in $L^{\overline{n}}$ if and only if 
\[\{-\mathfrak{a}+(n-1)\mathfrak{b}\} \mathfrak{b}>0.\]
That means the determinant condition holds at $c'$ if and only if for $\mathfrak{b}\gtrless 0$, 
\begin{equation}\label{eq3:twop}(-\mathfrak{a}+(n-1)\mathfrak{b})\gtrless 0 \iff (n-1)\mathfrak{b} \gtrless \mathfrak{a}\iff (n-1)|\mathfrak{b}|>|\mathfrak{a}| \end{equation}

\noindent Next, we proof the result using induction hypothesis. For $n=1$, $L^{\overline{1}}$ is a quasi-alternating link since property (I) holds at $c$ in $L$. For $n=2$, using Eq.~(\ref{eq3:twop}) the determinant condition holds at $c'$ whenever $|\mathfrak{b}|>|\mathfrak{a}|$. Furthermore, $L^{\overline{2}}_{\infty}$ is nothing but the link  $L^{\overline{1}}$ and $L^{\overline{2}}_{0}$ is equivalent to $L_{\infty}$. Since both $L^{\overline{1}}$ and $L_{\infty}$ are quasi-alternating by hypothesis, $c'$ is a quasi-alternating crossing of $L^{\overline{2}}$. Assume that the result is true for all positive integers less than $n$. We have to prove the result for $n$. 

\noindent Given that $\det(L_0)<\det(L_{\infty})$, this implies $|\mathfrak{a}|<|\mathfrak{b}|$. It is evident that for $|\mathfrak{a}|<|\mathfrak{b}|$, Eq.~(\ref{eq3:twop}) holds, and hence the determinant condition is satisfied at $c'$ in $L^{\overline{n}}$. Moreover, $L^{\overline{n}}_{\infty}$ is nothing but the link $L^{\overline{n-1}}$, which is quasi-alternating by induction hypothesis, and $L^{\overline{n}}_{0}$ is equivalent to $L_{\infty}$ which is also quasi-alternating since $L$ is quasi-alternating at $c$. Since both $L^{\overline{n}}_{0}$ and  $L^{\overline{n}}_{\infty}$ are quasi-alternating, $c'$ is a quasi-alternating crossing of  $L^{\overline{n}}$. Hence $L^{\overline{n}}$ is a quasi-alternating link.\\

\noindent Now, we consider the case of horizontal half twists.

\noindent Case (2) Let $c'$ be any crossing that belongs to the new $n$ horizontal half twists in  $L^{-\overline{n}}$. Without lose of generality, we take $c'$ as shown in Fig.~\ref{twop}(c), and $e'$ be its corresponding edge in $\mathcal{G}(L^{-\overline{n}})$.  Then

\[s_{v}(L^{\overline{n}})=n~\sharp \left \{\text{ spanning trees ~} \mathcal{T} \text{~of~} L ~| ~e\in \mathcal{T}, \text{~and~} \mathsf{v}(\mathcal{T})=v+1 \right \} + \]
 \[\hspace{1cm}  \sharp \left \{ \text{ spanning trees ~} \mathcal{T} \text{~of~} L ~| ~e\notin \mathcal{T}, \text{~and~} \mathsf{v}(\mathcal{T})=v \right \} . \]
 
 \[ s_{v}(L^{-\overline{n}})=ns_{v}(L_0)+s_{v}(L_{\infty})\]

 \begin{equation}\label{eq5:twop}
\begin{split}
 \det(L^{-\overline{n}})&=\Big| \displaystyle \sum_{v}(-1)^{v}s_{v}(L^{-\overline{n}})\Big|=\Big|\displaystyle \sum_{v}(-1)^{v} \left\{ ns_{v}(L_0)+ s_{v}(L_{\infty})\right\} \Big| \\
&=\Big|n\displaystyle \sum_{v}(-1)^{v} s_{v}(L_0)+\displaystyle \sum_{v}(-1)^{v}s_{v}(L_{\infty} )\Big| = \Big|-n\mathfrak{a}+\mathfrak{b} \Big|
\end{split}
\end{equation}
Since $c$ is a quasi-alternating crossing of $L$,  we have $\mathfrak{a}\mathfrak{b}>0$ holds by Remark~\ref{rem:ab}.
It is evident that by performing smoothings at crossing $c'$ in $L^{-\overline{n}}$,  $L^{-\overline{n}}_{\infty}$ is equivalent to the link $L_{0}$ and $L^{-\overline{n}}_{0}$ corresponds to the link $L^{-\overline{(n-1)}}$.

 \noindent Therefore, we have, $\det(L^{-\overline{n}}_{\infty})=\det(L_{0})=|\mathfrak{a}|$. Using Eq.~(\ref{eq5:twop}), we find that 
\[\det(L^{-\overline{n}}_{0})=\det(L^{-\overline{(n-1)}})=|-(n-1)\mathfrak{a}+\mathfrak{b}|.\]
Thus, from Eq.~(\ref{eq5:twop}) the determinant condition holds at crossing $c'$ in $L^{-\overline{n}}$ if and only if 
\[\{-(n-1)\mathfrak{a}+\mathfrak{b}\} (-\mathfrak{a})>0.\]
That means determinant condition holds if and only if for $\mathfrak{a}\gtrless 0$, $\mathfrak{-a}\lessgtr 0$ and
\begin{equation}\label{eq6:twop}
-(n-1)\mathfrak{a}+\mathfrak{b} \lessgtr 0 \iff \mathfrak{b} \lessgtr (n-1)\mathfrak{a}\iff |\mathfrak{b}|<(n-1)|\mathfrak{a}| 
\end{equation}

\noindent Next, we use induction hypothesis to prove the result. Clearly, for $n=1$, $L^{-\overline{1}}$ is a quasi-alternating link since property (I) holds at $c$ in $L$. For $n=2$, using Eq.~(\ref{eq6:twop}) the determinant condition holds at $c'$, whenever $|\mathfrak{b}|<|\mathfrak{a}|$. Further, $L^{-\overline{2}}_{0}$ is nothing but the link $L^{-\overline{1}}$ and $L^{-\overline{2}}_{\infty}$ is equivalent to $L_{0}$. Since both $L^{-\overline{1}}$ and $L_{0}$ are quasi-alternating by hypothesis, $c'$ is a quasi-alternating crossing of $L^{-\overline{2}}$ and hence $L^{-\overline{2}}$ is quasi-alternating. Assume that the result is true for all positive integers less than $n$. We have to prove the result for $n$.

\noindent Given that $\det(L_0)>\det(L_{\infty})$, this implies $|\mathfrak{a}|>|\mathfrak{b}|$. It is evident that for $|\mathfrak{b}|<|\mathfrak{a}|$,
 Eq.~(\ref{eq6:twop}) holds well, and the determinant condition holds at $c'$ of $L^{-\overline{n}}$. Moreover, $L^{-\overline{n}}_{0}$ is corresponds to the link $L^{-\overline{(n-1)}}$, which is quasi-alternating by induction hypothesis, and $L^{\overline{n}}_{\infty}$ is equivalent to $L_{0}$ which is also quasi-alternating since $L$ is quasi-alternating at $c$. Since both $L^{-\overline{n}}_{0}$ and  $L^{-\overline{n}}_{\infty}$ are quasi-alternating, $c'$ is a quasi-alternating crossing of $L^{-\overline{n}}$ and hence $L^{-\overline{n}}$ is quasi-alternating link. Hence the proof of the result.
\end{proof}

\begin{example} A diagram, $L$, of  $(2,8)$-torus link as shown in Fig.~\ref{exm:In1}(a) is quasi-alternating at crossing $c$. It is evident that $c$ holds property (I) with
\[\det(L_{\infty})=1<\det(L_0)=7.\] 
If crossing $c$ is replaced by two horizontal half twists of opposite type, then the resulting knot, $8_1$ depicted in Fig.~\ref{exm:8_1 & 11n}(a) is quasi-alternating.
\end{example}
 \begin{figure}[!ht] {\centering
 \subfigure[$8_1$ knot]{\includegraphics[scale=0.35]{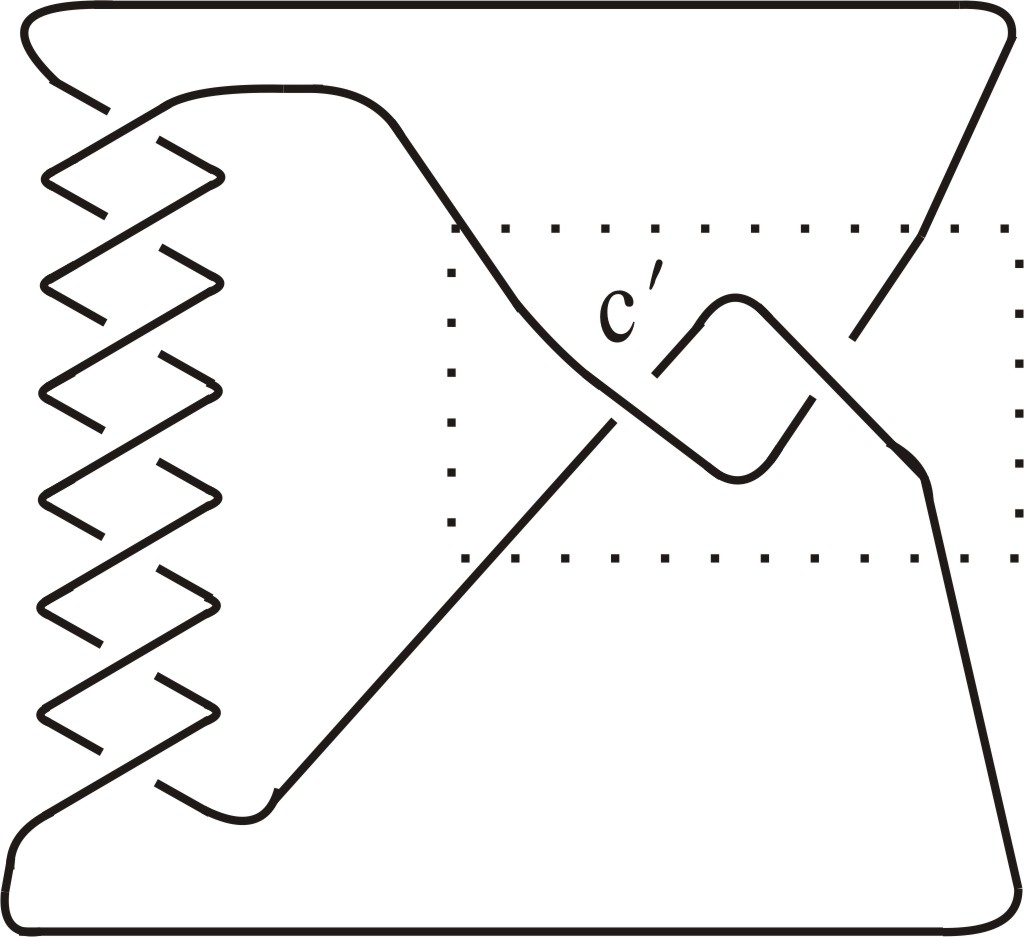} } \hspace{1.3cm}
\subfigure[$11n_{90}$ knot] {\includegraphics[scale=0.46]{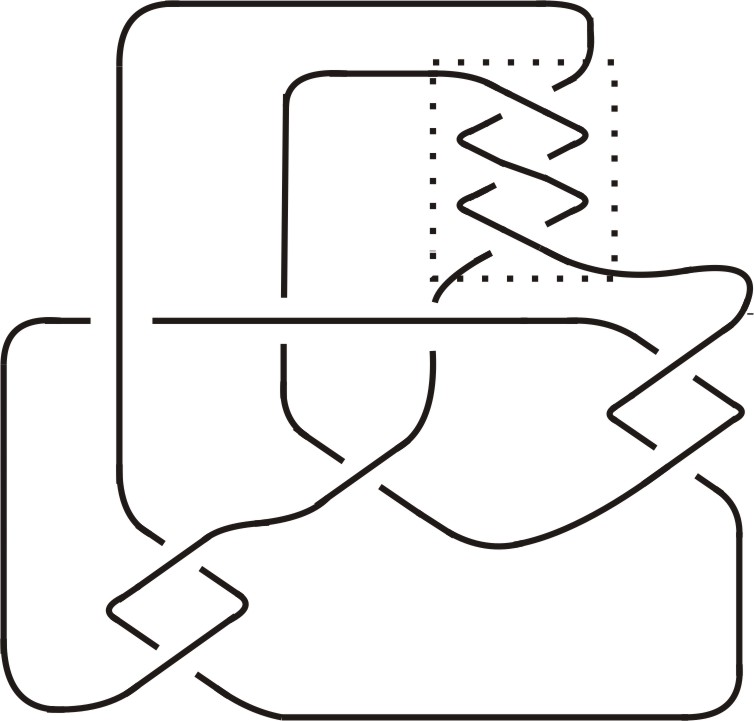}} \hspace{1.3cm}
\subfigure[$9_{44}$ knot] {\includegraphics[scale=0.35]{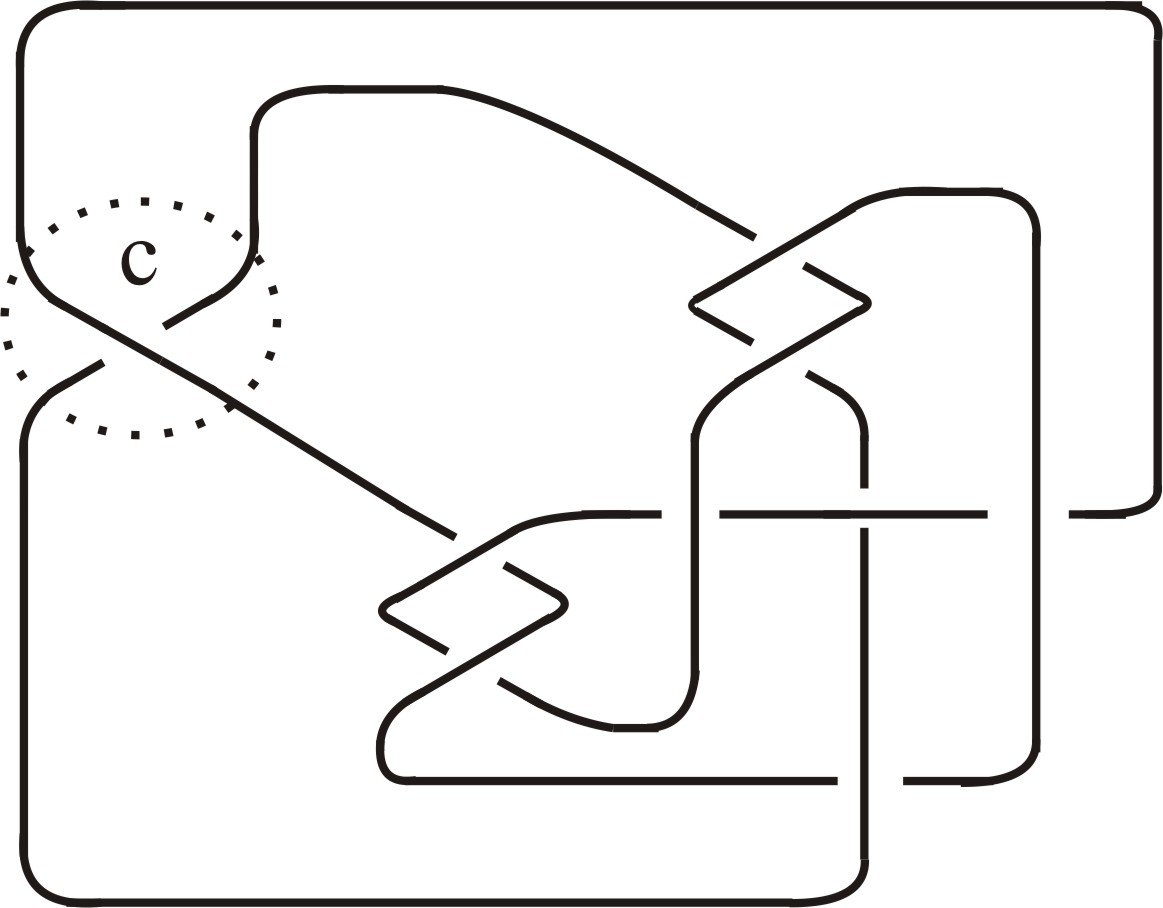}}
\caption{ }\label{exm:8_1 & 11n}}\end{figure}

\begin{example} Consider a crossing $c$ of a quasi-alternating knot $L$ as shown in Fig.~\ref{exm:43}. Then $L_{0}$ represents a diagram of the trivial knot, while $L_{\infty}$ corresponds to an alternating link, as illustrated in Fig.~\ref{exm:43in}. Further, 
\[\det(L_0)=1, \quad \det(L_{\infty})=14 , \quad\text{~ and~} \quad\det(L)=15.\]
  It is evident that both $L_0$ and $L_{\infty}$ are quasi-alternating and the determinant property holds at $c$. Additionally, the knot obtained from $L$ by switching crossing $c$ is a quasi-alternating $9_{43}$ knot. Therefore, $c$ is a quasi-alternating crossing and satisfying property (I), with $\det(L_0)<\det(L_{\infty})$.
 Replacing crossing $c$ with three vertical half twists of opposite type yields the resulting knot $11n_{90}$, depicted in Fig.~\ref{exm:8_1 & 11n}(b), is also quasi-alternating.
\end{example}
 \begin{figure}[!ht] {\centering
 \includegraphics[scale=0.45]{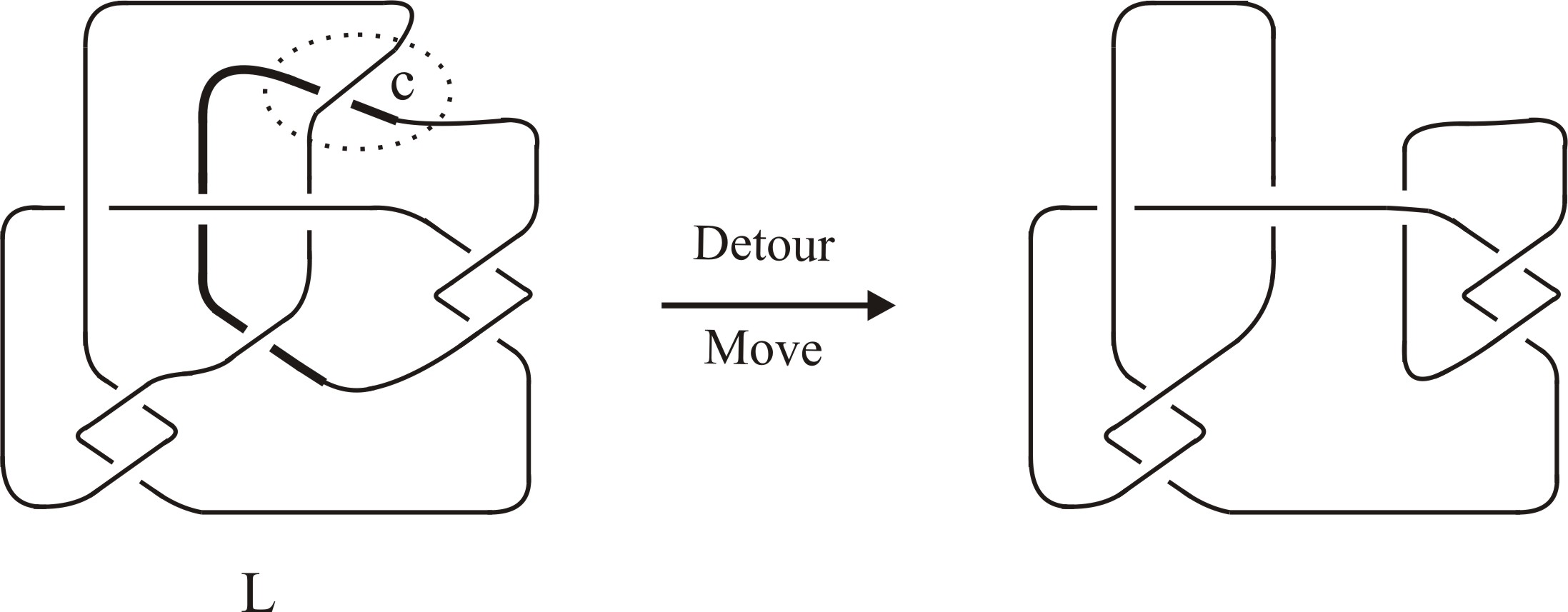} 
\caption{Knot equivalent to connected sum of figure eight and trefoil knot.}\label{exm:43}}\end{figure}

 \begin{figure}[!ht] {\centering
 \includegraphics[scale=0.45]{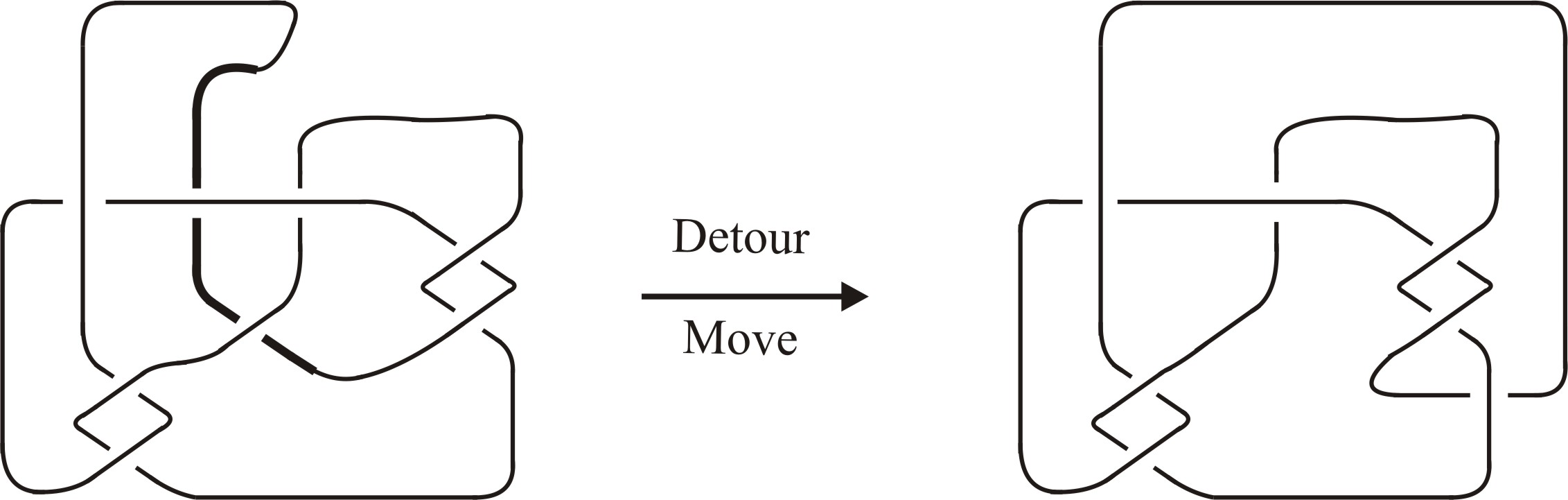} 
\caption{A diagram of an alternating link.}\label{exm:43in}}\end{figure}

\begin{example} Let $c$ be a quasi-alternating crossing of a quasi-alternating knot $8_{21}$, as shown in Fig.~\ref{exm:In2}(a). In this case,  
\[\det(L_{\infty})=2, \quad \text{~and~}\quad \det(L_0)=13.\]
 Moreover, the knot obtained from  $8_{21}$ by switching crossing $c$ is $6_2$ knot, which is  quasi-alternating. Thus, crossing $c$ satisfies property (I), with $\det(L_{\infty})<\det(L_0)$. By replacing crossing $c$ with three vertical half twists of opposite type, the resulting knot $9_{42}$ is a non quasi-alternating knot, see Fig.~\ref{exm:In2}(b). 
\end{example}

\begin{example} Let $c$ be  a quasi-alternating crossing of a quasi-alternating knot, $L$, as shown in Fig.~\ref{exm:8_1 & 11n}(c). It is evident that $L_0$ is Hopf link, $L_{\infty}$ is $8_{21}$ knot, and the knot obtained by switching crossing $c$ is $6_3$. 
 Therefore, we have
 \[\det(L)=17=2+15=\det(L_0)+\det(L_{\infty}).\] 
 Since Hopf link, $8_{21}$  and $6_3$ knots are quasi-alternating, crossing $c$ is quasi-alternating and holds property (I) with $\det(L_0)<\det(L_{\infty})$. In this case, if crossing $c$ is replaced  by seven horizontal half twists of opposite type, then the resulting knot, $L^{-\overline{7}}$, is illustrated in Fig.~\ref{exm:In1}(b) has $\det(L^{-\overline{7}})=1$. Since $\det(L^{-\overline{7}})=1$, there does not exist any crossing in any diagram of $L^{-\overline{7}}$ satisfying the determinant property. Therefore, $L^{-\overline{7}}$ is not quasi-alternating.
\end{example}

\noindent Now, we will prove the Theorem~\ref{Thm:alt_tan_op}.
\begin{proof}
Consider a quasi-alternating link $L$ with a quasi-alternating crossing $c$ satisfies either property (II) or property (III). Let $T$ be an alternating tangle of opposite type with respect to crossing $c$.\\
(1) Assume that $c$ holds property (II). Then, $L^{\overline{2}}$ is a quasi-alternating link by Lemma~\ref{lemma:opptwist}. Let $c'_1$ and $c'_2$ be new vertical twisted crossings in $L^{\overline{2}}$ as illustrated in Fig.~\ref{fig1:thm:I II}(b). Then both $c'_1$ and $c'_2$ are quasi-alternating crossings by of Lemma~\ref{lemma:opptwist}, and the tangle $T$ is an alternating tangle of same type with respect to both $c'_1$ and $c'_2$.

\noindent Let $L'_1$ and $L'_2$ be the links obtained from $L^{\overline{2}}$ by  replacing crossing $c'_1$ and $c'_2$ with tangle $T$, respectively,  see Fig.~\ref{fig1:thm:I II}.
For $i=1,2$,  $c'_i$ is a quasi-alternating crossing and $T$ is an alternating tangle of same type with respect to $c'_i$. It follows from Theorem~\ref{Thm:pre:m} that  $L'_1$ and $L'_2$ are quasi-alternating links. Furthermore, Theorem~\ref{Thm:pre:m} implies that every crossing of $T$ is a quasi-alternating crossing in $L'_{i}$, for $i=1,2$.\\
It is evident that the links $L'_1$ and $L'_2$ are precisely the links obtained from $L$ by replacing crossing $c$ with alternating tangles $T_+$ and $T_{-}$, respectively, of opposite type. Hence  the result is true in this case.  Thus, whenever property (II) $c$ holds at crossing $c$ in $L$, the links obtained from $L$ by replacing crossing $c$ with tangles $T_+$ or $T_-$ are quasi-alternating. Hence, this construction preserves the quasi-alternating nature of the links.\\

 \begin{figure}[!ht] {\centering
 \subfigure[$L$]{\includegraphics[scale=0.5]{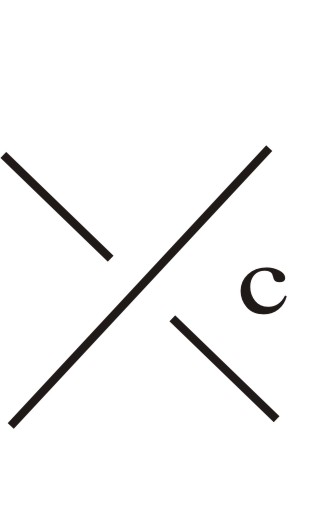} } \hspace{1.3cm}
\subfigure[$L^{\overline{2}}$] {\includegraphics[scale=0.45]{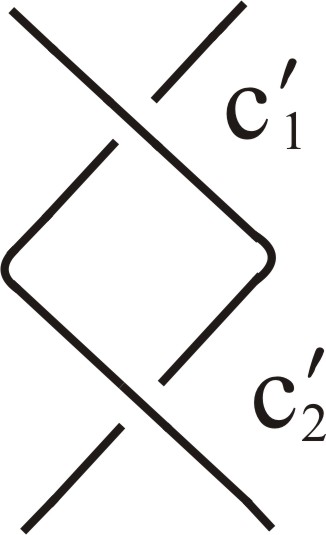}}\hspace{1.3cm}
\subfigure[$L'_1$] {\includegraphics[scale=0.5]{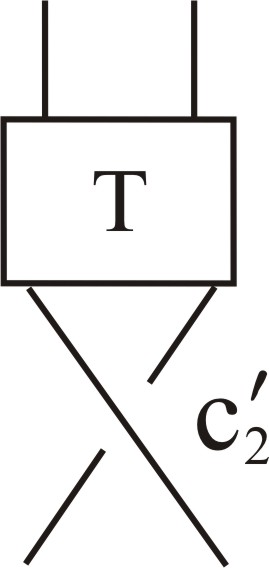}}\hspace{1.3cm}
\subfigure[$L'_2$] {\includegraphics[scale=0.5]{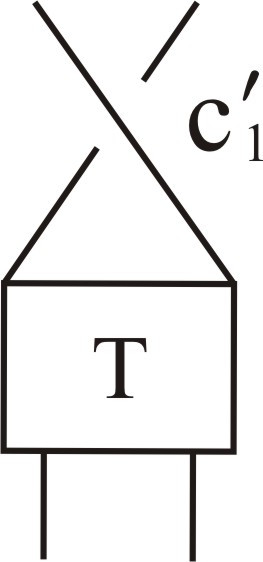}}
\caption{ }\label{fig1:thm:I II}}\end{figure}

\noindent(2) In the case where the crossing $c$ satisfies property (III), the  link $L^{-\overline{2}}$ is a quasi-alternating link by Lemma~\ref{lemma:opptwist}. Let $c'_1$ and $c'_2$ be new horizontal twisted crossings in $L^{-\overline{2}}$ as illustrated in Fig.~\ref{fig2:thm:I II}(b). Then both $c'_1$ and $c'_2$ are quasi-alternating crossings by of Lemma~\ref{lemma:opptwist}, and the tangle $T$ is an alternating tangle of same type with respect to both $c'_1$ and $c'_2$. 
 \begin{figure}[!ht] {\centering
 \subfigure[$L$]{\includegraphics[scale=0.45]{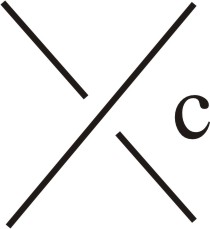} } \hspace{1cm}
\subfigure[$L^{-\overline{2}}$] {\includegraphics[scale=0.45]{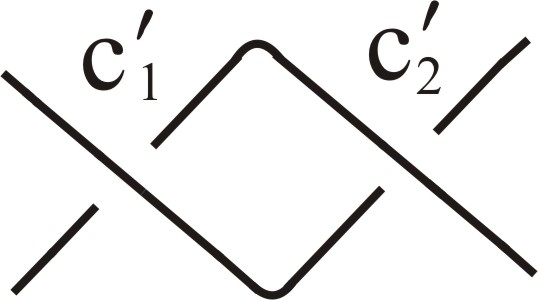}}\hspace{1cm}
\subfigure[$L'_1$] {\includegraphics[scale=0.45]{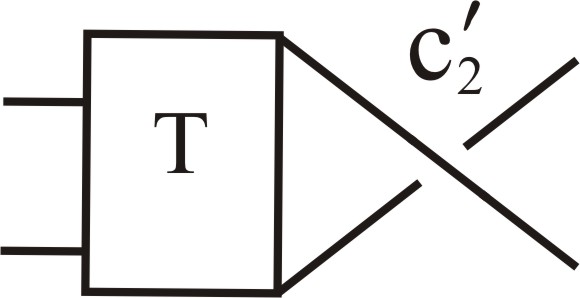}}\hspace{1cm}
\subfigure[$L'_2$] {\includegraphics[scale=0.45]{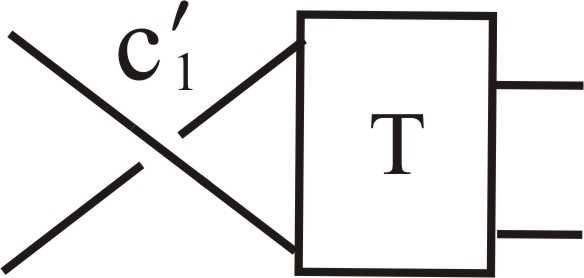}}
\caption{ }\label{fig2:thm:I II}}\end{figure}
The remaining proof will proceed in a manner similar to part 1.

\noindent Specifically, we define the links $L'_1$  and $L'_2$  obtained from $L$ by replacing crossings $c'_1$ and $c'_2$ with the tangle $T$, respectively, see Fig.~\ref{fig2:thm:I II}. Since both $c'_1$  and $c'_2$ are quasi-alternating crossings, and $T$ is an alternating tangle of the same type with respect to these crossings. It follows from Theorem ~\ref{Thm:pre:m} that $L'_1$ and $L'_2$ are quasi-alternating links. 
Additionally, Theorem~\ref{Thm:pre:m} guarantees that every crossing of $T$ is a quasi-alternating crossing in $L'_1$ and $L'_2$. Thus, the links $L'_1$   and  $L'_2$ are precisely the links obtained from $L$ by replacing crossing $c$ with alternating tangles $T^+$ and $T^{-}$, respectively, of opposite type. Hence, this construction preserves the quasi-alternating nature of the links.
\end{proof}

\noindent A natural question arises about the existence of links and crossings satisfying property (I), (II) and (III). The answer to this question is stated in the following theorem.
\begin{theorem}\label{Thm1} 
There exist infinitely many quasi-alternating links with quasi-alternating crossings satisfying property (II) and (III).
\end{theorem}
\noindent Before proving this theorem, we state and prove the following lemma.
\begin{lemma}\label{lemma:sametwist} 
Let $L'$ be a link obtained from a quasi-alternating diagram $L$ by replacing a quasi-alternating crossing $c$ of $L$ with $n>1$ half twists of same type. Let $c'$ be any new half twist in $L'$. Then 
\begin{enumerate}
\item $\det(L'_{\infty})<\det(L'_{0})$, whenever the twists at $c$ are vertical half twists as shown in Fig.~\ref{twsa}(a), and
\item $\det(L'_{\infty})>\det(L'_{0})$, whenever the twists at $c$ are horizontal half twists as shown in Fig.~\ref{twsa}(c), 

\end{enumerate} 
where $L'_{\infty}$ and $L'_{0}$ are the links obtained from $L'$ by perform smoothings at crossing $c'$ of $L'$ as shown in Fig.~\ref{1}. 
\end{lemma}
\begin{proof}
Let $c$ be a quasi-alternating crossing of a quasi-alternating link diagram $L$ and $e$ be its corresponding edge in $\mathcal{G}(L)$.
Let $L'$ be  the link obtained from $L$ by replacing  crossing $c$ with $n>1$ half twists of same type, as shown in Fig.~\ref{twsa}. 
 \begin{figure}[!ht] {\centering
 \subfigure[$L^{n}$]{\includegraphics[scale=0.5]{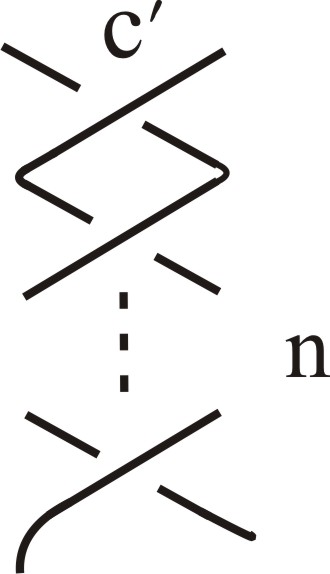} } \hspace{1cm}
\subfigure[$L$] {\includegraphics[scale=0.5]{Fig/19octa.jpg}}\hspace{1cm}
\subfigure[$L^{-n}$] {\includegraphics[scale=0.5]{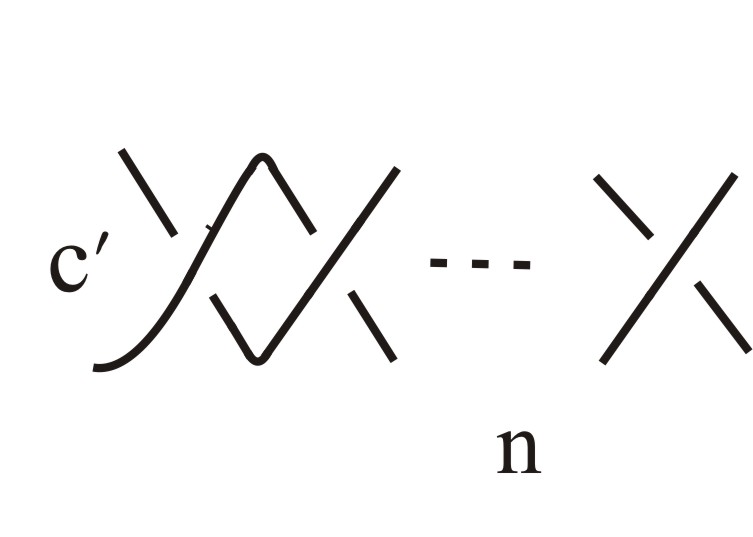}}
\caption{ Links $L$, $L^{n}$ and $L^{-n}$.}\label{twsa}}\end{figure}
\noindent Choose the checkerboard coloring of $L$ such that the edge $e(=u_1u_2)$ in $\mathcal{G}(L)$ is positive. Let $\mathcal{G}(L')$ be the Tait graph of $L'$ with the induced checkerboard coloring. First, we consider the case of vertical half twists.

\noindent Case (1) Let $c'$ be any crossing that belongs to the new $n$ vertical half twists in  $L'$. Without lose of generality, let $c'$ be as shown in Fig.~\ref{twsa}(a), and $e'$ be its corresponding edge in $\mathcal{G}(L')$. Denote $L'$ by $L^{n}$ here. Then
\[s_{v}(L^{n})=\sharp \left \{\text{ spanning trees ~} \mathcal{T} \text{~of~} L ~| ~e\in \mathcal{T}, \text{~and~} \mathsf{v}(\mathcal{T})=v-n+1 \right \} + \]
 \[\hspace{2cm} n~~ \sharp \left \{ \text{ spanning trees ~} \mathcal{T} \text{~of~} L ~| ~e\notin \mathcal{T}, \text{~and~} \mathsf{v}(\mathcal{T})=v-n+1 \right \} . \]
 \begin{equation}\label{eq1:twsame}
  s_{v}(L^{n})=s_{v-n}(L_0)+ns_{v-n +1}(L_{\infty})
  \end{equation}
 \begin{equation}\label{eq2:twsame}
\begin{split}
 \det(L^{n})&=\Big| \displaystyle \sum_{v}(-1)^{v}s_{v}(L^{n})\Big|=\Big|\displaystyle \sum_{v}(-1)^{v} \{s_{v-n}(L_0)+n s_{v-n+1}(L_{\infty}) \}\Big| \\
&=\Big|\displaystyle \sum_{v}(-1)^{v} s_{v-n}(L_0)+n\displaystyle \sum_{v}(-1)^{v}s_{v-n+1}(L_{\infty}) \Big|\\
& = \Big|-(-1)^n\mathfrak{a}+n(-1)^{n-1}\mathfrak{b} \Big|=\Big|\mathfrak{a}+n\mathfrak{b} \Big|
\end{split}
\end{equation}
 Since $c$ is a quasi-alternating crossing of $L$, we have  $\mathfrak{a} \mathfrak{b}>0$ by Remark~\ref{rem:ab}.
Moreover, by perform smoothings at crossing $c'$ in $L^n$,  $L^n_{0}$ corresponds to the link $L^{n-1}$ 
 and $L^{n}_{\infty}$ is equivalent to the link $L_{\infty}$.
Therefore, $\det(L^{n}_{\infty})=\det(L_{\infty})=|\mathfrak{b}|$. Using Eq.~(\ref{eq2:twsame}), we have 
\[\det(L^{n}_{0})=\det(L^{n-1})=|\mathfrak{a}+(n-1)\mathfrak{b}|>|\mathfrak{b}|=\det(L^{n}_{\infty}),\]
since $\mathfrak{a}\mathfrak{b}>0$ and $n>1$.
Hence the desired result.\\

\noindent Case (2) Suppose that twists in $L'$ are horizontal half twists. Let $c'$ be any crossing that belongs to the new $n$ horizontal half twists in  $L'$. Without lose of generality, let $c'$ be as shown in Fig.~\ref{twsa}(c), and $e'$ be its corresponding edge in $\mathcal{G}(L')$. Denote $L'$ by $L^{-n}$ here. Then
\[s_{v}(L^{n})=n~\sharp \left \{\text{ spanning trees ~} \mathcal{T} \text{~of~} L ~| ~e\in \mathcal{T}, \text{~and~} \mathsf{v}(\mathcal{T})=v \right \} + \]
 \[\hspace{2cm}\sharp \left \{ \text{ spanning trees ~} \mathcal{T} \text{~of~} L ~| ~e\notin \mathcal{T}, \text{~and~} \mathsf{v}(\mathcal{T})=v \right \} . \]
 \begin{equation}\label{eq4:twsame} s_{v}(L^{-n})=ns_{v-1}(L_0)+s_{v}(L_{\infty})
 \end{equation}
 \begin{equation}\label{eq5:twsame}
\begin{split}
 \det(L^{-n})&=\Big| \displaystyle \sum_{v}(-1)^{v}s_{v}(L^{-n})\Big|=\Big|\displaystyle \sum_{v}(-1)^{v} \{ns_{v-1}(L_0)+ s_{v}(L_{\infty})\} \Big| \\
&=\Big|n\displaystyle \sum_{v}(-1)^{v} s_{v-1}(L_0)+\displaystyle \sum_{v}(-1)^{v}s_{v}(L_{\infty}) \Big| = \Big|n\mathfrak{a}+\mathfrak{b} \Big|
\end{split}
\end{equation}
Since $c$ is a quasi-alternating crossing of $L$, we have $\mathfrak{a} \mathfrak{b}>0$ by Remark~\ref{rem:ab}.
It is evident that by smoothing at crossing $c'$ in $L^{-n}$,  $L^{-n}_{0}$ is equivalent to the link $L_{0}$ and $L^{-n}_{\infty}$ corresponds to the link $L^{-(n-1)}$. 
 Therefore, $\det(L^{-n}_{0})=\det(L_{0})=|\mathfrak{a}|$. Using Eq.~(\ref{eq5:twsame}), we have 
\[\det(L^{-n}_{\infty})=\det(L^{-(n-1)})=|(n-1)\mathfrak{a}+\mathfrak{b}|>|\mathfrak{a}|=\det(L^{-n}_{0}),\]
since $\mathfrak{a}\mathfrak{b}>0$ and $n>1$.
Hence, the desired result.
\end{proof}

\noindent Now we are going to prove Theorem~\ref{Thm1}.
\begin{proof}
Consider a quasi-alternating link diagram $L$ with a quasi-alternating crossing $c$. We construct two links $L^{-n}$ and $L^n$ from $L$ by making $n>1$  horizontal and vertical half  twists, respectively,  at crossing $c$ in $L$.  Let $c'$ and $c''$ be any new half twisted crossings in $L^{-n}$ and $L^n$, respectively. Then by Lemma~\ref{lemma:sametwist}, at crossing $c'$ and $c''$ following inequality holds  

\begin{equation}\label{eq.II,III}
\det(L^{-n}_{0})<\det(L^{-n}_{\infty}),\quad \text{and}\quad \det(L^{n}_{\infty})<\det(L^{n}_{0}).
\end{equation} 
Further, by apply crossing change operation at $c''$ the resulting diagram becomes either $L_{0}$ for $n=2$, or $L^{-(n-2)}$ for $n>2$. Similarly, by apply crossing change operation at $c'$ the resulting diagram becomes  either $L_{\infty}$ for $n=2$, or  $L^{(n-2)}$ for $n>2$. Since $L$ is quasi-alternating at $c$, both $L_0$ and $L_{\infty}$ are quasi-alternating. Moreover, $L^{-(n-2)}$ and $L^{(n-2)}$ for $n>2$ are quasi-alternating by Theorem~\ref{Thm:pre:m}. Thus Property (I) holds at crossings $c'$ and $c''$ in $L^{-n}$ and $L^{n}$, respectively. \\
It is evident from Eq.~(\ref{eq.II,III}) that the crossings $c'$ holds property (II) in $L^{-n}$, and  the crossings $c''$ holds property (III) in $L^{n}$.

\noindent Thus, for a given quasi-alternating link diagram with a quasi-alternating crossing, one can construct new examples of distinct quasi-alternating links with quasi-alternating crossings that meet property (II) and (III). Since there are infinitely many quasi-alternating links, there exist infinitely many quasi-alternating links with quasi-alternating crossings satisfying these properties.   This concludes the proof of the theorem.
\end{proof}

\subsection{Constructing quasi-alternating links via non-alternating tangles}
\label{subsection:NonAlt}

 \begin{figure}[ph] {\centering
 \subfigure[crossing $c$]{\includegraphics[scale=.55]{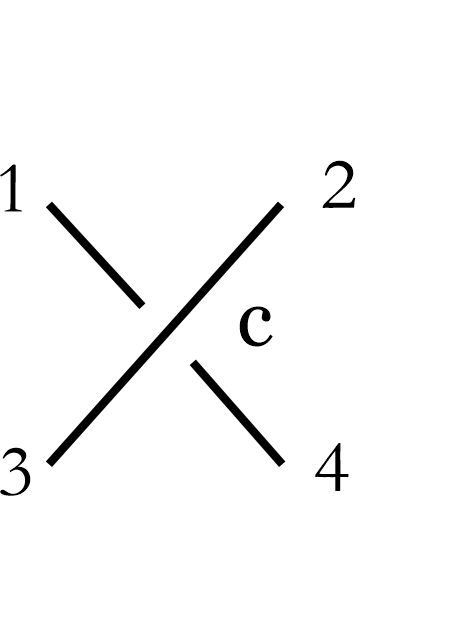} \label{Fig.lc}}\hspace{2cm}
 \subfigure[$T_{[q]}$]{\includegraphics[scale=.5]{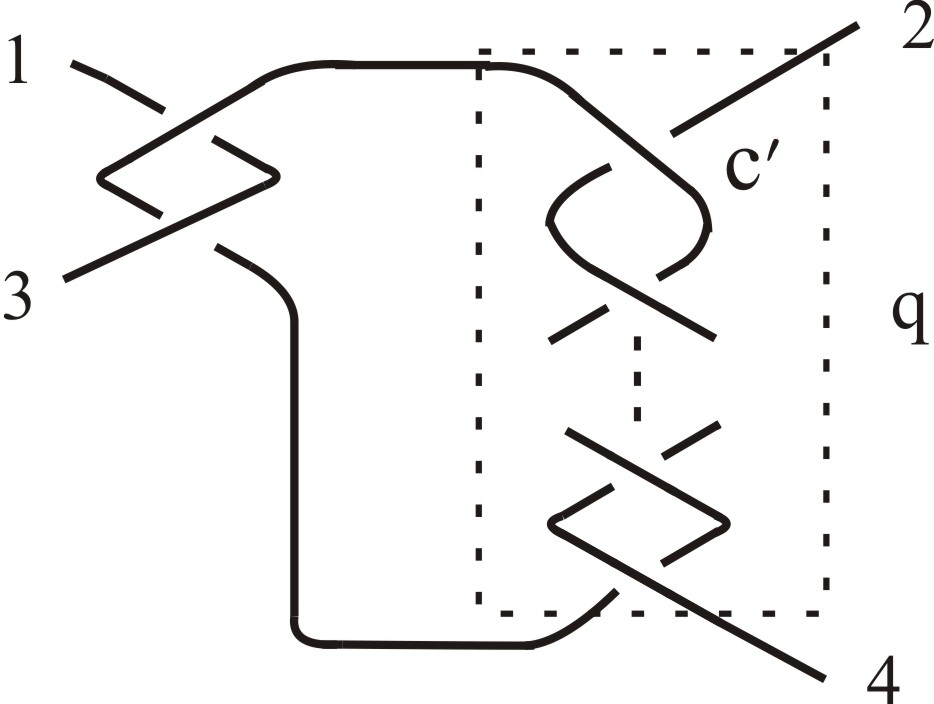} \label{Fig.t[q]}}\\
 
  \vspace{.3cm}
 \subfigure[$T^p$] {\includegraphics[scale=.5]{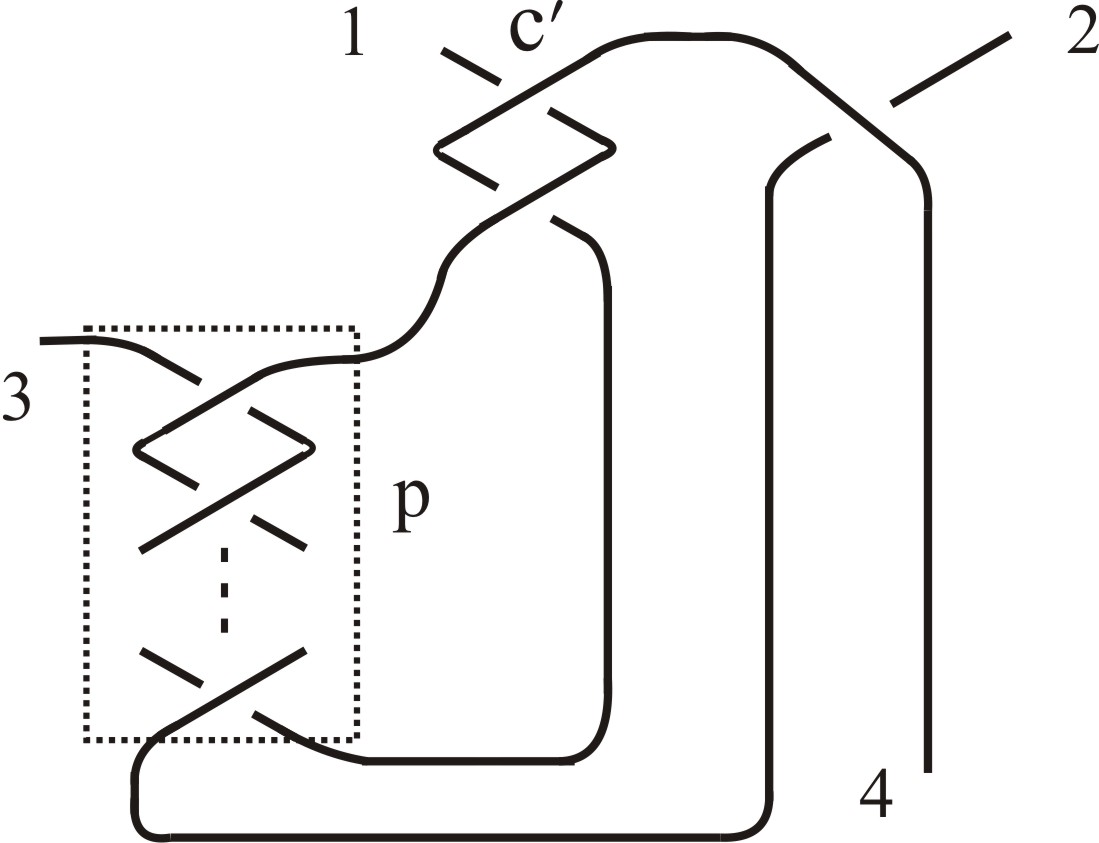}\label{Fig.t^p}}\hspace{1.5cm}
 \subfigure[$T_q$]{\includegraphics[scale=.5]{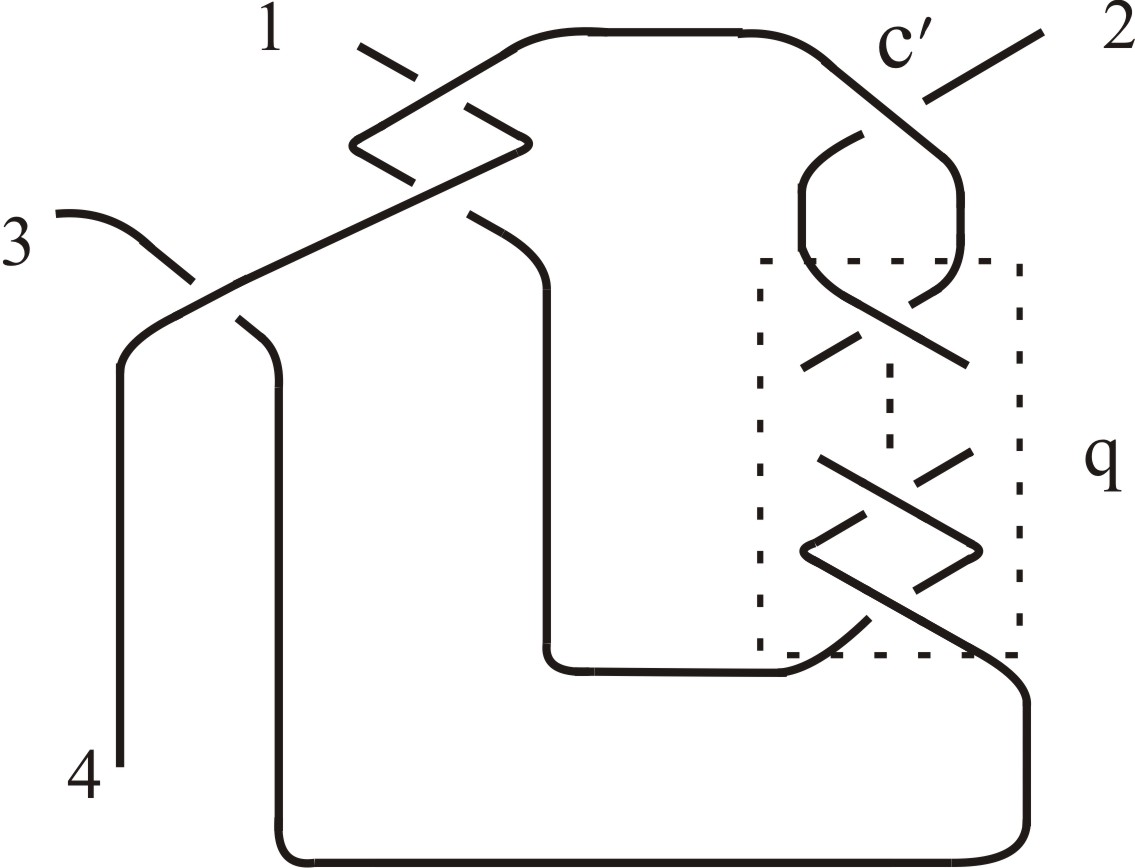}\label{Fig.t_q}}\\

 \subfigure[$T^{p,q}$]{\includegraphics[scale=.5]{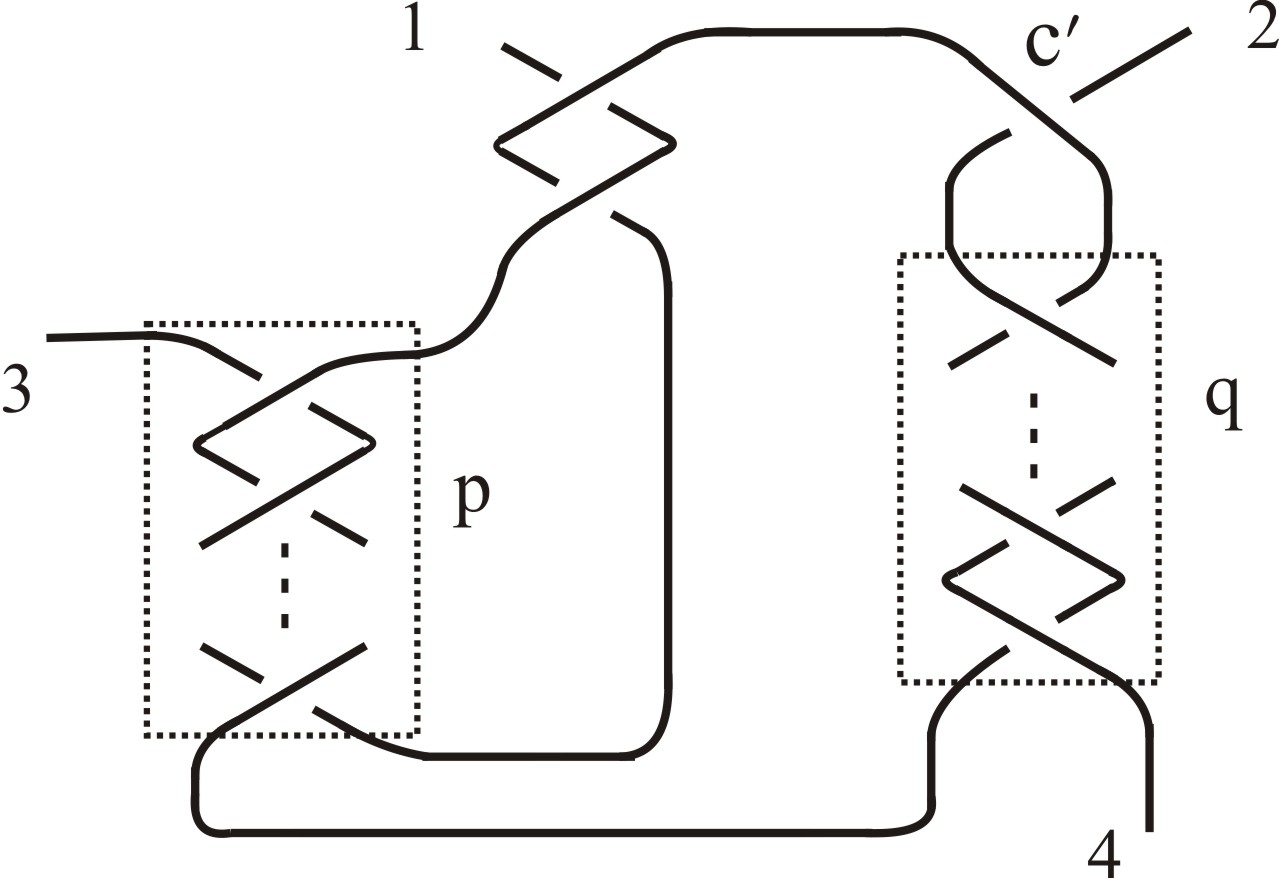}\label{Fig.t^pq}}\hspace{1.5cm}
 \subfigure[$T_{p,q}$]{\includegraphics[scale=.5]{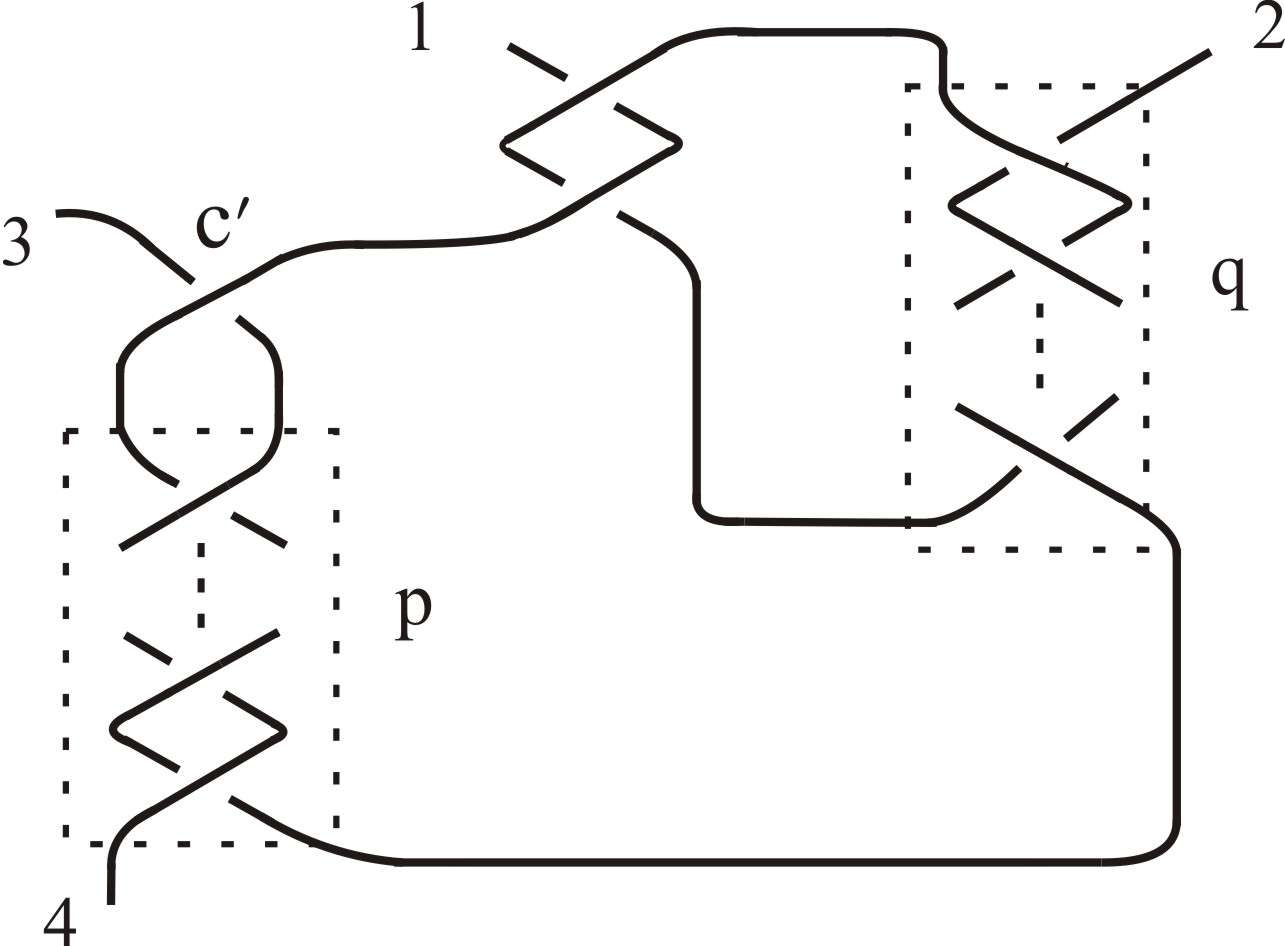}\label{Fig.t_pq}}\\

 \subfigure[$T^{[p,q]}_{\infty}$]{\includegraphics[scale=.5]{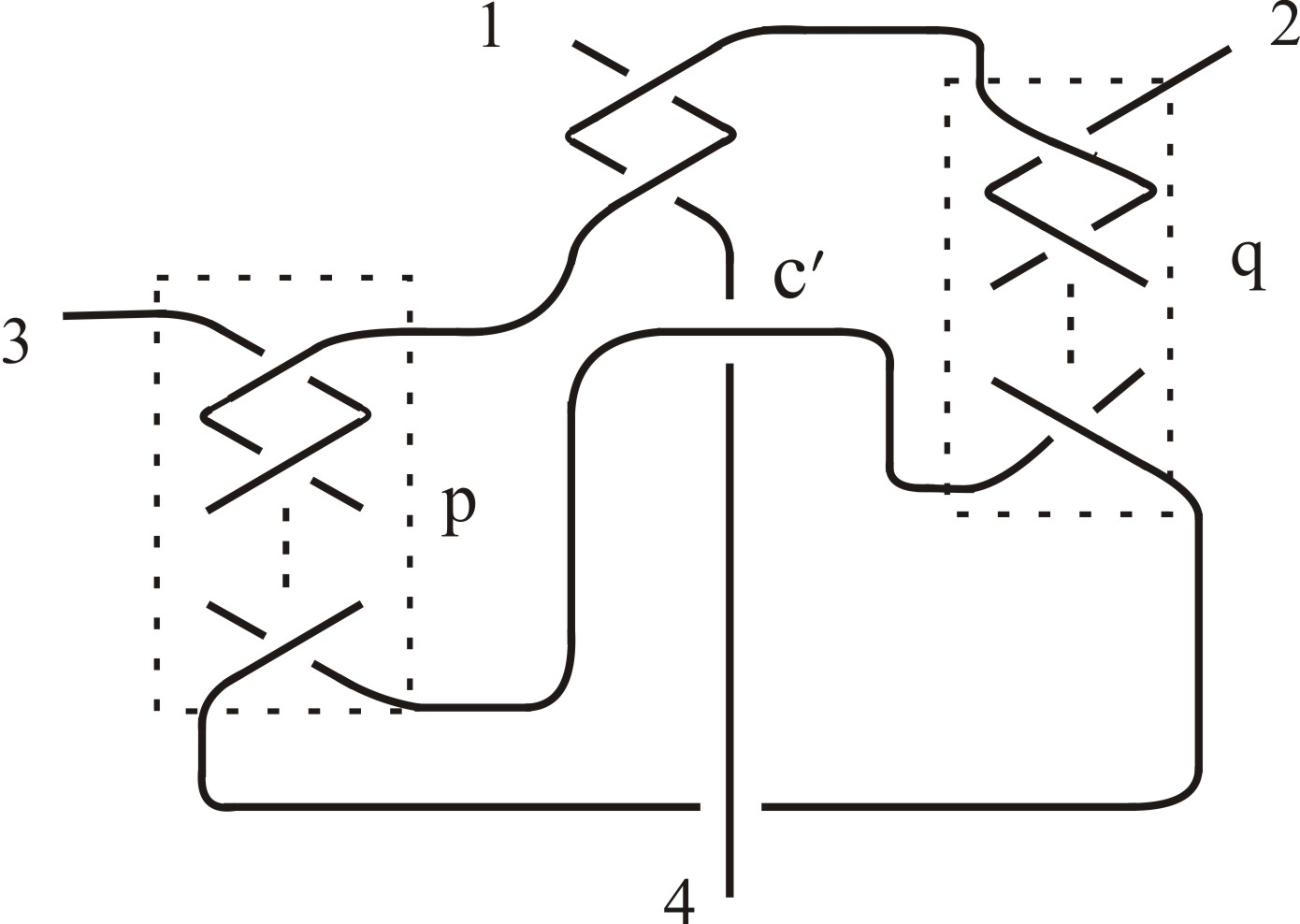} \label{Fig.t^{pq}} }\hspace{.5cm} 
 \subfigure[$T^{[p,q]}$]{\includegraphics[scale=.5]{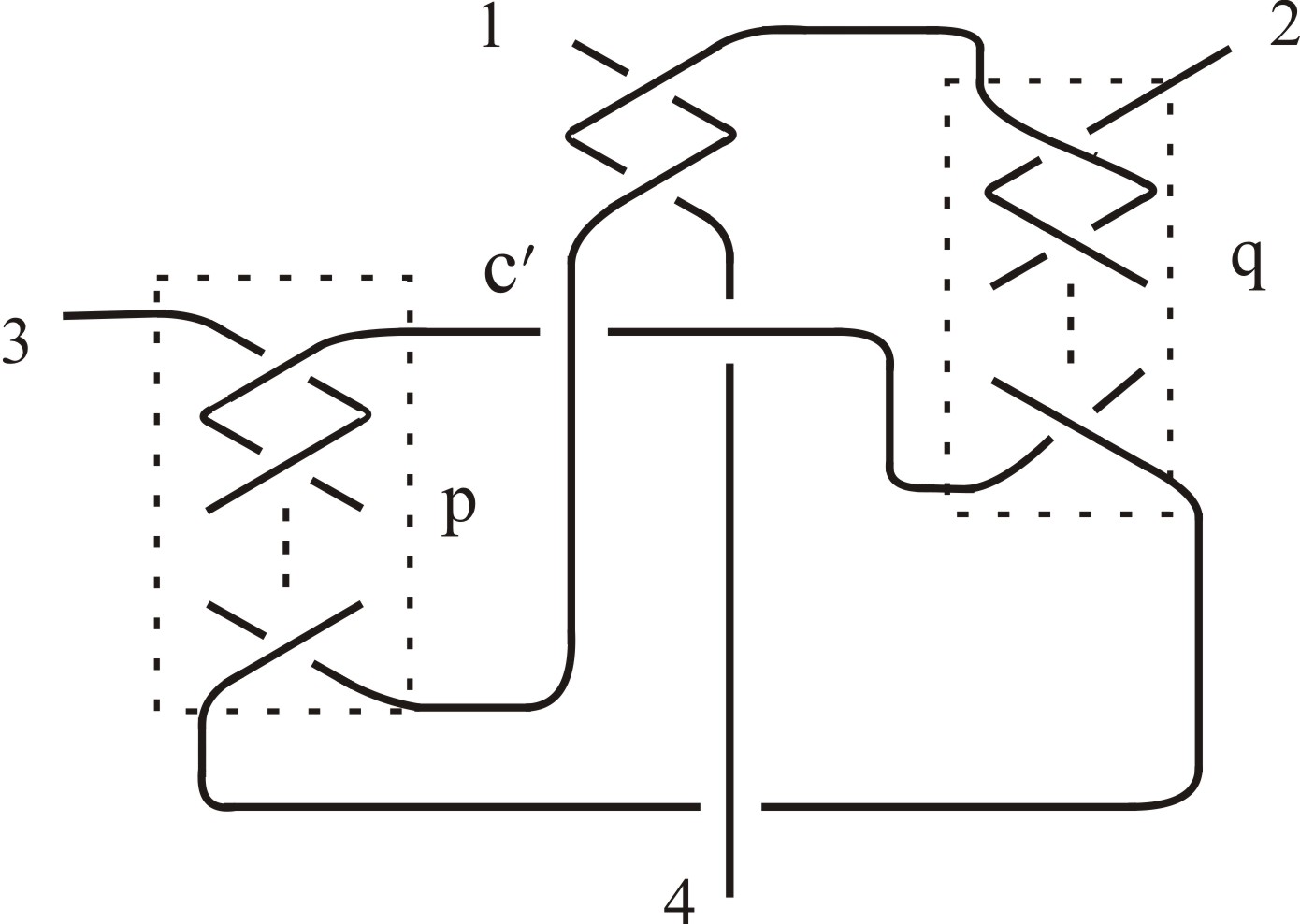}\label{Fig.t^[pq]}}
\caption{Crossing $c$ and non-alternating tangles with crossing $c'$.}\label{ntangles}}\end{figure}

 \begin{figure}[!ht] {\centering
 \subfigure[$\mathcal{G}(T_{[q]})$]{\includegraphics[scale=.5]{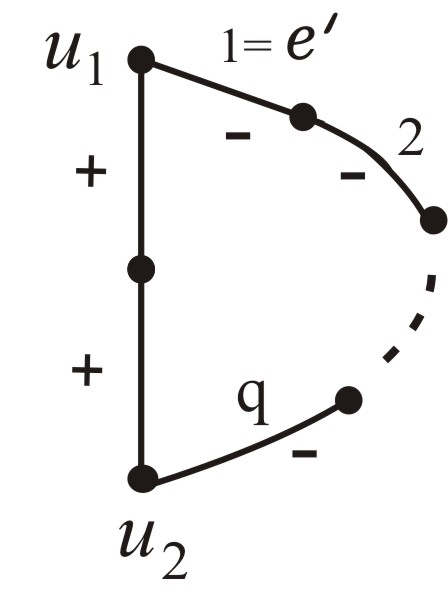} \label{Fig.gt[q]}}\hspace{.2cm}
 \subfigure[$\mathcal{G}(T^p)$] {\includegraphics[scale=.5]{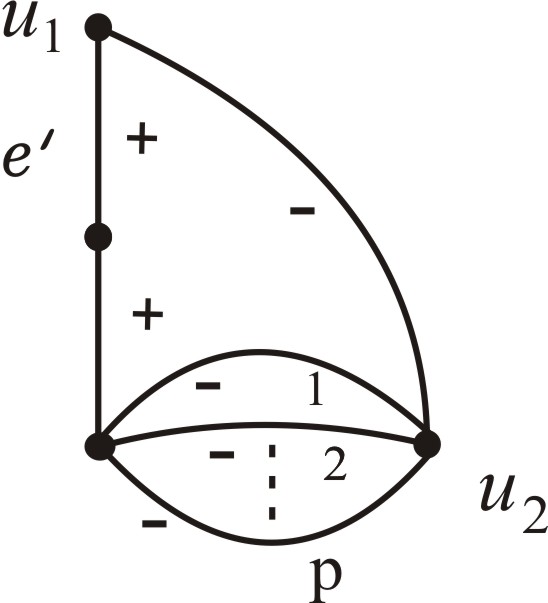}\label{Fig.gt^p}}\hspace{.2cm}
 \subfigure[$\mathcal{G}(T_q)$]{\includegraphics[scale=.5]{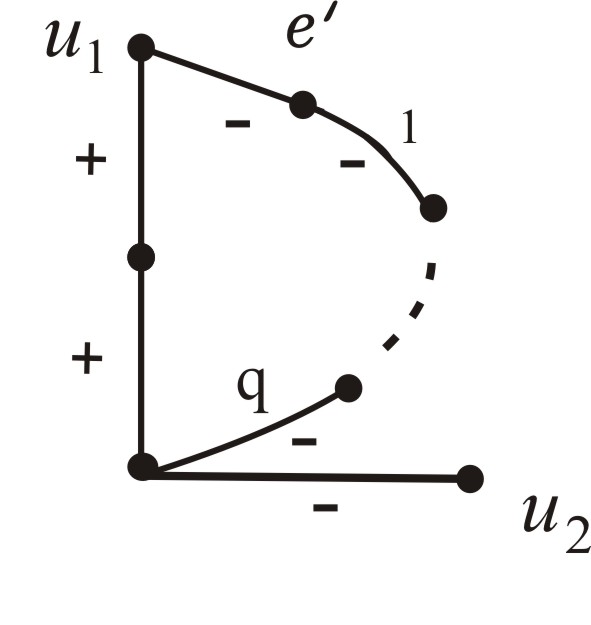}\label{Fig.gt_q}}\hspace{.2cm}
 \subfigure[$\mathcal{G}(T_{p,q}$)]{\includegraphics[scale=.5]{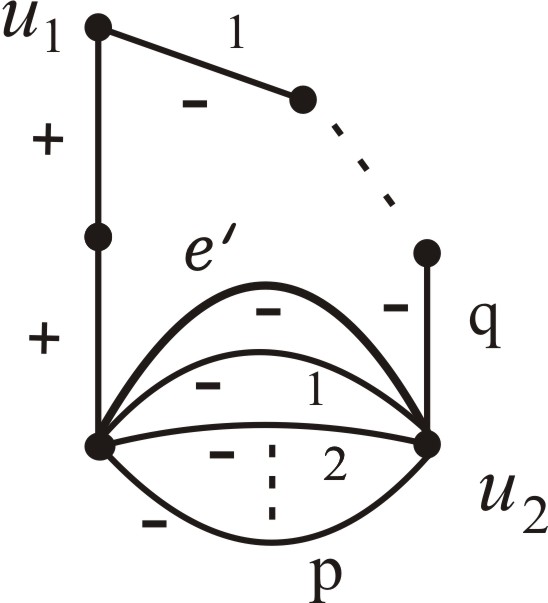}\label{Fig.gt_pq}}\\
\vspace{.4cm}
 \subfigure[$\mathcal{G}(T^{p,q})$]{\includegraphics[scale=.5]{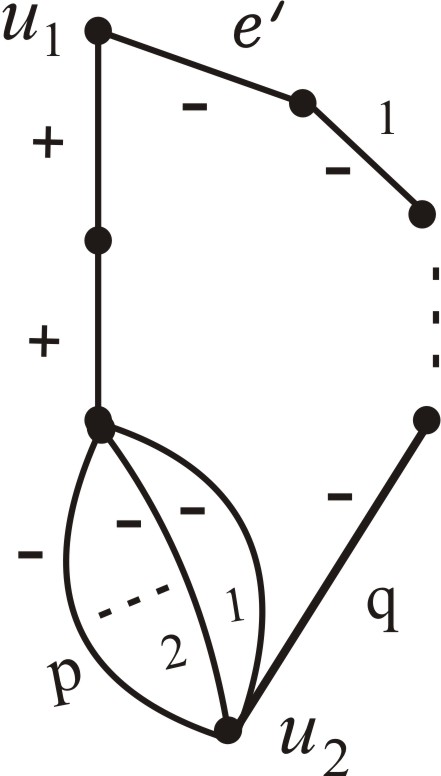}\label{Fig.gt^pq}}\hspace{.5cm} 
 \subfigure[$\mathcal{G}(T^{[p,q]}_{\infty})$]{\includegraphics[scale=.5]{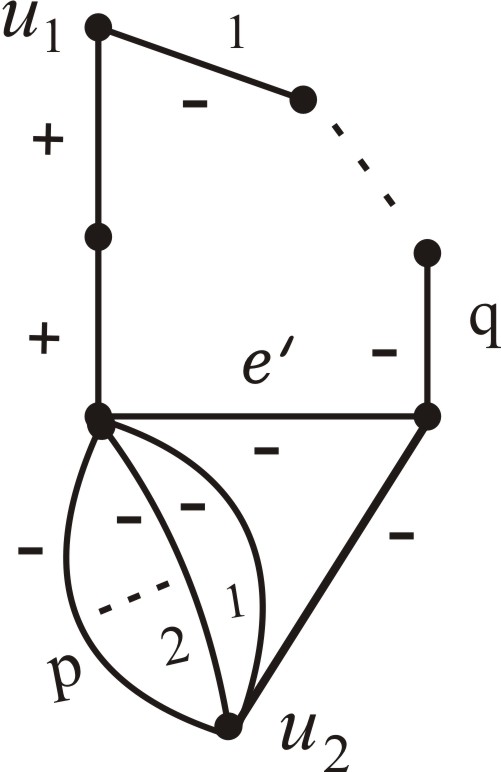} \label{Fig.gt^pq1} } 
 \hspace{.7cm}\subfigure[$\mathcal{G}(T^{[p,q]})$]{\includegraphics[scale=.5]{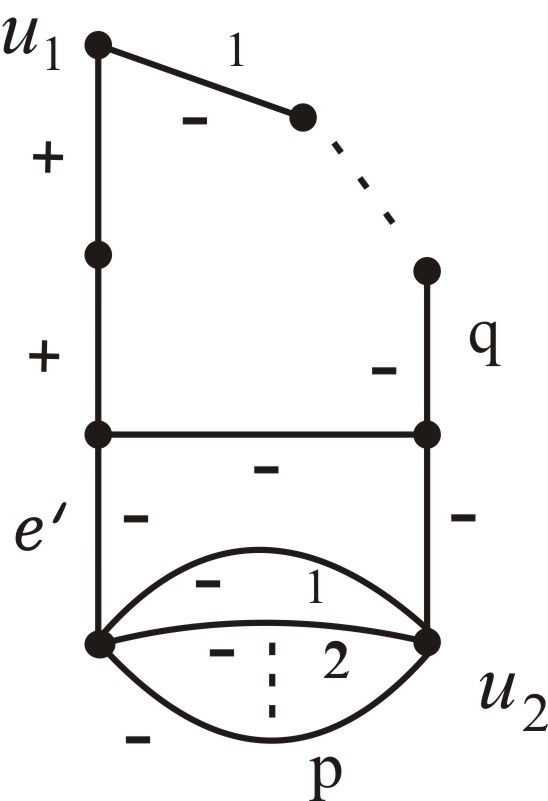}\label{Fig.gt^[pq]}}
\caption{Tait graphs $\mathcal{G}(T)$ corresponding to tangles $T\in \Omega$.}\label{tait:graph}}\end{figure}

\noindent Consider a set $\Omega$ of family of non-alternating tangles 
\[\Omega=\{T_{[q]},~T^p,~T_q, ~T^{p,q},~T_{p,q},~T^{[p,q]}_{\infty},~T^{[p,q]}: p,q\in \mathbb{N}\},\]
where $T_{[q]},~T^p,~T_q, ~T^{p,q},~T_{p,q},~T^{[p,q]}_{\infty},$ and $T^{[p,q]}$ are the tangles as shown in Fig.~\ref{ntangles}. \\
Note that the denominator closure of $T^{[1,1]}$  tangle results in the mirror image of the non-alternating link $L7n2$. If the tangle $T^{[1,1]}$ possessed an alternating tangle diagram, then this would contradict the fact that its denominator closure is non-alternating. Therefore, the tangle $T^{[1,1]}$ cannot have an alternating tangle diagram. It follows that $T^{[p,q]}$ is a non-alternating tangle, since it is derived from $T^{[1,1]}$ by twisting existing crossings. 
Moreover, the denominator closure of $T^{[1,2]}$ and $T^{[2,1]}$ tangles results in the non-alternating links,  mirror image of $L8n1$ and  knot $8_{21}$, respectively. Thus, by the same argument as for $T^{[1,1]}$, the  tangles $T^{[2,1]}$ and $T^{[1,2]}$ are also non-alternating.   
On the other hand, the tangle $T^{p}$ possesses an alternating diagram as shown in Fig.~\ref{Fig.tpalt}. With the exception of the $T^p$, all the tangles in $\Omega$  appear to lack alternating tangle diagrams. Nevertheless, we use these non-alternating tangle diagrams to prove our result in Theorem~\ref{thm:nalt}. 
For our convenience, in this section crossing $c'$ of tangle $T\in \Omega$ is the crossing depicted in Fig.~\ref{ntangles}. Additionally, we refer to a non-alternating tangle diagram simply as a non-alternating tangle.

\begin{lemma}\label{lem:nalt}
Let $L$ be a quasi-alternating link diagram with a quasi-alternating crossing $c$, and let $L'$ be a link  obtained from $L$ by replacing  crossing $c$ with a non-alternating tangle $T\in \Omega$.  Then the determinant property holds at crossing $c'$ of  $T\in \Omega$ in $L'$ under the following conditions:\\
\begin{enumerate}
\item If $T=T_{[q]}$, then the determinant property holds at  $c'$   if and only if  either $q\geq 3$  or  $q=2$ with $\det(L_0)<2\det(L_{\infty})$.\\
\item If $T=T^{p}$, then the determinant property holds at  $c'$ if and only if $\det(L_0)<\det(L_{\infty})$. \\

\item If $T=T_{q}$, then the determinant property holds at $c'$  if and only if  either $q\geq 2$  or  $q=1$ with $\det(L_0)<3\det(L_{\infty})$.\\

\item If $T=T^{p,q}$, then the determinant property holds at $c'$  if and only if either $ q\geq 2$,  or $p=q=1$, or $q=1$ with $\det(L_0)<\left(1+p/(p-1)\right)\det(L_{\infty})$.\\

\item If $T=T_{p,q}$, then the determinant property holds at  $c'$ if and only if  either  $q\geq 2$,  or $q=1$ with $\det(L_0)<2\det(L_{\infty})$.\\

\item If $T=T^{[p,q]}_{\infty}$, then determinant property holds at $c'$  if and only if  either $q\geq (2p+1)$  or  $q<(2p+1)$ with $\det(L_0)<\left(\dfrac{q(2p+1)+2}{(2p+1)-q}\right)\det(L_{\infty})$.\\

\item If $T=T^{[p,q]}$, then the determinant property holds at  $c'$ if and only if  either $q>1$,  or  $q=1$ with $\det(L_0)<(6+1/p)\det(L_{\infty})$.
 \end{enumerate}
\end{lemma}
\begin{proof}
Consider a quasi-alternating link diagram $L$ with a quasi-alternating crossing $c$ and let $e$ be the edge in $\mathcal{G}(L)$ corresponding to crossing $c$. Let $L'$ be the link obtained from $L$ by replacing $c$ with a tangle $T\in \Omega$.

\noindent Choose the checkerboard coloring of $L$ such that the edge $e(=u_1u_2)$ in $\mathcal{G}(L)$ is positive. Let $\mathcal{G}(L')$ be the Tait graph of $L'$ with the induced checkerboard coloring, and $\mathcal{G}(T)$ be the sub-graph $\mathcal{G}(L')\restr{T}$ of $\mathcal{G}(L')$.
Total number of positive edges in  $\mathcal{G}(T)$ are two, as depicted in Fig.~\ref{tait:graph}. Clearly, the total number of positive edges in both the spanning tree and almost spanning tree of  $\mathcal{G}(T)$ is either one or two.\\

\noindent For $i=1,2$, let $x_i$ be the number of spanning trees of $\mathcal{G}(T)$ with $i$ number of positive edges, and $y_i$ denote the number of almost spanning trees of $\mathcal{G}(T)$ with respect to $u_1$ and $u_2$ having $i$ number of positive edges. 
Consider the $c'$ of $L'$ that belongs to $T$, as depicted in Fig.~\ref{ntangles}. Let $e'$ be its corresponding edge in $\mathcal{G}(T)$. We define $x_{i_{e'}}$ and $y_{i_{e'}}$ as the number of spanning trees and almost spanning trees of $\mathcal{G}(T)$, respectively, that contain edge $e'$ and have $i$ number of positive edges. 
 
\noindent For each spanning tree $\mathcal{T}$ of $\mathcal{G}(L)$ such that $e \in \mathcal{T}$, there exist $\displaystyle \sum_{i=1,2}x_{i_{e'}}$ number of spanning trees containing edge $e'$ and $\displaystyle \sum_{i=1,2}(x_i-x_{i_{e'}})$ number of spanning tree without edge $e'$ in $\mathcal{G}(L')$. These satisfy the following conditions:
\begin{enumerate}
\item $u_1$ and $u_2$ are connected through sub-graph $\mathcal{G}(T)$ of $\mathcal{G}(L')$, more precisely, through spanning trees of $\mathcal{G}(T)$.
\item number of positive edges in these spanning tress of $\mathcal{G}(L')$ are  $\mathsf{v}(\mathcal{T})+(i-1)$, whenever the spanning tree of $\mathcal{G}(T)$  having $i$ number of positive edges. 
\end{enumerate} 
 Therefore, corresponding to all spanning tress of $\mathcal{G}(L)$ that contain edge $e$, there exist: 
 \begin{enumerate}
\item  $\displaystyle \sum_{i=1,2}x_{i_{e'}}s_{v-i}(L_0)$ number of spanning trees in $\mathcal{G}(L')$ that include edge $e'$ with $v$ number of positive edges, and \\

\item $\displaystyle \sum_{i=1,2}(x_i-x_{i_{e'}})s_{v-i}(L_0)$ number of  spanning trees in $\mathcal{G}(L')$ that do not include edge $e'$ while still having $v$ number of positive edges.
\end{enumerate}

\noindent For each spanning tree $\mathcal{T}$ of $\mathcal{G}(L)$ such that $e \notin \mathcal{T}$, there exist $\displaystyle \sum_{i=1,2}y_{i_{e'}}$ number of spanning trees containing edge $e'$ and $\displaystyle \sum_{i=1,2}(y_i-y_{i_{e'}})$ spanning trees does not containing edge $e'$ in $\mathcal{G}(L')$. These satisfy the following conditions: 
\begin{enumerate}
\item $u_1$ and $u_2$ are not connected through sub-graph $\mathcal{G}(T)$ of $\mathcal{G}(L')$, more precisely, through almost trees of $\mathcal{G}(T)$.
\item number of positive edges in these spanning tress of $\mathcal{G}(L')$ are  $\mathsf{v}(\mathcal{T})+i$,  whenever the spanning tree of $\mathcal{G}(T)$  having $i$ number of positive edges.

\end{enumerate}

 \noindent Therefore, corresponding to all spanning tress of $\mathcal{G}(L)$ that do not contain $e$, there exist: 
 \begin{enumerate}
\item  $\displaystyle \sum_{i=1,2}y_{i_{e'}}s_{v-i}(L_{\infty})$ number of spanning trees in $\mathcal{G}(L')$ that include edge $e'$ with $v$ number of positive edges, and \\

\item $\displaystyle \sum_{i=1,2}(y_i-y_{i_{e'}})s_{v-i}(L_\infty)$ number of  spanning trees in $\mathcal{G}(L')$ that do not include edge $e'$   while still having $v$ number of positive edges.
\end{enumerate}
\noindent For any integer $v$, we have
\vspace{-.2cm}
\begin{align*}
s_{v}(L')&=\text{number of spanning trees in~} \mathcal{G}(L')  \text{~ having~} v \text{~ number of positive edges}\\
&=x_1s_{v-1}(L_0)+x_2s_{v-2}(L_0)+y_1s_{v-1}(L_\infty)+y_2s_{v-2}(L_\infty)
\end{align*}
\vspace{-.8cm}
\begin{equation}\label{eqL'}
\begin{split} 
 \det(L')&=\Big| \displaystyle \sum_{v}(-1)^{v}s_{v}(L') \Big|\\
 &=\Big| \displaystyle \sum_{v}(-1)^{v} \left\{ x_1 s_{v-1}(L_0)+x_2s_{v-2}(L_0)+y_1s_{v-1}(L_\infty)+y_2s_{v-2}(L_\infty)\right \}\Big|\\
&=\Big| \displaystyle \sum_{v}(-1)^{v} \{x_1s_{v-1}(L_0)+x_2s_{v-2}(L_0)\}+\displaystyle \sum_{v}(-1)^{v}\{y_1s_{v-1}(L_\infty)+y_2s_{v-2}(L_\infty)\}\Big|\\
&=\Big|(x_1-x_2) \displaystyle \sum_{v'}(-1)^{v'-1}s_{v'}(L_0)+(-y_1+y_2)\displaystyle \sum_{v'}(-1)^{v'}s_{v'}(L_\infty)\Big|\\
&=\Big|(x_1-x_2)\mathfrak{a}+(-y_1+y_2)\mathfrak{b}\Big|,
\end{split}
\end{equation}
where $L'_0$ and $L'_\infty$ are the links obtained from $L'$ by performing smoothings at crossing $c'$. Since $L$ is quasi-alternating at $c$  and $e$ is a positive edge, by Remark~\ref{rem:ab}

\vspace{-.5cm}
\[\displaystyle \sum_{v}(-1)^{v} s_{v-1}(L_0). \displaystyle \sum_{v}(-1)^{v}s_{v}(L_\infty)>0, \quad \text{~that is}\quad \mathfrak{a}.\mathfrak{b}>0.\]
\vspace{-.5cm}

\noindent If the edge $e'$ in $\mathcal{G}(T)$ is negative, then the determinate of $L'_0$ and $L'_{\infty}$ are given  by the following expressions:
\vspace{-.7cm} 

\begin{align*}
s_{v}(L'_{0})&=\text{number of spanning trees in~} \mathcal{G}(L') \text{~ without egde~} e' \text{~ having~} v \text{~ number of positive edges}\\
&=(x_1-x_{1_{e'}})s_{v-1}(L_0)+(x_2-x_{2_{e'}})s_{v-2}(L_{0})+(y_1-y_{1_{e'}})s_{v-1}(L_\infty)+(y_2-y_{2_{e'}})s_{v-2}(L_\infty)
\end{align*}
\begin{equation}\label{eqL'1}
\begin{split}
 \det(L'_0)&=\Big| \displaystyle \sum_{v}(-1)^{v}s_{v}(L'_0)\Big|\\
&=\Big|\displaystyle \sum_{v}(-1)^{v} \{(x_1-x_{1_{e'}})s_{v-1}(L_0)+(x_2-x_{2_{e'}})s_{v-2}(L_{0})\\
&\hspace{.2cm}+(y_1-y_{1_{e'}})s_{v-1}(L_\infty)+(y_2-y_{2_{e'}})s_{v-2}(L_\infty)\} \Big| \\
&=\Big|\displaystyle \sum_{v}(-1)^{v}\{(x_1-x_{1_{e'}})s_{v-1}(L_0)+(x_2-x_{2_{e'}})s_{v-2}(L_{0})\}\\
&\hspace{.2cm}+\displaystyle \sum_{v}(-1)^{v}\{(y_1-y_{1_{e'}})s_{v-1}(L_\infty)+(y_2-y_{2_{e'}})s_{v-2}(L_\infty)\}\Big| \\
&=\Big|\{(x_1+x_{1_{e'}})-(x_2-x_{2_{e'}})\}\displaystyle \sum_{v'}(-1)^{v'-1}s_{v'}(L_0)+\{-(y_1-y_{1_{e'}})+(y_2-y_{2_{e'}})\} \displaystyle \sum_{v'}(-1)^{v'}s_{v'}(L_\infty)\Big| \\
&=\Big|\{(x_1-x_{1_{e'}})-(x_2-x_{2_{e'}})\}\mathfrak{a}+\{(y_2-y_{2_{e'}})-(y_1-y_{1_{e'}})\}\mathfrak{b}\Big| \\
\end{split}
\end{equation}

\begin{equation}\label{eqL'2}
\begin{split}
s_{v}(L'_{\infty})&= \text{number of spanning trees in~} \mathcal{G}(L') \text{~ with egde~} e' \text{~ having~} v \text{~ number of positive edges}\\
&=x_{1_{e'}}s_{v-1}(L_0)+x_{2_{e'}}s_{v-2}(L_{0})+y_{1_{e'}}s_{v-1}(L_\infty)+y_{2_{e'}}s_{v-2}(L_\infty)\\
\det(L'_{\infty})&=\Big| \displaystyle \sum_{v}(-1)^{v}s_{v}(L'_{\infty})\Big|\\
&=\Big|\displaystyle \sum_{v}(-1)^{v}\{ x_{1_{e'}}s_{v-1}(L_0)+x_{2_{e'}}s_{v-2}(L_{0})+y_{1_{e'}}s_{v-1}(L_\infty)+y_{2_{e'}}s_{v-2}(L_\infty)\} \Big| \\
&=\Big|\displaystyle \sum_{v}(-1)^{v}\{x_{1_{e'}}s_{v-1}(L_0)+x_{2_{e'}}s_{v-2}(L_{0})\}+\displaystyle \sum_{v}(-1)^{v}\{y_{1_{e'}}s_{v-1}(L_\infty)+y_{2_{e'}}s_{v-2}(L_\infty))\}\Big| \\
&=\Big|(x_{1_{e'}}-x_{2_{e'}})\displaystyle \sum_{v'}(-1)^{v'-1}s_{v'}(L_0)+(-y_{1_{e'}}+y_{2_{e'}}) \displaystyle \sum_{v'}(-1)^{v'}s_{v'}(L_\infty)\Big| \\
&=\Big|(x_{1_{e'}}-x_{2_{e'}})\mathfrak{a}+(y_{2_{e'}}-y_{1_{e'}})\mathfrak{b}\Big| \\
\end{split}
\end{equation}
\vspace{.1cm}

\noindent If the edge $e'$ in $\mathcal{G}(T)$ is positive, then the determinate of $L'_0$ and $L'_{\infty}$ are given  by the following expressions:
\begin{align*}
s_{v-1}(L'_{0})&=\text{number of spanning trees in~} \mathcal{G}(L') \text{~ with egde~} e' \text{~ having~} v \text{~ number of positive edges}\\
&=x_{1_{e'}}s_{v-1}(L_0)+x_{2_{e'}}s_{v-2}(L_{0})+y_{1_{e'}}s_{v-1}(L_\infty)+y_{2_{e'}}s_{v-2}(L_\infty)
\end{align*}

\vspace{-.5cm}
\begin{equation}\label{eqL'1+}
\begin{split}
 \det(L'_0)&=\Big| \displaystyle \sum_{v}(-1)^{v}s_{v}(L'_0)\Big|\\
&= \Big|\displaystyle \sum_{v}(-1)^{v}\{ x_{1_{e'}}s_{v}(L_0)+x_{2_{e'}}s_{v-1}(L_{0})+y_{1_{e'}}s_{v}(L_\infty)+y_{2_{e'}}s_{v-1}(L_\infty)\} \Big| \\
&=\Big|\displaystyle \sum_{v}(-1)^{v}\{x_{1_{e'}}s_{v}(L_0)+x_{2_{e'}}s_{v-1}(L_{0})\}+\displaystyle \sum_{v}(-1)^{v}\{y_{1_{e'}}s_{v}(L_\infty)+y_{2_{e'}}s_{v-1}(L_\infty)\}\Big| \\
&=\Big|(x_{1_{e'}}-x_{2_{e'}})\displaystyle \sum_{v'}(-1)^{v'}s_{v'}(L_0)+(y_{1_{e'}}-y_{2_{e'}}) \displaystyle \sum_{v'}(-1)^{v'}s_{v'}(L_\infty)\Big| \\
&=\Big|(x_{1_{e'}}-x_{2_{e'}})(-\mathfrak{a})+(y_{1_{e'}}-y_{2_{e'}})\mathfrak{b}\Big| =\Big|(x_{1_{e'}}-x_{2_{e'}})\mathfrak{a}+(y_{2_{e'}}-y_{1_{e'}})\mathfrak{b}\Big|\\
\end{split}
\end{equation}

\begin{align*}
s_{v}(L'_{\infty})&=\text{number of spanning trees in~} \mathcal{G}(L') \text{~ without egde~} e' \text{~ having~} v \text{~ number of positive edges}\\
&=(x_1-x_{1_{e'}})s_{v-1}(L_0)+(x_2-x_{2_{e'}})s_{v-2}(L_{0})+(y_1-y_{1_{e'}})s_{v-1}(L_\infty)+(y_2-y_{2_{e'}})s_{v-2}(L_\infty)
\end{align*}

\noindent Using Eq.~(\ref{eqL'1}), we have
\begin{equation}\label{eqL'2+}
\begin{split}
 \det(L'_{\infty})&=\Big| \displaystyle \sum_{v}(-1)^{v}s_{v}(L'_{\infty})\Big|\\
&=\Big|\{(x_1-x_{1_{e'}})-(x_2-x_{2_{e'}})\}\mathfrak{a}+\{(y_2-y_{2_{e'}})-(y_1-y_{1_{e'}})\}\mathfrak{b}\Big| \\
\end{split}
\end{equation}

\noindent {\underline{Case(1)}} Suppose that $T=T_{[q]}$, as shown in Fig.~\ref{Fig.t[q]}. Then the edge $e'$ in the Tait sub-graph $\mathcal{G}(T_{[q]})$ is negative, as depicted in Fig.~\ref{tait:graph}(a), and 
 \[x_1=2, ~x_2=q, ~x_{1_{e'}}=2, ~x_{2_{e'}}=q-1,\]
  \[y_1=2q, ~y_2=0, ~y_{1_{e'}}=2q-2,~ y_{2_{e'}}=0.\]
Using Eq.~(\ref{eqL'}), (\ref{eqL'1}), and (\ref{eqL'2}), we have
\[\det(L'_0)=|\mathfrak{a}+2\mathfrak{b}|, \quad \det(L'_{\infty})=|(q-3)\mathfrak{a}+2(q-1)\mathfrak{b}|, \text{~and}\]
\[\det(L')=|(q-2)\mathfrak{a}+2q\mathfrak{b}|=|(\mathfrak{a}+2\mathfrak{b})+\{(q-3)\mathfrak{a}+2(q-1)\mathfrak{b}\}|\]
It is evident that $\det(L')=\det(L'_0)+\det(L'_{\infty})$ holds if and only if 
\begin{equation}\label{eqt^{p}}
(\mathfrak{a}+2\mathfrak{b})\{(q-3)\mathfrak{a}+2(q-1)\mathfrak{b})\}>0.
\end{equation}
Clearly, Eq.~(\ref{eqt^{p}}) holds for all $q\geq 3$. 
If $q=2$, then Eq.~(\ref{eqt^{p}}) holds if and only if for $\mathfrak{a},\mathfrak{b}\gtrless0$, $(\mathfrak{a}+2\mathfrak{b})\gtrless 0$ and
\[(q-3)\mathfrak{a}+2(q-1)\mathfrak{b}=-\mathfrak{a}+2\mathfrak{b}\gtrless 0\quad \iff \quad 2\mathfrak{b} \gtrless \mathfrak{a} \quad \iff \quad |\mathfrak{a}|<2|\mathfrak{b}|.\]

\noindent {\underline{Case(2)}} Suppose that $T=T^{p}$, as shown in Fig.~\ref{Fig.t^p}. Then the edge $e'$ in the Tait sub-graph $\mathcal{G}(T^p)$ is positive, as depicted in Fig.~\ref{tait:graph}(b), and 
 \[x_1=2p,~x_2=p+1,~ x_{1_{e'}}=p,~ x_{2_{e'}}=p+1,\]
  \[y_1=2p,~y_2=1, ~y_{1_{e'}}=p, ~y_{2_{e'}}=1\]
Using Eq.~(\ref{eqL'}), (\ref{eqL'1+}), and (\ref{eqL'2+}), we have
\[\det(L'_0)=|\mathfrak{a}+(p-1)\mathfrak{b}|, \quad \det(L'_{\infty})=|p\mathfrak{a}-p\mathfrak{b}|, \text{~and}\]
\[\det(L')=|(p-1)\mathfrak{a}+(1-2p)\mathfrak{b}|=|(p\mathfrak{a}-p\mathfrak{b})+(-\mathfrak{a}-(p-1)\mathfrak{b})|.\]
It is evident that $\det(L')=\det(L'_0)+\det(L'_{\infty})$ holds if and only if 
\begin{equation}\label{eqt^p}
(p\mathfrak{a}-p\mathfrak{b})\{-\mathfrak{a}-(p-1)\mathfrak{b}\}>0.\end{equation}
For $a,b\gtrless0$, $(-\mathfrak{a}-(p-1)\mathfrak{b})\lessgtr 0$ and Eq.~(\ref{eqt^p}) holds if and only if  
\[p\mathfrak{a}-p\mathfrak{b}\lessgtr 0\quad \iff \quad \mathfrak{a} \lessgtr \mathfrak{b} \iff  |\mathfrak{a}|<|\mathfrak{b}|.\]

\noindent {\underline{Case(3)} } Suppose that $T=T_{q}$, as shown in Fig.~\ref{Fig.t_q}. Then the edge $e'$ in the Tait sub-graph $\mathcal{G}(T_q)$ is negative, as depicted in Fig.~\ref{tait:graph}(c), and 
 \[x_1=2, ~x_2=q+1, ~ x_{1_{e'}}=2, ~ x_{2_{e'}}=q\]
   \[y_1=2q+4, ~y_2=q+1, ~y_{1_{e'}}=2q+2, ~y_{2_{e'}}=q\]
Using Eq.~(\ref{eqL'}), (\ref{eqL'1}), and (\ref{eqL'2}), we have
\[\det(L'_0)=|a+b|, \quad \det(L'_{\infty})=|(q-2)a+(q+2)b|, \text{~and}\]
\[\det(L')=|(q-1)\mathfrak{a}+(q+3)\mathfrak{b}|=|(\mathfrak{a}+\mathfrak{b})+\{(q-2)\mathfrak{a}+(q+2)\mathfrak{b}\}|\] 
It is evident that $\det(L')=\det(L'_0)+\det(L'_{\infty})$ holds if and only if 
\begin{equation}\label{eqt_q}
(\mathfrak{a}+\mathfrak{b})\{(q-2)\mathfrak{a}+(q+2)\mathfrak{b})\}>0.\end{equation}
Clearly, Eq.~(\ref{eqt_q}) holds for all $q\geq 2$. For $q=1$, Eq.~(\ref{eqt_q}) becomes
\[(\mathfrak{a}+\mathfrak{b})(-\mathfrak{a}+3\mathfrak{b})> 0,\]
and it holds if and only if 
\[-\mathfrak{a}+3\mathfrak{b}\gtrless 0,\text{~whenever~} \mathfrak{a}\mathfrak{b}\gtrless 0 ,\text{~that is~} \iff  |\mathfrak{a}|<3|\mathfrak{b}|.\]

\noindent {\underline{Case(4)}} Suppose that $T=T^{p,q}$, as shown in Fig.~\ref{Fig.t^pq}. Then the edge $e'$ in the Tait sub-graph $\mathcal{G}(T^{p,q})$ is negative, as depicted in Fig.~\ref{tait:graph}(e), and 
 \[x_1=2p, ~x_2=pq+p+1, ~x_{1_{e'}}=2p, ~x_{2_{e'}}=pq+1\]
   \[y_1=2pq+2p, ~y_2=q+1, ~ y_{1_{e'}}=2pq, ~y_{2_{e'}}=q\]
Using Eq.~(\ref{eqL'}), (\ref{eqL'1}), and (\ref{eqL'2}), we have
\[\det(L'_0)=|p\mathfrak{a}+(2p-1)\mathfrak{b}|, \quad \det(L'_{\infty})=|(-2p+pq+1)\mathfrak{a}+(2pq-q)\mathfrak{b}|, \text{~and}\]
\begin{align*}
\det(L')&=|(pq-p+1)\mathfrak{a}+(2pq+2p-q-1)\mathfrak{b}|\\
&=|\{p\mathfrak{a}+(2p-1)\mathfrak{b}\}+\{(-2p+pq+1)\mathfrak{a}+(2pq-q)\mathfrak{b}\}|
\end{align*} 
It is evident that $\det(L')=\det(L'_0)+\det(L'_{\infty})$ holds if and only if 
\begin{equation}\label{eqt^pq}
\{p\mathfrak{a}+(2p-1)\mathfrak{b}\}\{(-2p+pq+1)\mathfrak{a}+(2pq-q)\mathfrak{b}\}>0.
\end{equation}
Clearly, Eq.~(\ref{eqt^pq}) holds for all $p\geq 1$, $q\geq 2$ and $p=q=1$. In case $q=1$ and $p>1$, Eq.~(\ref{eqt^pq}) becomes
\begin{equation}\label{eqt^pq:n}
(p\mathfrak{a}+(2p-1)\mathfrak{b})\{(1-p)\mathfrak{a}+(2p-1)\mathfrak{b}\}\gtrless 0.\end{equation}
 Moreover, for $\mathfrak{a},\mathfrak{b}\gtrless0$, $\mathfrak{a}+(2p-1)\mathfrak{b}\gtrless 0$ and  Eq.~(\ref{eqt^pq:n}) holds if and only if
 \[(1-p)\mathfrak{a}+(2p-1)\mathfrak{b}\gtrless 0,\quad \text{that is}\quad |\mathfrak{a}|<\left(1+\dfrac{p}{p-1}\right)|\mathfrak{b}|.\]

\noindent {\underline{Case(5)}} Suppose that $T=T_{p,q}$, as shown in Fig.~\ref{Fig.t_pq}. Then the edge $e'$ in the Tait sub-graph $\mathcal{G}(T_{p,q})$ is negative, as depicted in Fig.~\ref{tait:graph}(d), and 
 \[x_1=2p+2, ~x_2=pq+q, ~ x_{1_{e'}}=2, ~ x_{2_{e'}}=q\]
   \[y_1=2pq+2q+2, ~y_2=q, ~ y_{1_{e'}}=2q, ~ y_{2_{e'}}=0\]
Using Eq.~(\ref{eqL'}), (\ref{eqL'1}), and (\ref{eqL'2}), we have
\[\det(L'_0)=|(pq-2p)\mathfrak{a}+(2pq-q+2)\mathfrak{b}|, \quad \det(L'_{\infty})=|(q-2)\mathfrak{a}+2q\mathfrak{b}|, \text{~and}\]
\begin{align*}
\det(L')&=|(-2p+q+pq-2)\mathfrak{a}+(2pq+q+2)\mathfrak{b}|\\
&=|\{(q-2)\mathfrak{a}+2q\mathfrak{b}\}+\{(pq-2p)\mathfrak{a}+(2pq-q+2)\mathfrak{b}\}|
\end{align*} 
It is evident that $\det(L')=\det(L'_0)+\det(L'_{\infty})$ holds if and only if 
\begin{equation}\label{eqt_pq}
\{(q-2)\mathfrak{a}+2q\mathfrak{b}\}\{(pq-2p)\mathfrak{a}+(2pq-q+2)\mathfrak{b}\}>0.
\end{equation}
Clearly, Eq.~(\ref{eqt_pq}) holds for all $p\geq 1$ and $q\geq 2$. In case, when $q=1$ and $p\geq 1$, Eq.~(\ref{eqt_pq})
 becomes
\begin{equation}\label{eqt_pq:n}
(-\mathfrak{a}+2\mathfrak{b})\{-p\mathfrak{a}+(2p+1)\mathfrak{b}\}>0.
\end{equation} 
It is evident that for $\mathfrak{a},\mathfrak{b}\gtrless0$, Eq.~(\ref{eqt_pq:n}) holds if and only if both
\[(-\mathfrak{a}+2\mathfrak{b}) \gtrless0 \quad \text{~and~} \quad -p\mathfrak{a}+(2p+1)\mathfrak{b}\gtrless0.\] 
\[\text{~That is, ~} \quad \mathfrak{b} \gtrless -\mathfrak{a} \quad  \text{~and~}\quad \dfrac{(2p+1)}{p}\mathfrak{b}\gtrless  \mathfrak{a}.\] 
It is easy to observe that $2\mathfrak{b} \gtrless \mathfrak{a}$ implies $\dfrac{(2p+1)}{p}\mathfrak{b}\gtrless  \mathfrak{a}.$ Hence, Eq.~(\ref{eqt_pq:n}) holds for any $p\geq 1$ if and only if $|a|<2|b|.$ \\

\noindent {\underline{Case(6)}} Suppose that $T=T^{[p,q]}_{\infty}$, as shown in Fig.~\ref{Fig.t^pq1}. Then the edge $e'$ in the Tait sub-graph $\mathcal{G}(T^{[p,q]}_{\infty})$ is negative, as depicted in Fig.~\ref{tait:graph}(f), and 
 \[x_1=4p+2,~x_2=pq+p+q+2,~ x_{1_{e'}}=2p+2,~ x_{2_{e'}}=q+1\]
   \[y_1=4pq+2p+2q+2,~y_2=2q+1, ~y_{1_{e'}}=2pq+2q+2,~ y_{2_{e'}}=q\]
Using Eq.~(\ref{eqL'}), (\ref{eqL'1}), and (\ref{eqL'2}), we have
\[\det(L'_0)=|(pq-p+1)\mathfrak{a}+(2pq+2p-q-1)\mathfrak{b}| =|(p(q-1)+1)\mathfrak{a}+(2p-1)(q+1)\mathfrak{b}|,\]
\[ \det(L'_{\infty})=|(-2p+q-1)\mathfrak{a}+(2pq+q+2)\mathfrak{b}|=|(q-(2p+1))\mathfrak{a}+(2pq+q+2)\mathfrak{b}|, \text{~and}\]
\begin{align*}
\det(L')&=|(pq-3p+q)\mathfrak{a}+(4pq+2p+1)\mathfrak{b}|\\
&=|\{(pq-p+1)\mathfrak{a}+(2pq+2p-q-1)\mathfrak{b}\}+\{(-2p+q-1)\mathfrak{a}+(2pq+q+2)\mathfrak{b}\}|\\
&=|\{(p(q-1)+1)\mathfrak{a}+(2p-1)(q+1)\mathfrak{b}\}+\{(q-(2p+1))\mathfrak{a}+(2pq+q+2)\mathfrak{b}\}|
\end{align*} 
It is evident that $\det(L')=\det(L'_0)+\det(L'_{\infty})$ holds if and only if 
\begin{equation}\label{eqt'}
\{(p(q-1)+1)\mathfrak{a}+(2p-1)(q+1)\mathfrak{b}\}\{(q-(2p+1))\mathfrak{a}+(2pq+q+2)\mathfrak{b}\}>0.
\end{equation}
Clearly, Eq.~(\ref{eqt'}) holds for all $p\geq 1$ and $q\geq (2p+1)$. 
For $q<(2p+1)$ and $p\geq1$,
\[(p(q-1)+1)\mathfrak{a}+(2p-1)(q+1)\mathfrak{b}\gtrless 0, \text{~whenever~} a,b\gtrless0.\]
Therefore, Eq.~(\ref{eqt'}) holds in this case if and only if for $\mathfrak{a}\mathfrak{b}\gtrless0$,
\[(q-(2p+1)\mathfrak{a}+(2pq+q+2)\mathfrak{b}\gtrless 0,\quad \text{that is} \quad |\mathfrak{a}|<\left(\dfrac{q(2p+1)+2}{(2p+1)-q}\right)|\mathfrak{b}|.\]

\noindent {\underline{Case(7)}} Suppose that $T=T^{[p,q]}$ and $p, q\geq 1$ be positive integers, as shown in Fig.~\ref{Fig.t^[pq]}. Then the edge $e'$ in the Tait sub-graph $\mathcal{G}(T^{[p,q]})$ is negative, as depicted in Fig.~\ref{tait:graph}(g), and 
 \[x_1=6p+2, ~x_2=3pq+2p+q+1, x_{1_{e'}}=4p+2,~ x_{2_{e'}}=2pq+p+q+1\]
   \[y_1=6pq+4p+2q+2,~y_2=2pq+p+2q+1,~ y_{1_{e'}}=4pq+2p+2q+2, ~y_{2_{e'}}=2q+1\]
Using Eq.~(\ref{eqL'}), (\ref{eqL'1}), and (\ref{eqL'2}), we have
\[\det(L'_0)=|(p(q-1)+1)\mathfrak{a}+p\mathfrak{b}|, \quad \det(L'_{\infty})=|(2pq-3p+q-1)\mathfrak{a}+(4pq+2p+1)\mathfrak{b}|, \text{~and}\]
\begin{align*}
\det(L')&=|(3pq-4p+q-1)\mathfrak{a}+(4pq+3p+1)\mathfrak{b}|\\
&=|\{(pq-p)\mathfrak{a}+p\mathfrak{b}\}+\{(2pq-3p+q-1)\mathfrak{a}+(4pq+2p+1)\mathfrak{b}\}|
\end{align*} 
It is evident that $\det(L')=\det(L'_0)+\det(L'_{\infty})$ holds if and only if 
\begin{equation}\label{eqtf}
\{(pq-p)\mathfrak{a}+p\mathfrak{b}\}\{(2pq-3p+q-1)\mathfrak{a}+(4pq+2p+1)\mathfrak{b}\}>0.
\end{equation}
Clearly, Eq.~(\ref{eqtf}) holds for all $p\geq 1$ and $q>1$. For $p\geq 1$ and $q=1$, Eq.~(\ref{eqtf}) becomes
\[p\mathfrak{b}\{(-p)\mathfrak{a}+(6p+1)\}>0,\]
and it holds if and only if, for $\mathfrak{a}\mathfrak{b}\gtrless 0$
\[(-p)\mathfrak{a}+(6p+1)\mathfrak{b} \gtrless 0 \iff \dfrac{(6p+1)}{p}\mathfrak{b} \gtrless  \mathfrak{a} \iff |\mathfrak{a}|<\dfrac{(6p+1)}{p}|\mathfrak{b}|.\]

\noindent Hence the proof of the lemma.
\end{proof}
\begin{figure}[!ht] {\centering
 \subfigure[$T^{[p,q]}_{0}$]{\includegraphics[scale=.4]{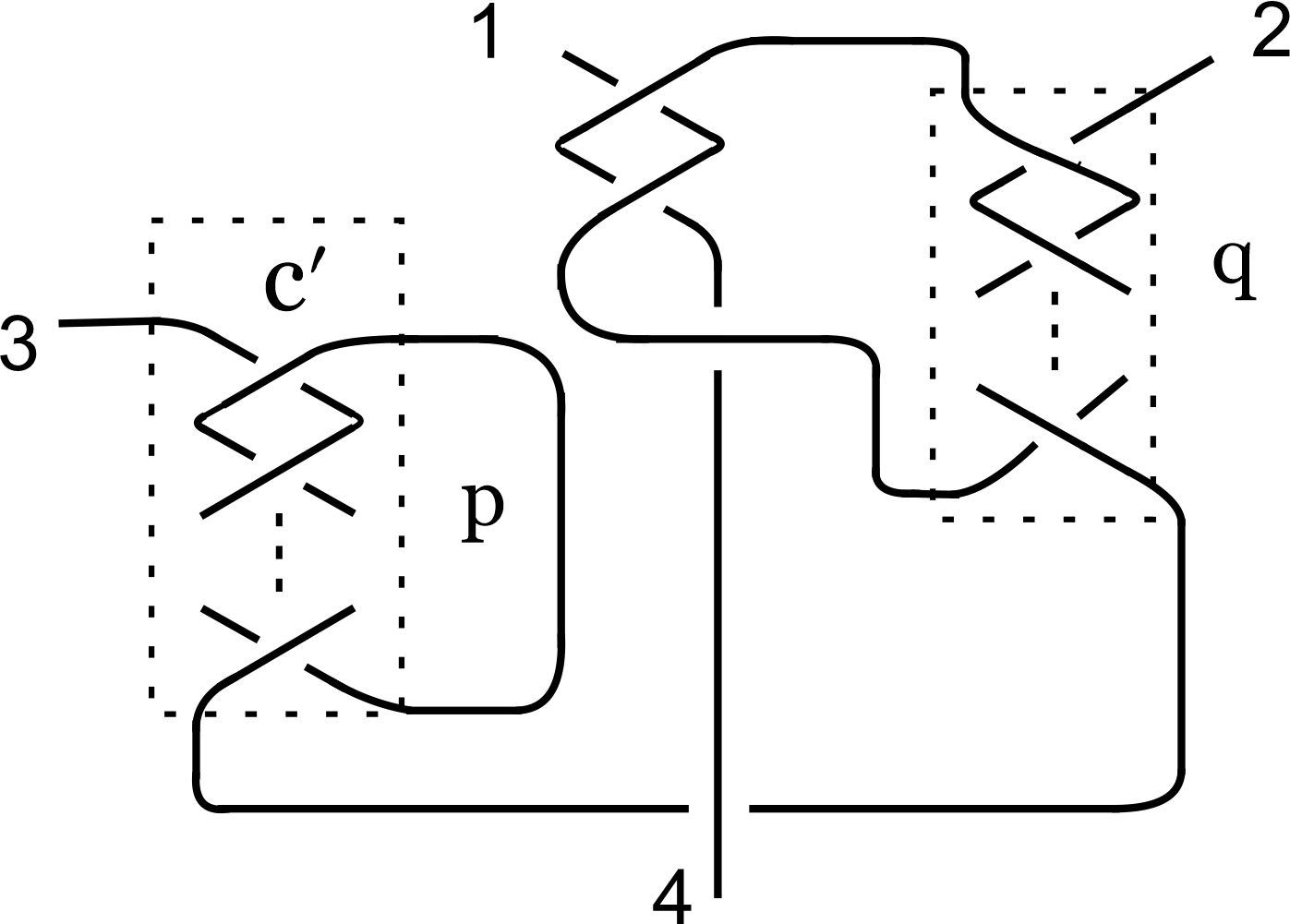} \label{Fig.t^pq1} }\hspace{1.5cm}\subfigure[$T'$]{\includegraphics[scale=.58]{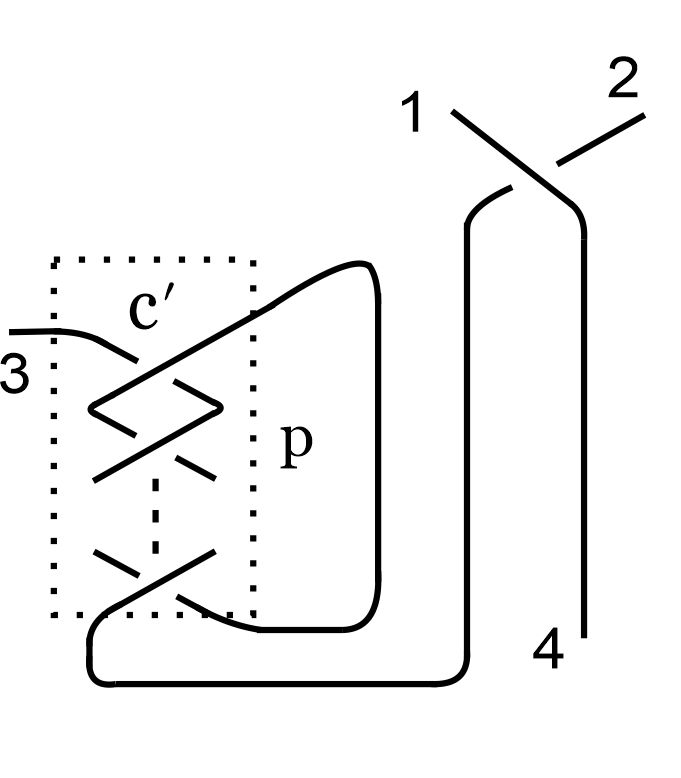}\label{Fig.t'}}
\caption{ Tangles with crossing $c'$.}\label{ntangles2}}\end{figure}

\begin{proposition}\label{prop:mth1}
Let $L'$ be a link obtained from a quasi-alternating link diagram $L$ by replacing a quasi-alternating crossing $c$ with tangle $T_{0}^{[p,q]}$, as shown in Fig.~\ref{ntangles2}(a). Then $L'$ is a quasi-alternating link.
\end{proposition}
\begin{proof} Consider a quasi-alternating link diagram $L$ with a quasi-alternating crossing $c$. Let $L'$ be a link  obtained from  $L$ by replacing crossing $c$ with tangle $ T_{0}^{[p,q]}$, and $c'$ be a crossings of $T$ in $L'$, as illustrated in Fig.~\ref{ntangles2}(a). \\
Then $L'$ be equivalent to a link diagram that represents a connected sum of $L^{-(q-1)}$ and $(2,p)$-torus link, as depicted in Fig.~\ref{L'_0}. It is evident that $\det(L')=p\det(L^{-(q-1)})$.
For $p=1$, $L'$ is equivalent to the link $L^{-(q-1)}$, which is quasi-alternating by Theorem~\ref{Thm:pre:m}.
\noindent Now, consider the case where $p=2$. By performing smoothings at crossing $c'$ in $L'$, it is evident that both $L'_{0}$ and $L'_{\infty}$ are equivalent to $L^{-(q-1)}$. Since $L^{-(q-1)}$ is quasi-alternating by Theorem~\ref{Thm:pre:m}, both $L'_{0}$ and $L'_{\infty}$  are quasi-alternating. Moreover, in this case
  \[\det(L')=2\det(L^{-(q-1)})=\det(L^{-(q-1)})+\det(L^{-(q-1)})=\det(L'_{0})+\det(L'{\infty}).\]
  Hence, $c'$ is a quasi-alternating crossing.  By replacing $c'$ with $p-1$ vertical half twists of same type, the resulting link remains quasi-alternating by Theorem~\ref{Thm:pre:m}. Thus, the desired result is achieved.
\end{proof}

 \begin{figure}[!ht] {\centering
 \includegraphics[scale=.55]{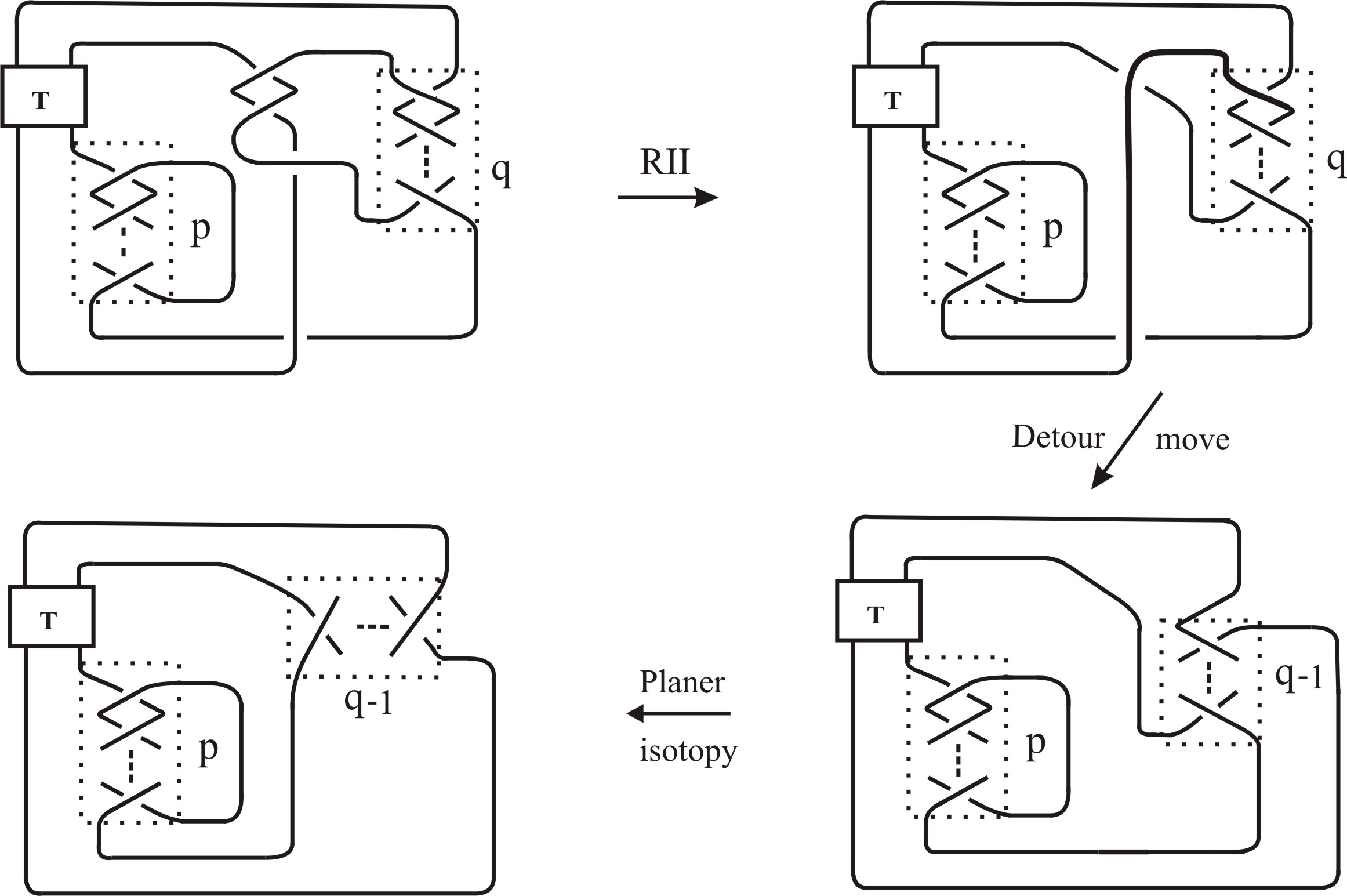} 
 \caption{ }\label{L'_0}}
 \end{figure}

\begin{proposition}\label{prop:mth2}
Let $L'$ be a link obtained from a quasi-alternating link diagram $L$ by replacing a quasi-alternating crossing $c$ with tangle  $T'$, as shown in Fig.~\ref{ntangles2}(b). Then, whenever crossing $c$ satisfies property (I), $L'$ is a quasi-alternating link.
\end{proposition}
\begin{proof} Consider a quasi-alternating link diagram $L$ with a quasi-alternating crossing $c$ that satisfies property (I). Let $L'$ be a link  obtained from  $L$ by replacing crossing $c$ with tangle $T'$, and $c'$ be the crossings of $T$ in $L'$, as illustrated in Fig.~\ref{ntangles2}(b). 
Then $L'$ be equivalent to a link diagram that represents a connected sum of $L^{\overline{1}}$ and $(2,p)$-torus link.
It is evident that $\det(L')=p\det(L^{\overline{1}})$. 

\noindent For $p=1$, $L'$ is equivalent to the link $L^{\overline{1}}$, which is quasi-alternating since property (I) holds at crossing $c$ in $L$.
 Now, consider the case where $p=2$. Then by performing smoothings at crossing $c'$ in $L'$, we find  that both $L'_{0}$ and $L'_{\infty}$ are equivalent to $L^{\overline{1}}$. Since property (I) holds at $c$, $L^{\overline{1}}$ is quasi-alternating. Hence, both $L'_{0}$ and $L'_{\infty}$  are quasi-alternating. Moreover, in this case
  \[\det(L')=2\det(L^{'})=\det(L^{'})+\det(L^{'})=\det(L'_{0})+\det(L'_{\infty}).\]
  Hence, $c'$ is a quasi-alternating crossing.  By replacing $c'$ with $p-1$ vertical half twists of same type, the resulting link remains quasi-alternating by Theorem~\ref{Thm:pre:m}. Thus, the desired result is achieved.
\end{proof}

\begin{theorem}\label{thm:nalt} 
Let $L$ be a quasi-alternating link diagram with a quasi-alternating crossing $c$ that satisfies property (II). Then a link $L'$ obtained from $L$ by replacing crossing $c$ with a non-alternating tangle $T\in\Omega$ is quasi-alternating.  
\end{theorem}
\begin{proof}
Consider a quasi-alternating link diagram $L$ with a quasi-alternating crossing $c$ that satisfies property (II). Construct a link $L'$ from $L$ by replacing crossing $c$ to a non-alternating tangle $T\in \Omega$. Let $c'$ be a crossing of $T\in \Omega$  in $L'$ as illustrated in Fig.~\ref{ntangles}. Denote $L'_0$ and $L'_{\infty}$ be the links obtained from $L'$ by performing  smoothings at crossing $c'$, as shown in Fig.~\ref{1}. Our aim is to prove the crossing $c'$ to  be a quasi-alternating crossing. Since property (II) holds at crossing $c$ in $L$,  property (I) satisfied at $c$ and $\det(L_0)<\det(L_{\infty})$. By applying Lemma~\ref{lem:nalt}, it is evident that the  determinant property holds at crossing $c'$ of $L'$ in all cases, except $q=1$ in case $T_{[q]}$.    Remaining proof of the theorem will be divided into following parts.\\
\begin{figure}[!ht] {\centering
 \subfigure[Tangle $T_{[1]}$]{\fbox{ \parbox{.4\textwidth}{\includegraphics[scale=.5]{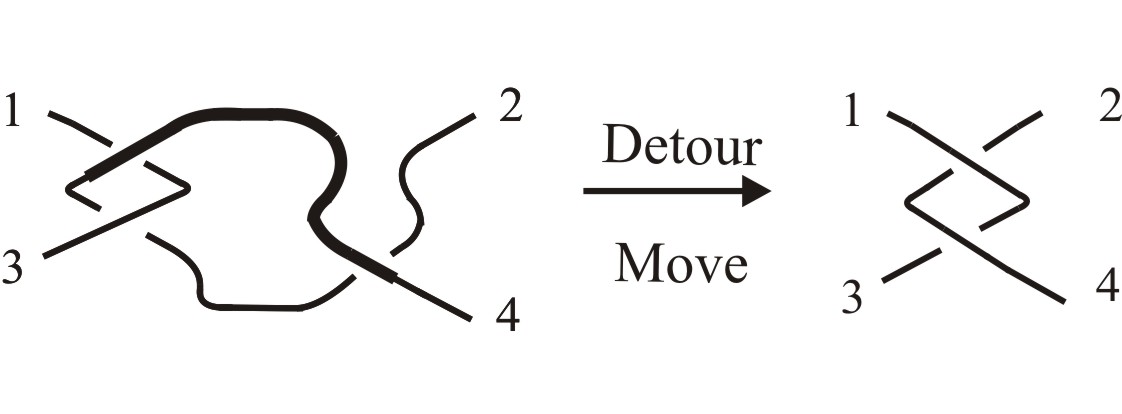} }}}\hspace{.2cm}
 \subfigure[Tangle $T_{0}$]{\fbox{ \parbox{.45\textwidth}{ \includegraphics[width=.42\textwidth, height=.14\textwidth]{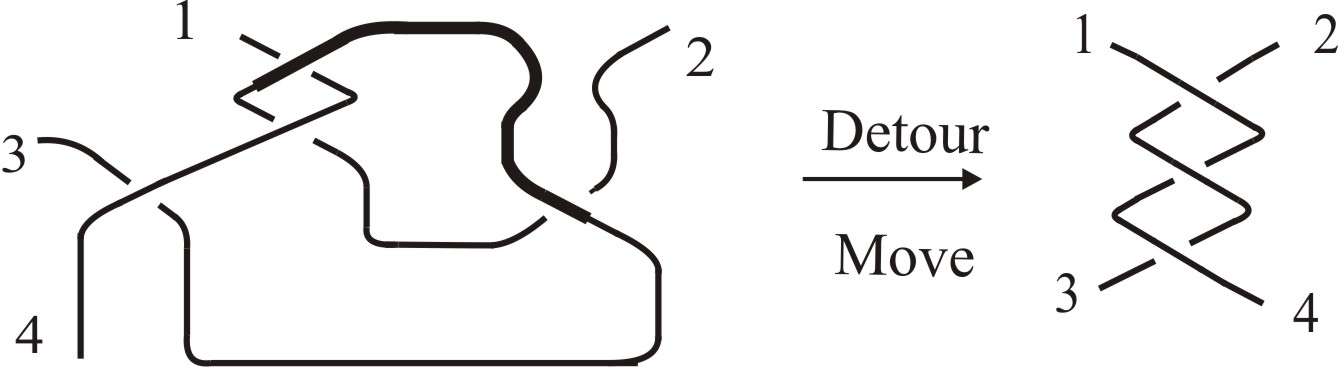} }}}
 \caption{}\label{mth0}}
\end{figure}

\noindent (1) Suppose that $T=T_{[q]}$, as illustrated in Fig.~\ref{Fig.t[q]}. For $q=1$, $L'$ is equivalent to a link that is obtained from $L$ by replacing crossing $c$ with two vertical half twists of opposite type, as shown in Fig.~\ref{mth0}(a).  Since Property (I) holds at crossing $c$ in $L$ and $\det(L_0)<\det(L_{\infty})$, the link $L'$ is quasi-alternating in this case by Lemma~\ref{lemma:opptwist}.\\
Now, consider the case when $q=2$. Then the links $L'_{0}$ and $L'_{\infty}$ are equivalent to the links $L^{2}$ and $L^{\overline{2}}$, respectively, see Fig.~\ref{mth1}. Again by using the Property (I) of the crossing $c$ and $\det(L_0)<\det(L_{\infty})$,  $L^{\overline{2}}$ is quasi-alternating  by Lemma~\ref{lemma:opptwist}. Moreover, $L^{2}$ is quasi-alternating by Theorem~\ref{Thm:pre:m}. Since the determinant property holds at $c'$ in $L'$,   $c'$ is a quasi-alternating crossing of $L'$.

\noindent Further, replacing crossing $c'$ by $q-1$  vertical half twists of same type,  the resulting link diagram is quasi-alternating at every new crossing by Theorem~\ref{Thm:pre:m}.  Hence, $c'$ is a quasi-alternating crossing of $L'$ for any arbitrary $q$ and $L'$ is a quasi-alternating link. 

\begin{figure}[!ht] {\centering
 \includegraphics[scale=.7]{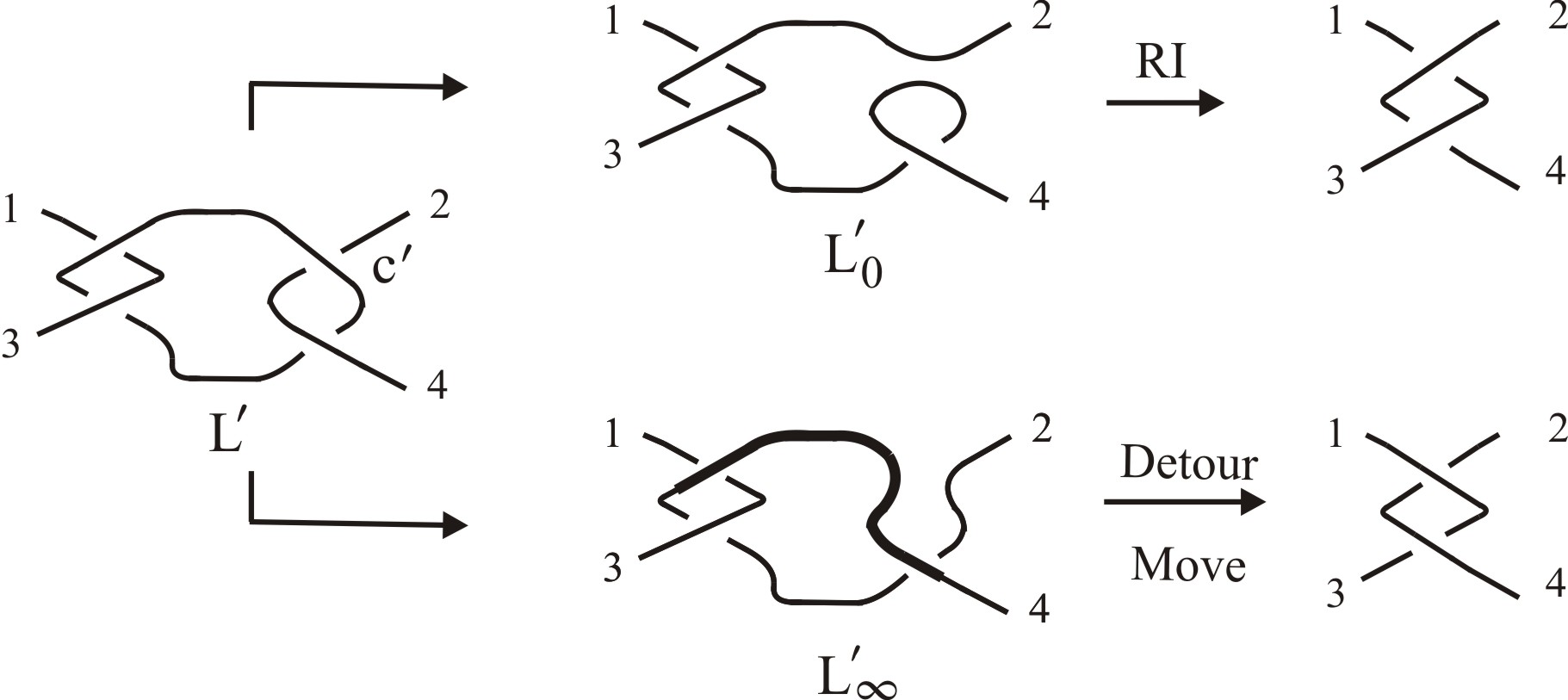} 
 \caption{Links $L'$, $L'_{0}$ and $L'_{\infty}$, when $T=T_{[2]}$. }\label{mth1}}
\end{figure}

\noindent (2) Suppose that $T=T_{q}$, as illustrated in Fig.~\ref{Fig.t_q}. Take $q=1$. Then the links $L'_{0}$ and $L'_{\infty}$ are equivalent to the links $L^{3}$ and $L^{\overline{3}}$, respectively, see Fig.~\ref{mth2}. Since Property (I) holds at crossing $c$ in $L$ and $\det(L_0)<\det(L_{\infty})$, $L^{\overline{3}}$ is quasi-alternating by Lemma~\ref{lemma:opptwist}. Moreover, $L^{3}$ is quasi-alternating by Theorem~\ref{Thm:pre:m}. Since the determinant property holds at $c'$, $c'$ is a quasi-alternating crossing of $L'$.

\noindent If crossing $c'$ is replaced by $q$  vertical half twists of same type, then the resulting link diagram is quasi-alternating at every new crossing by Theorem~\ref{Thm:pre:m}.  Hence, $c'$ is a quasi-alternating crossing of $L'$ for any $q$ and $L'$ is a quasi-alternating link. \\

\begin{figure}[!ht] {\centering
 \includegraphics[scale=.63]{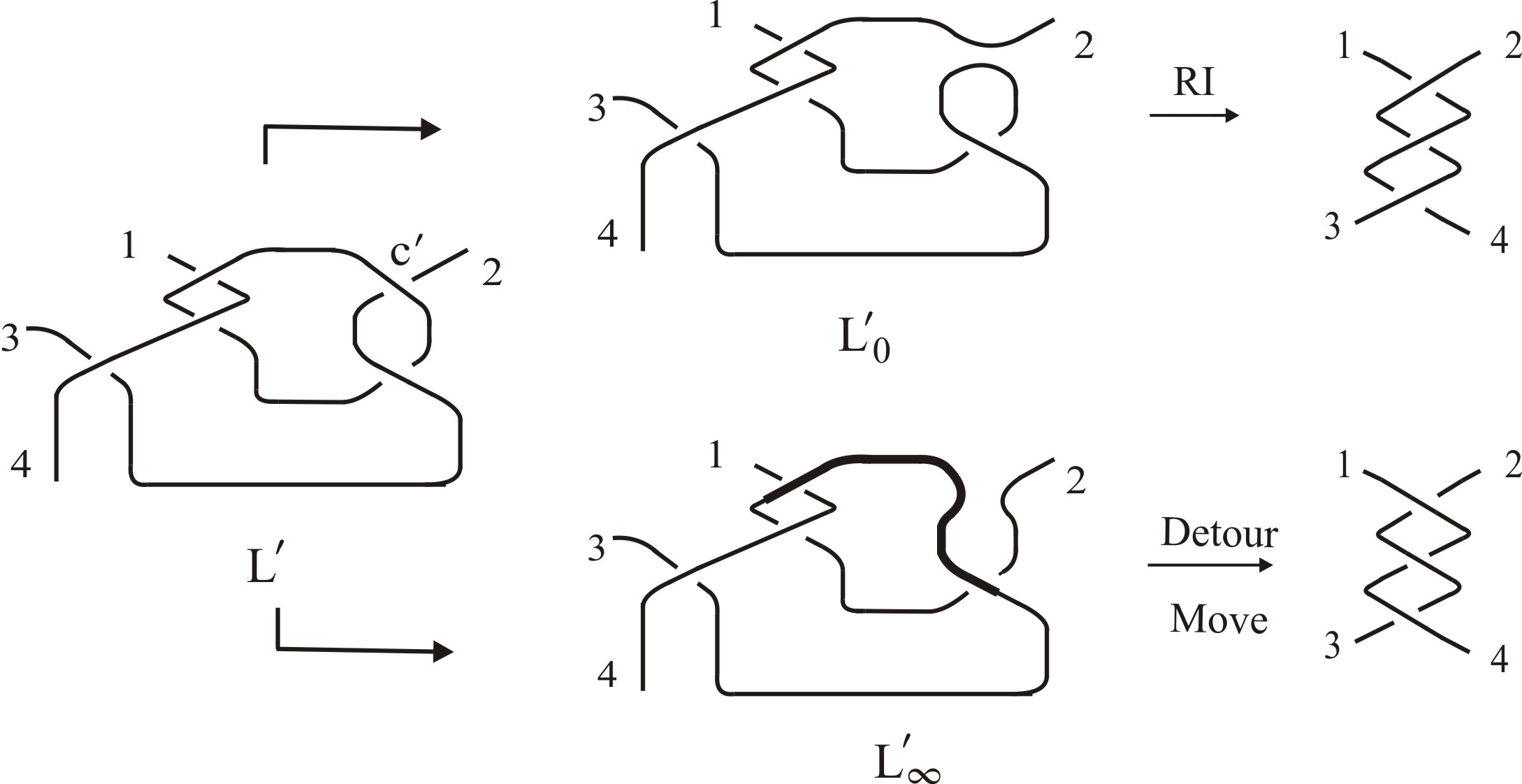} 
 \caption{Links $L'$, $L'_{0}$ and $L'_{\infty}$, when $T=T_{2}$. }\label{mth2}}
\end{figure}

\noindent (3) Suppose that $T=T^{[p]}$, as illustrated in Fig.~\ref{Fig.t^p}.   In this case, the link $L'_{0}$ is equivalent to the link $L^{p-1}$, which is quasi-alternating by Theorem~\ref{Thm:pre:m}. The link $L'_{\infty}$ is equivalent to the link which is connected sum of  $L^{\overline{1}}$ and $(2,p)$-torus link as shown in Fig.~\ref{mth3}. By Proposition~\ref{prop:mth2}, $L'_{\infty}$ is a quasi-alternating link. Since the determinant property holds at $c'$,  $c'$ is a quasi-alternating crossing of $L'$ for any $q$. Hence, $L'$ is a quasi-alternating link for any arbitrary $p$. \\

\begin{figure}[!ht] {\centering
 \includegraphics[scale=.6]{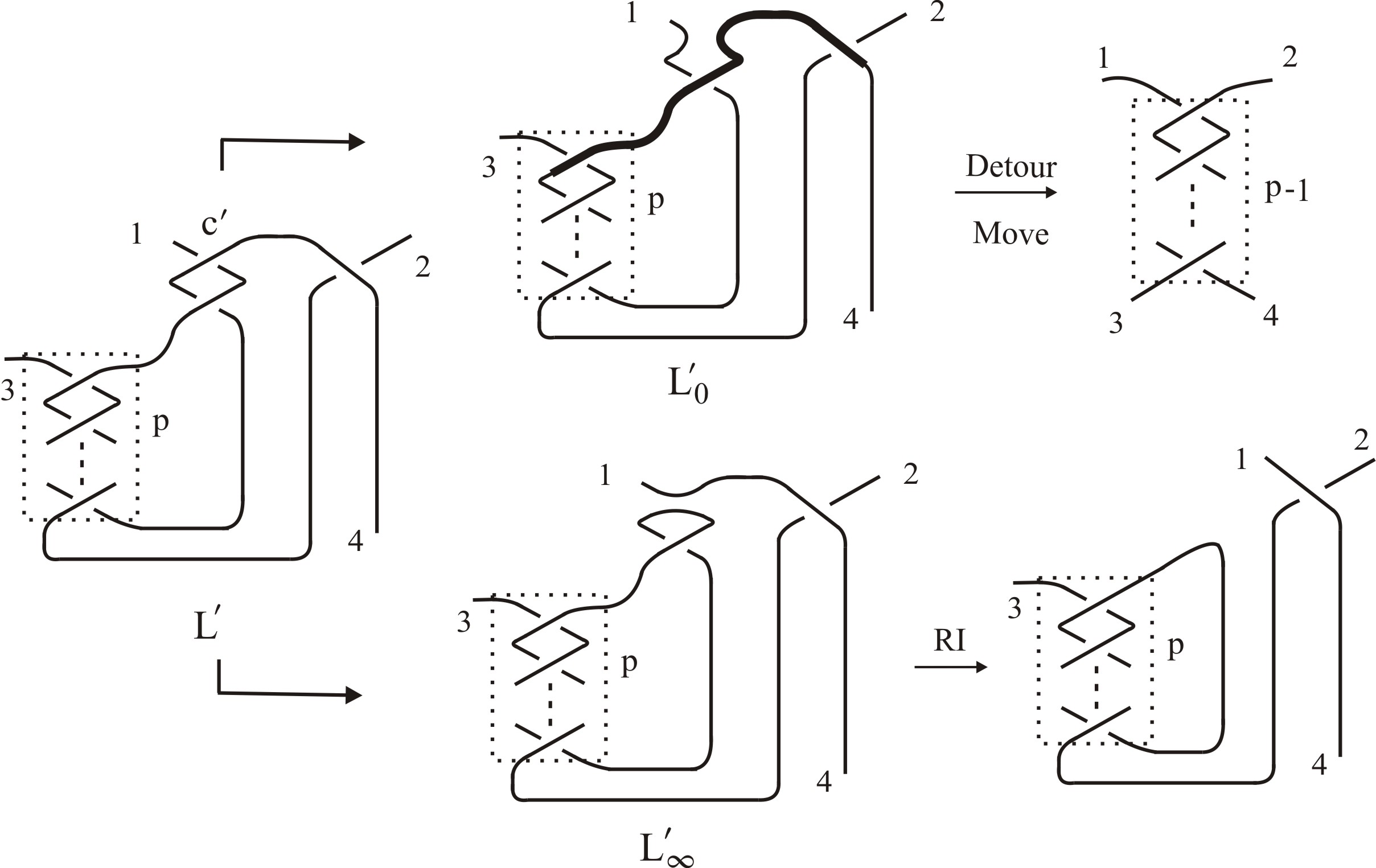} 
 \caption{Links $L'$, $L'_{0}$ and $L'_{\infty}$, when $T=T^{[p]}$. }\label{mth3}}
\end{figure}

 \noindent (4) Suppose that $T=T^{p,q}$, as illustrated in Fig.~\ref{Fig.t^pq}. Take $q=1$. Then the link $L'_{0}$ is equivalent to a link  which is obtained from $L$ by replacing $c$ with an alternating tangle of same type, see Fig.~\ref{mth4}. Hence, $L'_{0}$ is quasi-alternating by Theorem~\ref{Thm:pre:m}. Moreover, the link $L'_{\infty}$ is the link  obtained from  $L$ by replacing $c$ with tangle $T^{[p]}$, hence  $L'_{\infty}$ is quasi-alternating by Case(3) of Theorem~\ref{thm:nalt}. Since the determinant property holds at $c'$, $c'$ is a quasi-alternating crossing of $L'$.
 
 \noindent Further, if crossing $c'$ is replaced by $q$  vertical half twists of same type, then the resulting link diagram is quasi-alternating at every new crossing by Theorem~\ref{Thm:pre:m}.  Hence, $c'$ is a quasi-alternating crossing of $L'$ and $L'$ is a quasi-alternating link for any $p$ and $q$.\\
   
 \begin{figure}[!ht] {\centering
 \includegraphics[scale=.6]{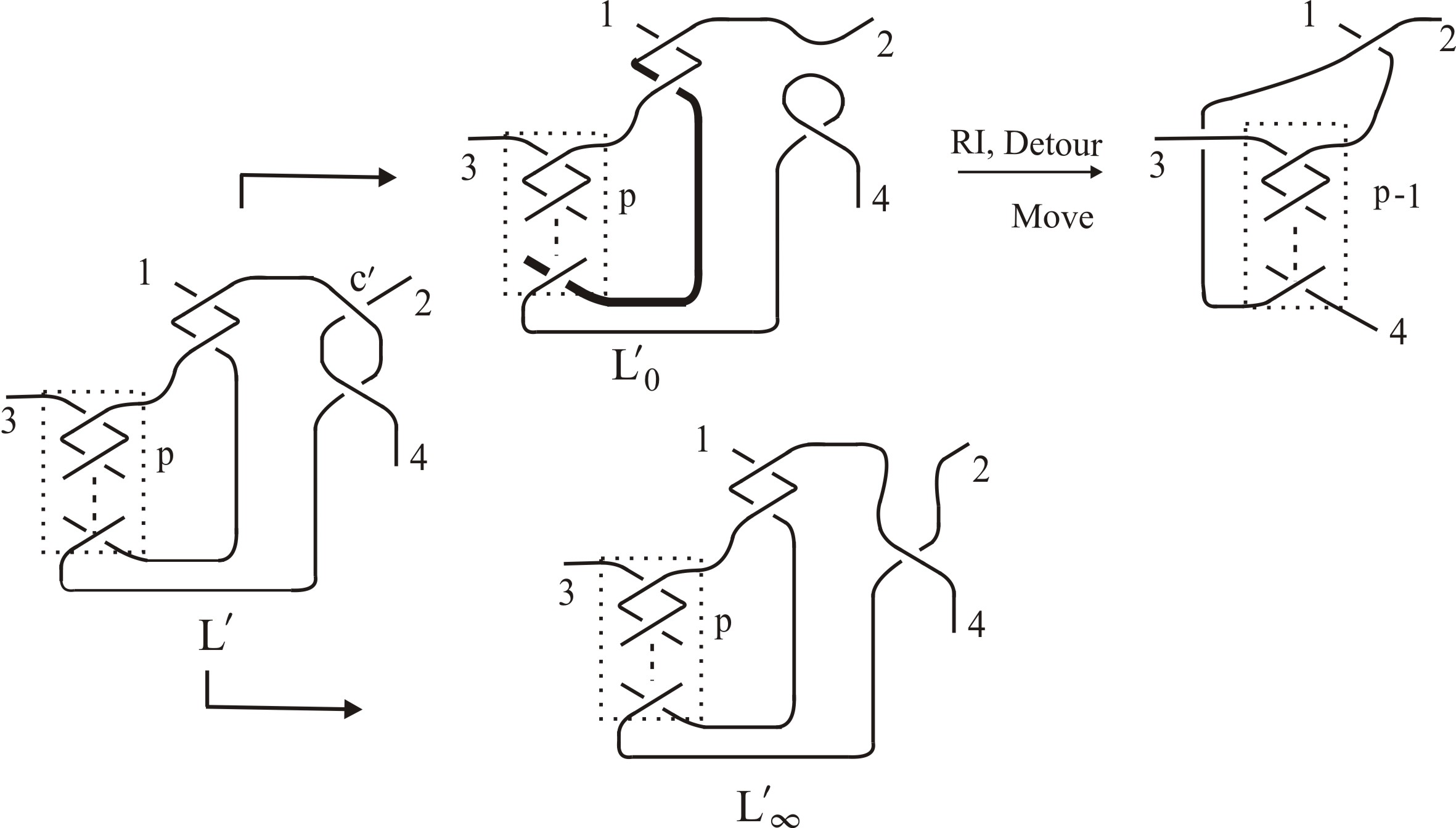} 
 \caption{Links $L'$, $L'_{0}$ and $L'_{\infty}$, when $T=T^{[p,2]}$. }\label{mth4} }
\end{figure} 

\noindent (5) Suppose that $T=T_{p,q}$, as illustrated in Fig.~\ref{Fig.t_pq}.  Take $p=1$. It is evident that the link $L'_{\infty}$ is equivalent to  the link that is obtained from $L$ by replacing crossing $c$ with tangle $T_{[q]}$, and link $L'_{0}$ is the link obtained from $L$ by replacing crossing $c$ with tangle $T_{q-1}$ For $q=1$, $T_{q-1}=T_{0}$, which is equivalent to a tangle of three vertical half twists of opposite type with respect to $c$, see Fig.~\ref{mth0}(b). Since property (I) holds at $c$ and $\det(L_0)<\det(L_{\infty})$, by Lemma~\ref{lemma:opptwist} $L'_{0}$ is quasi-alternating in this case. Otherwise, both   
 $L'_{0}$  and $L'_{\infty}$ are quasi-alternating by case (1) and (2) of Theorem~\ref{thm:nalt}. Since the determinant property holds at $c'$, $c'$ is a quasi-alternating crossing for $p=1$. 
 
\noindent Further, if crossing $c'$ is replaced by $p$ vertical half twists of same type, then the resulting link diagram is quasi-alternating at every new crossing by Theorem~\ref{Thm:pre:m}.  Hence, $c'$ is a quasi-alternating crossing of $L'$ and $L'$ is a quasi-alternating link for any $p$ and $q$.\\

\noindent (6) Suppose that $T=T^{[p,q]}_{\infty}$,  as illustrated in Fig.~\ref{Fig.t^{pq}}. It is evident that $L'_0$ and $L'_{\infty}$ are the links obtained from $L$ by replacing crossing $c$ with tangles $T_{p,q}$ and $T^{p,q}$, respectively. By part (4) and (5) of Theorem~\ref{thm:nalt}, both $L'_0$ and $L'_{\infty}$ are quasi-alternating. Since the determinant condition holds at crossing $c'$, $c'$ is a quasi-alternating crossing of $L'$ and hence $L'$ is a quasi-alternating link.\\

\noindent (7) Suppose that $T=T^{[p,q]}$, as illustrated in Fig.~\ref{Fig.t^[pq]}.  It is evident that $L'_{\infty}$ and $L'_0$ are the links obtained from $L$ by replacing crossing $c$ with tangles $T^{[p,q]}_{\infty}$ and  $T^{[p,q]}_{0}$, respectively. Moreover, $L'_{\infty}$ is quasi-alternating by part(6) of Theorem~\ref{thm:nalt}, and  $L'_0$ is quasi-alternating by Proposition~\ref{prop:mth1}.
Since the determinant condition holds at crossing $c'$, $c'$ is a quasi-alternating crossing of $L'$ and hence $L'$ is a quasi-alternating link. Hence the proof of the theorem.
\end{proof}
\begin{example}
A quasi-alternating crossing $c$ of a link as shown in Fig.~\ref{Exam:2Last}(a) satisfies property (II). By replacing crossing $c$ with tangle $T^{[2,1]}$, the resulting knot, $10_{151}$, depicted in Fig.~\ref{Exam:2Last}(b) is quasi-alternating.
\end{example}
\begin{figure}[!ht] {\centering
 \subfigure[$(2,4)$-torus link] {\includegraphics[scale=.35]{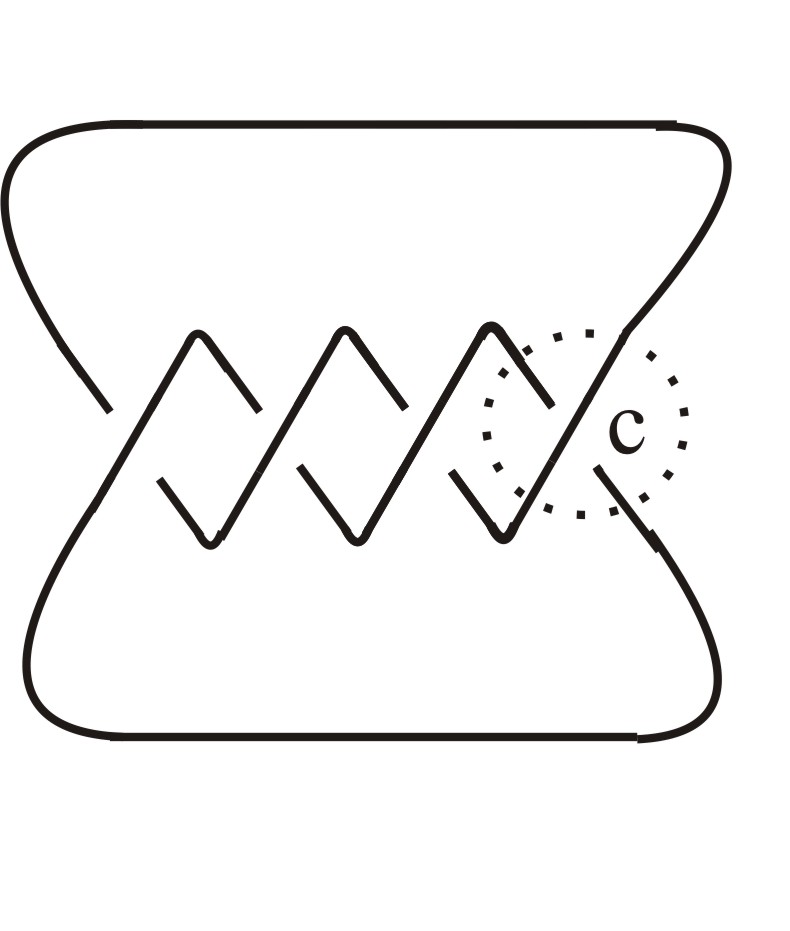} \hspace{.2cm}} \hspace{1.5cm}
\subfigure[$10_{151}$ knot] {\includegraphics[scale=.35]{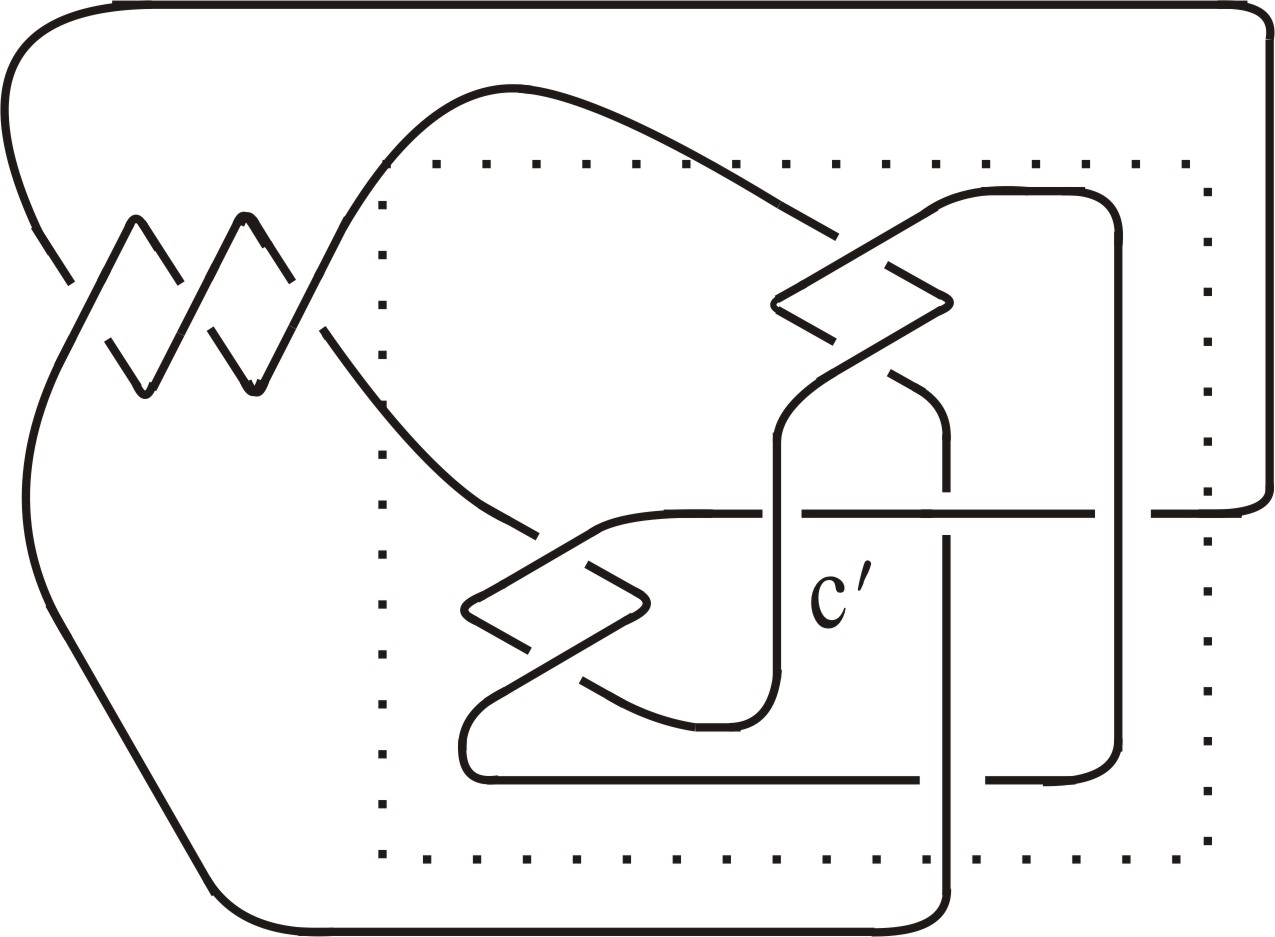}} 
\caption{A quasi-alternating $10_{151}$ knot obtained from $(2,4)$-torus link  by replacing quasi-alternating crossing $c$ satisfying property (II) to a non-alternating tangle $T^{[2,1]}$.}\label{Exam:2Last}}
\vspace{-.6cm}
\end{figure}

\begin{example}
A quasi-alternating crossing $c$ of a link, as shown in Fig.~\ref{exm:In1}(a), does not satisfy property (II). In this case, by replacing crossing $c$ with tangle $T^{[2,1]}$, the resulting link depicted in Fig.~\ref{exm:In1}(b) is not quasi-alternating.
\end{example}
\noindent The following remark is one of the significant implications of our results.
\begin{remark}\label{rem:tp}
Consider a quasi-alternating link $L$ with a quasi-alternating crossing $c$ satisfying property (II) and an alternating tangle $T$ of same type. The tangle $T^p$, shown in Fig.~\ref{Fig.t^p}, is equivalent to an alternating tangle, see Fig.~\ref{Fig.tpalt}. Theorem~\ref{thm:nalt} demonstrates  that the crossing $c'$ of $T^p$ is a quasi-alternating in  link $L'$, where $L'$ is the link obtained from $L$ by replacing crossing $c$ with tangle $T^p$. Clearly, the tangle $T$ is of  same type with respect to $c'$.  
Next, using the non-alternating tangle $T^p$ depicted in Fig.~\ref{Fig.t^p}, we can construct an alternating tangle of opposite type with respect to $c$, for which our construction remains valid. This construction involves the replacement of crossing $c'$ with an alternating tangle, as shown in Fig.~\ref{Fig.tpopp}. However, it is not possible to construct this alternating tangle from the alternating diagram of $T^p$ depicted in Fig.~\ref{Fig.tpalt}, by replacing any of its crossing with a tangle of $n+1$ crossings. 
\end{remark}
 \begin{figure}[!ht] {\centering
 \includegraphics[scale=0.5]{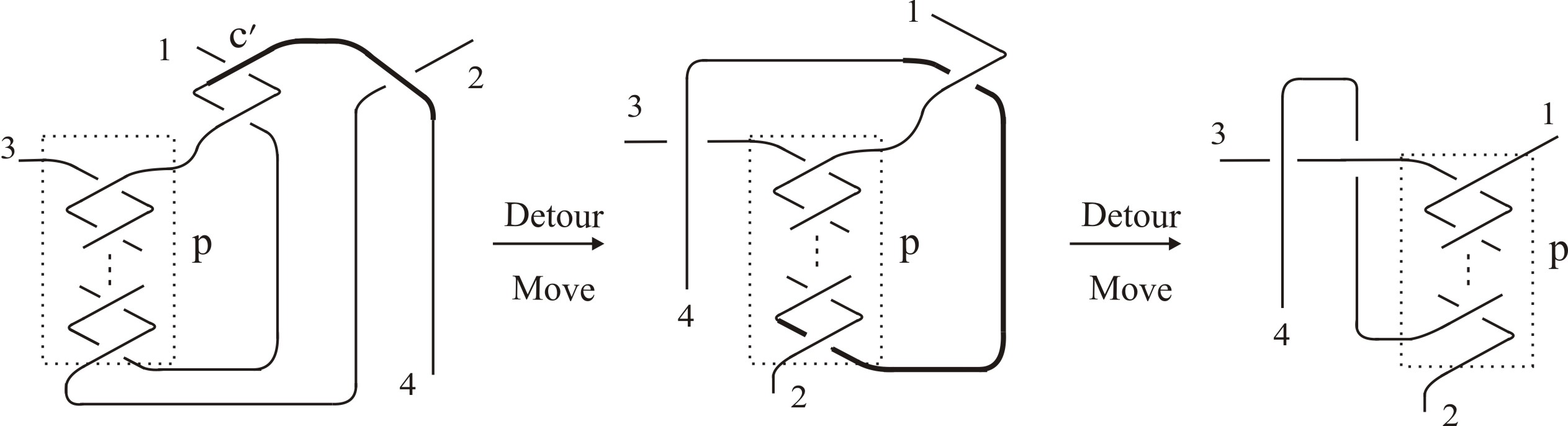} 
\caption{ $T^p$ is equivalent to an alternating tangle.}\label{Fig.tpalt}}\end{figure}

 \begin{figure}[!ht] {\centering
 
\includegraphics[scale=0.5]{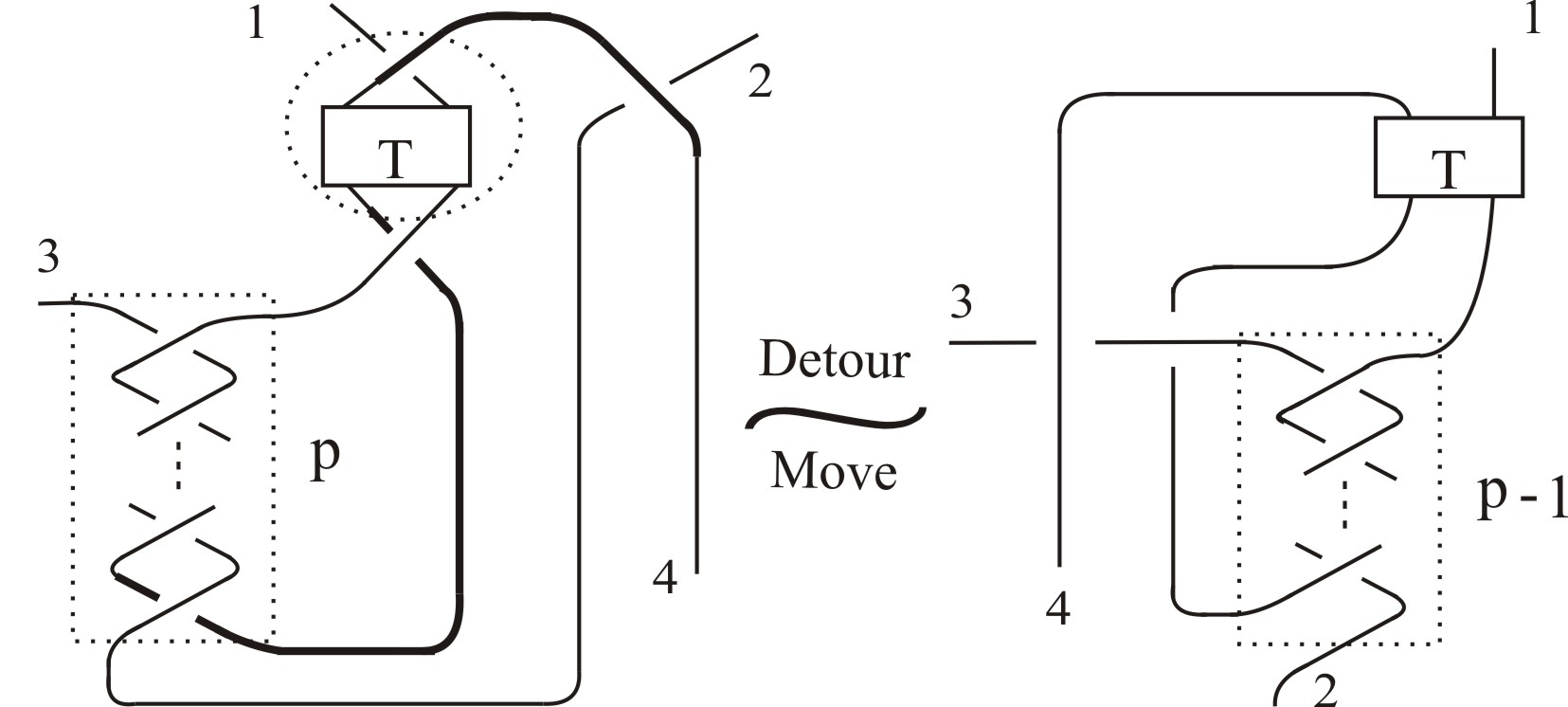}
\caption{ An alternating tangle of opposite type obtained from $T^p$.}\label{Fig.tpopp}}\end{figure}

\begin{theorem}\label{Thm:last} 
Let $L$ be a quasi-alternating link diagram with a quasi-alternating crossing $c$ and let $T\in \Omega$ be a non-alternating tangle. Then a link $L'$ obtained from $L$ by replacing crossing $c$ with a non-alternating tangle $\chi_T$ or $\chi_{\overline{T}}$ is quasi-alternating.  
\end{theorem}
\begin{proof}
Consider a quasi-alternating link diagram $L$ with a quasi-alternating crossing $c$ and let $T\in \Omega$ be any tangle. Let $c_1$ and $c_2$ be two horizontal half twisted crossings of $L^{-2}$ as illustrated Fig.~\ref{exm:last}. It is evident from the proof of Theorem~\ref{Thm1} that both $c_1$ and $c_2$ are  quasi-alternating crossings of $L^{-2}$ satisfying property (II). Replacing crossing $c_i$ in $L^{-2}$ with tangle $T$, the resulting link is a quasi-alternating link by Theorem~\ref{thm:nalt}. Hence the desired result.
\end{proof}

\begin{figure}[!ht] {\centering
 \subfigure[crossing $c$]{\hspace{.8cm}\includegraphics[scale=0.55]{Fig/Lthm0.jpg} \hspace{.7cm}} 
\subfigure[$L^{-2}$]{\includegraphics[scale=0.47]{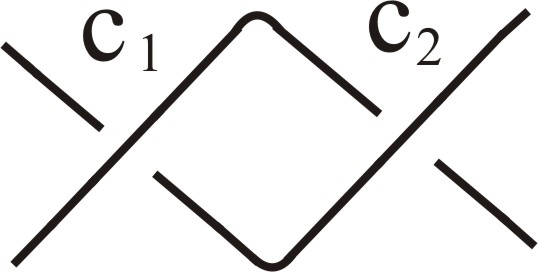}} \hspace{.5cm}
 \subfigure[$\chi_{T}$]{\includegraphics[scale=0.5]{Fig/Lthm1.jpg}} \hspace{.5cm}
\subfigure[$\chi_{\overline{T}}$]{\includegraphics[scale=0.5]{Fig/Lthm2.jpg}}
\caption{A crossing $c$ and non-alternating tangles with $T\in \Omega$.}\label{exm:last}}
\end{figure}


\section{Conclusion}

\noindent In this paper, we investigate the conditions necessary for constructing new quasi-alternating links by replacing a quasi-alternating crossing with alternating tangles of the opposite type. We also identify several families of non-alternating tangles that meet certain criteria for this construction to be valid. The examples illustrated in Fig.~\ref{exm:8_1 & 11n}(a), (b) and  \ref{Exam:2Last}(b)  support our proposed construction under the identified conditions. 
 
\noindent However, it is important to note that the conditions we established are sufficient but not necessary for the newly constructed link to be quasi-alternating. For instance, the crossing labeled $c$ in Fig.~\ref{exm:43} does not satisfy property (III), yet the knot obtained by replacing crossing $c$ with three horizontal half-twists is the quasi-alternating knot $11_{n88}$. 
 Similarly, replacing crossing $c$ in  Fig.~\ref{exm:In1}(a) with five vertical twists results in a quasi-alternating link, even though property (II) is not satisfied at $c$.

\noindent The results of this paper contribute to the construction and classification of new quasi-alternating knots and links with higher crossing number. Additionally, there may be other conditions and specific families of alternating tangles of the opposite type, as well as non-alternating tangles, for which this construction is applicable. 

\section{Acknowledgement}
\noindent The authors would like to thank Prof. Madeti Prabhakar and Prof. Nafaa Chibli for their remarks and suggestions during  the preparation of this paper. First author was supported by  National Board of Higher Mathematics (NBHM)(Ref. No. 0204/16(11)/2022/R $\&$D-II/11983), Government of India.

\end{document}